\documentclass[12pt]{article}
\usepackage{fullpage}
\usepackage{amsmath,amssymb,amsthm}
\usepackage[english]{babel}
\usepackage[left=3cm,right=3cm,top=3cm,bottom=4cm]{geometry}

\usepackage{tikz,tikz-cd}
\usetikzlibrary{positioning}
\usepackage{braket,mathtools,stmaryrd}

\SetSymbolFont{stmry}{bold}{U}{stmry}{m}{n}
\usepackage{lmodern,bm,bbm,mathrsfs,manfnt}
\usepackage{enumitem}

\usepackage[backref=page, colorlinks=true, linkcolor=blue, citecolor=red, menucolor=green,hyperindex,breaklinks]{hyperref}

\usepackage{color}

\theoremstyle{plain}
\newtheorem{theorem}{Theorem}[section]
\newtheorem{proposition}[theorem]{Proposition}
\newtheorem{lemma}[theorem]{Lemma}
\newtheorem{corollary}[theorem]{Corollary}

\theoremstyle{definition}
\newtheorem{definition}[theorem]{Definition}
\newtheorem{problem}{Problem}
\newtheorem{exercise}[theorem]{Exercise}

\newtheorem{example}[theorem]{Example}
\newtheorem{remark}[theorem]{Remark}
\newtheorem*{remark*}{Remark}

\newcommand{\lie}[1]{\mathfrak{#1}}
\newcommand{\bA}{\mathbb{A}}
\newcommand{\bC}{\mathbb{C}}
\newcommand{\bG}{\mathbb{G}}
\newcommand{\bP}{\mathbb{P}}
\newcommand{\bQ}{\mathbb{Q}}
\newcommand{\bR}{\mathbb{R}}
\newcommand{\bZ}{\mathbb{Z}}
\newcommand{\cA}{\mathcal{A}}
\newcommand{\cB}{\mathcal{B}}
\newcommand{\cD}{\mathcal{D}}
\newcommand{\cE}{\mathcal{E}}
\newcommand{\cK}{\mathcal{K}}
\newcommand{\cL}{\mathcal{L}}
\newcommand{\cM}{\mathcal{M}}
\newcommand{\cO}{\mathcal{O}}
\newcommand{\di}{\partial}
\newcommand{\et}{\text{\'et}}
\DeclareMathOperator{\ad}{ad}
\DeclareMathOperator{\Aut}{Aut}
\DeclareMathOperator{\Bun}{Bun}
\DeclareMathOperator{\diag}{diag}
\DeclareMathOperator{\End}{End}
\DeclareMathOperator{\ev}{ev}
\DeclareMathOperator{\Ext}{Ext}
\DeclareMathOperator{\Gal}{Gal}
\DeclareMathOperator{\Gr}{Gr}
\DeclareMathOperator{\gr}{gr}
\DeclareMathOperator{\Hom}{Hom}
\DeclareMathOperator{\id}{id}
\DeclareMathOperator{\im}{Im}
\DeclareMathOperator{\Jac}{Jac}
\DeclareMathOperator{\GL}{GL}
\DeclareMathOperator{\Map}{Map}
\DeclareMathOperator{\Mat}{Mat}
\DeclareMathOperator{\PGL}{PGL}
\DeclareMathOperator{\Pic}{Pic}
\DeclareMathOperator{\rank}{rank}
\DeclareMathOperator{\Res}{Res}
\DeclareMathOperator{\SL}{SL}
\DeclareMathOperator{\Sp}{Sp}
\DeclareMathOperator{\Spec}{Spec}
\DeclareMathOperator{\tr}{tr}
\DeclareMathOperator{\Val}{Val}
\DeclareMathOperator{\Vect}{Vect}
\DeclarePairedDelimiter{\inner}{\langle}{\rangle}
\DeclarePairedDelimiterX{\NO}[1]{{\vcentcolon}}{{\vcentcolon}\,}{{#1}}

\title{Hitchin systems and their quantization}
\author{Pavel Etingof and Henry Liu}

\begin{document}

\maketitle

\tableofcontents

\newpage

\begin{abstract} This is an expanded version of the notes by the second author of the lectures on Hitchin systems and their quantization given by the first author at the Beijing Summer Workshop in Mathematics and Mathematical Physics ``Integrable Systems and Algebraic Geometry" (BIMSA-2024).
\end{abstract} 

\section{Introduction} 

In 1987, Nigel Hitchin introduced a finite dimensional complex integrable system attached to a reductive group $G$ and a smooth projective curve $X$ over $\mathbb C$, which is now called the {\it Hitchin integrable system}. This system lives on (a partial compactification of) the cotangent bundle $T^*{\rm Bun}^\circ_G(X)$ to the moduli space ${\rm Bun}^\circ_G(X)$ of stable principal $G$-bundles on $X$, and admits a direct generalization to the ramified case, when $X$ has marked points with various kinds of level structure at them. In this generalized form, Hitchin systems (even restricted to genus $0$ and $1$) and their degenerations subsume (upon $q$-deformation) most of the known finite dimensional integrable systems. Also, Beilinson and Drinfeld showed in \cite{BD} that Hitchin systems admit a natural quantization, which is a crucial ingredient in their proposed {\it geometric Langlands correspondence}, finally proved in 2024 for all $G$ (in the stronger categorical form) in \cite{GR,GeL}. Finally, 
quantum Hitchin systems play a central role in the recently introduced {\it analytic Langlands correspondence} (\cite{Te,Fr,EFK,EFK2,EFK3,EFK4}) and the interpretation of the geometric and analytic Langlands correspondence in supersymmetric 4-dimensional  quantum gauge theory (\cite{KW,GW}). This puts Hitchin systems in the center of attention in several areas of mathematics and mathematical physics. 

The goal of this paper is to give an accessible introduction to Hitchin systems and their quantization. 
It is an expanded version of the notes by the second author of the lectures given by the first author at the Beijing Summer Workshop in Mathematics and Mathematical Physics ``Integrable Systems and Algebraic Geometry" (BIMSA-2024)
dedicated to the memory of Igor Krichever, who did foundational work on Hitchin systems and 
stressed their importance over several decades. 

We begin by recalling the notion of a principal bundle in algebraic geometry (Section 2), 
and then we discuss the moduli stack ${\rm Bun}_G(X)$, where $X$ is a smooth irreducible projective curve and $G$ a split reductive group\footnote{The basic examples of $G$ are $\GL_n$, $\SL_n$,
and $\PGL_n$, and also the classical groups (symplectic and
orthogonal). A particularly important case is the multiplicative group
$\GL_1$, and split tori, which are products of several copies of $\GL_1$.} over some field $\mathbf k$ (Section 3).
Then we proceed to the construction of ${\rm Bun}_G(X)$ as a double quotient of the loop group 
(Section 4). In Section 5, we introduce a smooth subvariety ${\rm Bun}_G^\circ(X)\subset {\rm Bun}_G(X)$ of stable bundles and consider its cotangent bundle, which is the space of Higgs pairs. In Section 6, we introduce the classical Hitchin system and prove Hitchin's theorem on its integrability (the proof of completeness is given only for type $A$). This leads us to the important notion of a {\it spectral curve}. In Section 7, we generalize the Hitchin system to the (tamely) ramified case, and show how some known integrable systems, such as the (twisted) Garnier system and elliptic Calogero-Moser system, arise as examples. Finally, in Section 8 we discuss quantization of Hitchin systems, which is based on the representation theory of affine Lie algebras. 

For pedagogical purposes the paper contains problems (supplied with solutions) and exercises (without solutions). Solutions of problems are given in Section 9. 

{\bf Acknowledgements.} We are very grateful to the organizers of the
Beijing Summer Workshop in Mathematics and Mathematical Physics for
their support and to BIMSA for hospitality. The work of the first
author was partially supported by the NSF grant DMS-2001318. The work
of the second author was supported by the World Premier International
Research Center Initiative (WPI), MEXT, Japan.

\section{Principal \texorpdfstring{$G$}{G}-bundles}

\subsection{Definition of a principal \texorpdfstring{$G$}{G}-bundle}

Let $\mathbf k$ be a field. The basic case for us will be $\mathbf k =
\bC$. Let $X$ be an algebraic variety and $G$ be an affine algebraic
group, both over $\mathbf k$.

\begin{definition}\label{prinbu}
   A {\it principal
    $G$-bundle} on $X$, also called a {\it $G$-torsor} over $X$, is an algebraic variety $P$ over $\mathbf k$ equipped with
  a regular map $\pi\colon P \to X$ and a right action $a\colon P\times G\to P$ of $G$ preserving  $\pi$, which (\'etale) locally on $X$ is isomorphic to the projection $G \times X \to X$ with the $G$-action by right multiplication on the
  $G$-factor.
\end{definition}

In other words, there exists an (\'etale) open cover $\{U_i\}_{i \in
  I}$ of $X$ such that for all $i\in I$, there is a $G$-invariant isomorphism $P|_{U_i} \cong G \times U_i$ under which $\pi$ goes to the projection $G\times U_i\to U_i$.\footnote{An \'etale chart is more general than a Zariski chart: it is a
Zariski open set $\overline U \subset X$ along with a finite \'etale morphism (cover) $U
\to \overline U$. (Recall that a morphism is {\it \'etale} if it is flat and unramified.) We will not focus on this too much, because one can
work with \'etale charts in much the same way as with Zariski charts.
Namely, by definition, intersections $U_i \cap U_j$ are given by fiber products $U_i \underset{X}{\times} U_j$, and $P|_{U_i} \coloneqq P \underset{X}{\times} U_i$. Naively we can just pretend that the \'etale charts $U_i$ are ordinary Zariski charts.}

The same definition can be made more
generally when $X$ (and hence $P$) is a scheme.

Let $P_x \coloneqq \pi^{-1}(x)$ be the fiber of $\pi$ at a point $x \in
X$. This is a principal homogeneous space for $G$. Thus a principal $G$-bundle on $X$ can be thought of as a family of principal homogeneous $G$-spaces parametrized by $X$. 

\subsection{Clutching functions}

It follows that a principal $G$-bundle can be realized concretely by
{\it clutching} (or {\it transition}) {\it functions}, namely, regular
functions $g_{ij}\colon U_i \cap U_j \to G$ defined when $U_i\cap
U_j\ne \emptyset$ such that:
\begin{enumerate}
\item $g_{ii}=\id$, $g_{ij} \circ g_{ji} = \id$;
\item the {\it $1$-cocycle condition}: $g_{ij} \circ g_{jk} \circ g_{ki} = \id$
  on the triple intersection $U_i \cap U_j \cap U_k$ (when it is non-empty).
\end{enumerate}
This means that the $\{g_{ij}\}$
form a {\it {\v C}ech $1$-cocycle} on $X$ with coefficients 
 in the sheaf $G_{X}$ of $G$-valued regular functions on $X$. 
The idea is that we glue trivial bundles on each $U_i$ using these
clutching functions.\footnote{More precisely, for \'etale
covers, gluing is slightly more complicated than for usual open covers, 
and one must use a
procedure called {\it faithfully flat descent} (\cite{O}, Chapter 4).}

In fact, this data classifies principal
$G$-bundles modulo transformations
\[ g_{ij} \mapsto h_i \circ g_{ij} \circ h_j^{-1} \]
for regular functions $h_i\colon U_i \to G$. In other words, 
cohomologous cocycles define isomorphic bundles, and vice versa. 
Hence $G$-bundles on $X$ up to isomorphism are classified by the {\it first \'etale cohomology} $H^1_{\et}(X,G_X)$. 

Note that this cohomology is in general just a set (not a group); it only has a natural group structure if $G$ is abelian (in which case the product is defined by multiplying clutching functions). 

\subsection{Vector bundles} 

\begin{definition}\label{vecbu}
  A {\it vector bundle} of rank $n$ on a scheme $X$ over $\mathbf k$ is
  a scheme $P$ over $\mathbf k$ equipped with a regular map $\pi\colon P
  \to X$ with operations of fiberwise addition
  $P\underset{X}{\times}P\to P$ and scalar multiplication $\mathbf
  k\times P\to P$ which (\'etale) locally on $X$ is isomorphic to the
  projection $\mathbf k^n\times X \to X$ with the usual fiberwise
  addition and scalar multiplication. In this case, for $x \in X$, the
  $n$-dimensional vector space $\pi^{-1}(x)$ is called the {\it fiber}
  of $\pi$ at $x$. A vector bundle of rank $1$ is called a {\it line
    bundle}.
\end{definition}
  
In other words, a vector bundle on $X$ of rank $n$ is defined by an atlas of \'etale charts $\lbrace U_i\rbrace$ on $X$ with transition functions 
$g_{ij}\colon U_i\times_X U_j\to \GL_n$ satisfying the 1-cocycle condition.
Thus the notion of a vector bundle of rank $n$ is equivalent 
to the notion of a principal $\GL_n$-bundle.     
  
It is easy to see that a vector bundle $E$ on $X$ is determined by its {\it sheaf of sections} 
$U\mapsto \Gamma(U,E)$ for \'etale charts $U$ on $X$, and if $U$ is affine then $\Gamma(U,E)$ 
is a locally free (i.e., projective) $\mathbf k[U]$-module\footnote{For an affine scheme $U$ over a field $\mathbf k$, the ring of regular functions on $U$ will be denoted by $\mathbf k[U]$ or $\mathcal O(U)$.} of rank $n$. Moreover, 
as we will see below, we may consider the usual Zariski charts instead of \'etale ones. 
Thus vector bundles of rank $n$ on $X$ are the same as locally free sheaves of the same rank.
In particular, if $X$ is affine then a vector bundle on $X$ of rank $n$ is the same thing as 
a projective $\mathbf k[X]$-module of the same rank. 
  
Finally, recall that on vector bundles over any scheme there are natural operations of taking 
direct sum, tensor product, and dual bundle $E\mapsto E^\vee$, which 
are defined by the corresponding operations on the fibers (or clutching functions). 
For example, $g_{E^\vee,ij}=(g_{E,ij}^{-1})^*$. 

\subsection{Associated bundles} 

Let $Y$ be an algebraic variety with a left action of $G$. 
Then any principal $G$-bundle $E$ on $X$ gives rise to the {\it associated bundle} $E\times_G Y$ on $X$ with fiber $Y$, with the same clutching functions as $E$
but viewed as acting on $Y$. 

For instance, if $G = \GL_n$ and $Y$ is the vector representation of $G$, then $E\times_G Y$ is the {\it vector bundle associated to the principal $G$-bundle $E$}. As noted above, this defines a one-to-one correspondence between principal $\GL_n$-bundles and vector bundles of rank $n$ on $X$.

Similarly, $\SL_n$-bundles bijectively correspond to vector bundles of
rank $n$ with trivial determinant, and, on a curve over an algebraically closed field, $\PGL_n$-bundles bijectively
correspond to vector bundles of rank $n$ up to tensoring with line
bundles\footnote{In higher dimensions there can be more $\PGL_n$-bundles, corresponding to Azumaya algebras.}. More generally, given a finite-dimensional rational\footnote{Meaning, every $v \in V$ is contained in a finite-dimensional $G$-invariant subspace on which $G$ acts algebraically. This holds automatically if $\rho$ is a morphism of group schemes.}
representation
$$
\rho\colon G \to \GL(V),
$$ 
every principal $G$-bundle $E$ on $X$ with clutching functions $g_{ij}$ gives rise to the associated vector bundle $E_\rho=E\times_G V$ of rank $\dim V$, whose clutching functions are $\rho(g_{ij})$. 

\subsection{Induced bundles}
Let $H,G$ be affine algebraic groups over $\mathbf k$, $E$ be a
principal $H$-bundle on $X$ and $\varphi\colon H \to G$ be a
homomorphism. Then the associated bundle $F\coloneqq E\times_H G$ is a
principal $G$-bundle on $X$ called the {\it induced $G$-bundle} from
$E$ via $\varphi$. Reciprocally, if $F$ is a principal $G$-bundle on
$X$ then an {\it $H$-structure} on $F$ is an $H$-bundle $E$ on $X$
with an isomorphism $F\cong E\times_H G$.

Similarly, we can talk about an $H$-structure on $F$ over a (locally
closed) subscheme $X'\subset X$. For example, if $H\subset G$ then an
$H$-structure on $F$ over a point $x\in X$ is just an $H$-orbit in the
fiber $F_x$. Note that choices of an $H$-structure on $F$ at $x$ are parametrized by $G/H$, and that $H$ matters only up to
conjugation, because right-multiplication by a group element $g \in G$
transforms an $H$-orbit into a $gHg^{-1}$-orbit.

\subsection{Classification of line bundles} 
  Let $G = \GL_1$ and assume for simplicity that $\mathbf k$ is algebraically closed. If $G_{X}$ denotes the (group-valued) sheaf associated to $G$, i.e. $G_X(U) \coloneqq \Gamma(U, \cO_U)^\times$, then its first \'etale cohomology is
  $$
  H^1_{\et}(X,G_X) = \Pic(X)
  $$ 
  (or, more precisely, $\Pic(X)(\mathbf k)$), the Picard group of $X$, which classifies line bundles on $X$. 
 In particular, if $X$ is a smooth irreducible projective curve over $\mathbf k$ then 
there is a (non-canonically split) short exact sequence
\begin{equation} \label{ses1}
0 \to \Pic_0(X) \to \Pic(X) \xrightarrow{\deg} \bZ \to 0 
\end{equation}
in which the projection $\Pic(X)\to \mathbb Z$ takes the degree of a line bundle. The kernel $\Pic_0(X)$ of this projection (the group of isomorphism classes of line bundles of degree $0$) is exactly the {\it Jacobian} $\Jac(X)$ of $X$; for $\mathbf k =
\bC$, as an analytic manifold it is a compact complex torus of dimension $g$, where $g$ is the genus of $X$. 
  
  For example, let $X=\mathbb P^1$. Then 
  $\Pic(X) = \bZ$ and all line bundles on $X$ have the form
  $\cO(n) \coloneqq \cO(1)^{\otimes n}$ for $n \in \bZ$. To construct
  $\cO(n)$ explicitly, cover $\bP^1 = \bA^1 \cup \infty$ with two
  charts, $U_0 \coloneqq \bA^1=\bP^1\setminus \infty$ and $U_\infty \coloneqq \bP^1 \setminus
  0$. Now, every vector bundle on $\bA^1$ is trivial; this is because
  vector bundles on an affine scheme $S$ are just finite projective modules over the ring of regular functions on $S$,
  and all finite projective modules over $\mathbf k[\mathbb A^1]=\mathbf k[x]$ are free. So we may
  use these charts to define vector bundles (in particular, line bundles) 
  on $\bP^1$. Their
  intersection is $U_0 \cap U_\infty = \bA^1 \setminus 0 = \bG_m$, so
  we must specify a clutching function
  \[ g\colon \bG_m \to \GL_1=\bG_m \]
  from $U_\infty$ to $U_0$, i.e., a non-vanishing regular function on $\mathbb G_m$. Such non-vanishing functions are all of
  the form $g(z) = cz^n$ for some non-zero constant $c$ and integer $n
  \in \bZ$. The constant $c$ may be absorbed by rescaling, for example
  on $U_0$, so we may take $g(z) = z^n$. The clutching function $g(z)
  = z^n$ defines the line bundle $\cO(n)$. This confirms that the line bundles on $\mathbb P^1$ 
  are $\cO(n)$, $n\in \mathbb Z$. 

  By definition, sections of $\cO(n)$ are pairs of polynomials
  $x_0(z)$ and $x_\infty(z^{-1})$ such that $x_0(z) = z^n
  x_\infty(z^{-1})$. Hence $\cO(n)$ has an $(n+1)$-dimensional space
  of sections if $n \ge 0$, and no non-zero sections if $n < 0$. Also 
  $\cO(n)^\vee\cong \cO(-n)$. 

\subsection{Smooth and analytic \texorpdfstring{$G$}{G}-bundles, Serre's GAGA theorem}

Definition \ref{prinbu} is completely parallel to the one used in
differential geometry and topology, in which case $G$ is a real Lie
group and we consider $C^\infty$ functions instead of regular
functions. The same is true for complex analytic geometry, where $G$
is a complex Lie group and we use holomorphic functions. (Of course,
in these cases, one may consider usual open sets instead of \'etale
ones.)

One of the most important basic theorems in geometry of $G$-bundles is 
{\it Serre's GAGA theorem}, named after his paper ``G\'eom\'etrie Alg\'ebrique et G\'eom\'etrie Analytique" (\cite{S}) where it was proved.

\begin{theorem}[Serre's GAGA theorem]\label{gaga}
  If $X$ is a complex projective variety, then there is an equivalence
  of categories
  \[ \begin{pmatrix} \text{algebraic }G\text{-bundles on $X$}\end{pmatrix} \xrightarrow{\sim} \begin{pmatrix} \text{analytic } G\text{-bundles on $X$}\end{pmatrix} \]
  obtained functorially by analytification: we view the algebraic variety $X$ 
  as an analytic variety $X^{\rm an}$, 
  and the same for $G$-bundles on $X$.
  \end{theorem}

\begin{remark} The equivalence of Theorem \ref{gaga} extends to coherent sheaves and therefore preserves cohomology of such sheaves. 
\end{remark} 

\begin{problem} \label{bafun}
  Let $L$ be a non-trivial line bundle of degree $0$ on an elliptic
  curve $X$ over $\bC$. Show that
  \begin{enumerate}[label=(\roman*)]
  \item $L|_{X \setminus 0}$ is non-trivial as an algebraic bundle,
    but
  \item $L|_{X \setminus 0}$ is trivial as an analytic bundle.
  \end{enumerate}
  Conclude that the GAGA theorem fails for the (non-projective) curve $X \setminus 0$.
\end{problem} 

\subsection{\'Etale vs Zariski charts}
\label{subsec:etale-vs-zariski-charts}

Finally, let us explain why we need to consider \'etale charts instead
of Zariski charts. The reason is that in usual topology, points have
neighborhoods which are contractible and therefore have trivial
cohomology, while in the Zariski topology, such neighborhoods do not
exist in general (e.g., think about Zariski open sets on the complex
line, which are just complements of finite sets). For this reason,
principal $G$-bundles on algebraic varieties need not be Zariski
locally trivial.

\begin{example}
  Let $X \coloneqq \bC^\times$ and $P \coloneqq \bC^\times$ with map
  $\pi\colon P \to X$ given by $z \mapsto z^2$. This is a principal
  $\mu_2$-bundle, where $\mu_2 =\{1, -1\}$. But it does not
  trivialize on any non-empty Zariski open set in $\bC^\times$. Indeed, 
  the monodromy around $0 \in \bC$
  is multiplication by $-1$, and removing a finite number of points
  will not change the existence of this monodromy.
\end{example}

Even if $G$ is connected reductive, this may still happen,
say, when $X$ is a surface. This is a much more subtle phenomenon which we will not talk about. But for curves, which is the case we are interested in, the situation is simpler, due to
the following (non-trivial) theorem.

\begin{theorem}[Borel--Springer \cite{BS}, Steinberg \cite{St}]\label{BSS}
  If $X$ is a smooth curve and $G$ is a connected reductive group,
  then any principal $G$-bundle $E$ on $X$ is Zariski-locally trivial. In other words, for any $x\in X$, $E$ trivializes after removing a finite subset  not containing $x$ from $X$ (possibly depending on $E$). 
\end{theorem}

\begin{remark} Note that for some connected reductive groups, like $G = \GL_n$ or $\SL_n$, this is actually true for any $X$ (not just smooth curves) by the non-abelian version of {\it Hilbert Theorem 90} (\cite{S2}).
\end{remark} 

\begin{problem}[\cite{Pr}] \label{usegaga} 
  Let $A\colon \bC^\times \to \Mat_{n \times n}(\bC)$ be a meromorphic
  function such that $\det A \not\equiv 0$. Let $q \in \bC^\times$ with
  $|q| < 1$ and consider the $q$-difference equation
  \begin{equation} \label{eq:basic-qde}
    f(qz) = A(z) f(z)
  \end{equation}
  for $f\colon \bC^\times \to \Mat_{n \times n}(\bC)$. Show that
  \eqref{eq:basic-qde} has a meromorphic solution $f(z)$ such that $\det f
  \not\equiv 0$. (Hint: use the GAGA theorem and Hilbert theorem 90.)
\end{problem}

For semisimple groups, one can make an even stronger statement: 

\begin{theorem}[Harder, \cite{Ha}]\label{Har}
  If $G$ is semisimple, then a $G$-bundle on any smooth affine curve
  $X$ is trivial, i.e. on a general irreducible smooth curve it is trivialized if you remove any one point.\footnote{Later this theorem was generalized to families by Drinfeld and
Simpson, \cite{DS}.}
\end{theorem}

Note that Theorem \ref{Har} is false for non-semisimple reductive groups! For instance, \eqref{ses1} implies that on $X \setminus x$, we have an identification
\[ \Pic(X \setminus x) = \Pic(X) / \inner{\cO(x)} \cong \Jac(X). \]
So if $g>0$, line bundles on $X$ are generally not trivialized by removing a point, and one must remove several points. Furthermore, without knowing the line bundle, we don't know which points to remove so that the bundle
becomes trivial. Said differently, given any finite set of points on $X$,
there exists a line bundle on $X$ which is not trivialized after removing those
points.

\begin{problem}\label{bunelcur}
  Let $\cL$ be a line bundle of non-zero degree on a complex elliptic curve
  $X$. Show that there exists a point $x \in X$ such that $\cL|_{X \setminus 
    x}$ is trivial. Is the same true on a genus $2$ curve?
\end{problem}

\section{Moduli of \texorpdfstring{$G$}{G}-bundles on smooth projective curves}
\label{sec:BunG}

\subsection{The stack of \texorpdfstring{$G$}{G}-bundles}
For any $X$ and $G$, let
\[ \Bun_G(X)(\mathbf k) \coloneqq \left\{\begin{array}{c}\text{isomorphism classes of principal }G\text{-bundles}\\\text{on } X \text{ defined over }\mathbf k\end{array}\right\}. \]
More generally, we can take an
extension of scalars: if $A$ is a commutative $\mathbf k$-algebra, then there
is similarly the set $\Bun_G(X)(A)$ of isomorphism classes of $G$-bundles on $X$ defined over $A$, with clutching maps now being regular
functions $U_i \cap U_j \to G$ defined over $A$. One can also say that
\[ \Bun_G(X)(A) \coloneqq \Map(\Spec A, \Bun_G(X)), \]
which is, by definition, the set of isomorphism classes of principal
$G$-bundles on the product $\Spec A \times X$. In other words, there is a {\it
  functor}
\[ \Bun_G(X)\colon A \mapsto \Bun_G(X)(A) \]
from the category of commutative $\mathbf k$-algebras to the category of sets. 
This is the kind of functor we consider in algebraic geometry when we
define schemes --- Grothendieck's {\it functor of points}. However,
our functor $\Bun_G(X)$ is not representable by a scheme; rather, it
is only representable by a more complicated object called an {\it algebraic
stack}. The reason is that bundles, unlike points of a scheme, can have nontrivial automorphisms. 

To keep this text short, we will not systematically discuss the notion of a stack, and will keep the exposition somewhat informal (a detailed discussion can be found, for example, in \cite{O}). As we have just mentioned, the main difference between schemes and stacks is that in schemes, points carry no additional data, but in stacks, every point comes with a group of automorphisms. 

\begin{example}
  Let $G$ be a finite group and let $*$ denote a point. Then there is
  a stack $[*/G]$, also sometimes denoted $BG$, called the {\it
    classifying stack for $G$}. Its functor of points is given by
  \[ \Map(S, [*/G]) = \{\text{principal }G\text{-bundles on } S\}. \]
  The right hand side is a category: it has one object, but this
  object has many automorphisms. So, geometrically, $[*/G]$ has one
  point, but this point has automorphisms.
\end{example}

This suggests that we should think of $\Bun_G(X)(A)$ as a category (more specifically, a groupoid, i.e., a category where all morphisms are isomorphisms) rather than just a set: its objects are $G$-bundles on $X$ defined over $A$ and morphisms are isomorphisms of such bundles.

Algebraic stacks are a sort of globalization of the notion of the
quotient of a variety (or, more generally, a scheme) by a group action. In particular, there is the important
special case of a {\it quotient stack}. Namely, if $Y$ is a variety and $H$ is
an affine algebraic group acting on $Y$, then there is the stack $[Y/H]$ with
functor of points given by
\[ \Map(S, [Y/H]) = \left\{\begin{array}{c} \text{principal }H\text{-bundles } P \to S\\\text{with an }H\text{-equivariant map } P \to Y\end{array}\right\}. \]
This data can be written as the Cartesian square
\[ \begin{tikzcd}
  P \ar{r} \ar{d} & Y \ar{d} \\
  S \ar{r} & {[Y/H]}.
\end{tikzcd} \]
If the $H$-action on $Y$ is sufficiently nice (e.g., free), then $Y/H$ is a nice
topological space, but in general it is not (it has poor separation properties), and the idea of stacks is to never take the quotient and to work equivariantly instead.

To be clear, a representation of a stack as $[Y/H]$, if exists, is not unique, as $[Y/H] = [Y\times_H H' / H']$ whenever $H\subset H'$. In particular, given an inclusion $H\hookrightarrow \GL_n$, 
we can write $[Y/H]$ as $[\widetilde Y/\GL_n]$, where 
$\widetilde Y\coloneqq Y\times_H \GL_n$. Thus we may always assume without loss of generality that $H=\GL_n$.

It turns out that $\Bun_G(X)$ is not a quotient stack of an algebraic variety by an algebraic group; for example 
$\Bun_{\GL_1}(\mathbb P^1)=\mathbb Z$, so it has infinitely many components. 
Yet this stack can be written as a nested union of open substacks $\Bun_G(X)_n$, $n\ge 0$, which are of the form $[Y_n/H_n]$ where $Y_n$ is a smooth variety and $H_n$ an affine algebraic group (and we may arrange that $H_n=\GL_{m_n}$).

Thus $\Bun_G(X)(\mathbf k)$ has a natural topology (taking quotient topology on $\Bun_G(X)_n$ and then the inductive limit in $n$), but this topology is very non-separated; in particular, $\mathbf k$-points do not have to be closed. This is in contrast to schemes, where $\mathbf k$-points are closed by definition.

\begin{example} \label{pgl}
We see that if $G=\GL_1$ then 
$\Bun_G(X)=\Pic(X)$ is a stack which is not a scheme, because $\Aut(\cL) = {\mathbf k}^\times$ for any line bundle $\cL$; it is usually called the {\it Picard groupoid} of $X$ (this is the category whose objects are line bundles on $X$ and morphisms are isomorphisms of line bundles). But this stack is not so different from a scheme:  every point has the {\it same} automorphism group, and we can {\it    rigidify} to get a scheme. To ``rigidify'' here means to consider
  bundles along with the extra data of an element in the fiber over
  some fixed point $x_0 \in X$. This extra data kills the automorphism
  group. 

  In contrast, if $G = \PGL_2$, and $X =
  \bP^1$, the simplest possible curve, then we can have automorphism groups of arbitrarily large dimension. The more degenerate the bundle becomes, the bigger the automorphism group will be. So in this case $\Bun_G(X)$ is already a ``true" stack. 

Indeed, consider the $\PGL_2$-bundle $E_n \coloneqq \cO(n) \oplus \cO(0)$
  for $n > 0$, whose clutching function is the diagonal matrix $g(z) =
  \begin{pmatrix} z^n & 0\\ 0 & 1\end{pmatrix}$. Automorphisms of $E_n$ have the form
 $\begin{pmatrix} a_{11} & a_{12} \\ a_{21} & a_{22} \end{pmatrix},$
where $a_{11}, a_{22} \in \bC$, $a_{12} \in \Hom(\cO(0), \cO(n)) = H^0(\cO(n))$ lives in an
  $(n+1)$-dimensional space, and $a_{21} \in \Hom(\cO(n), \cO(0)) =
  H^0(\cO(-n)) = 0$. The constraint that the automorphism is
  an element in $\PGL_2$ means that we can set $a_{22}=1$. Putting it all
  together, we get that $\dim \Aut(E_n) = (n + 1) + 1 = n + 2$.
\end{example}

\begin{remark}\label{genus>2} If the genus of $X$ is $\ge 2$ and the group $G$ is semisimple and adjoint (i.e., has trivial center) then one can show that a generic $G$-bundle on $X$ has no non-trivial automorphisms. 
This means that the stack $\Bun_G(X)$ has a dense open set of ``generic" bundles, which is a scheme (in fact, a smooth algebraic variety). Yet understanding the geometry of $\Bun_G(X)$ requires considering arbitrarily degenerate bundles, around which this stack is no longer a scheme and has a very complicated structure. 
\end{remark} 

\subsection{Classification of vector bundles on \texorpdfstring{$\bP^1$}{P1}}

\begin{theorem}\label{grrank2}
  Every rank $2$ vector bundle $E$ on $\bP^1$ is isomorphic to $\cO(n)
  \oplus \cO(m)$ for unique $n\ge m$.
\end{theorem}

\begin{proof}
  Uniqueness is clear, because we can recover
  \[ n = \max\{i : \Hom(\cO(i), E) \neq 0\} \]
  from $E$, and $m=\deg E-n$. 
  
  Existence of $(n,m)$ is more interesting. We start with a simple lemma. 
  
  \begin{lemma}\label{l1} Let $E$ be a vector bundle on $\mathbb P^1$.
  
 (i) There exists $m\in \mathbb Z$ such that $\Hom(\cO(m),E)\ne 0$. 
   
 (ii) If $\varphi\colon \cO(m)\to E$ is a nonzero homomorphism then for some $r\ge 0$ there exists a nonvanishing homomorphism $\widehat\varphi\colon \cO(m+r)\to E$, with $r>0$ unless 
  $\varphi$ is nonvanishing.
 
 (iii) If $m\ge n-1$ then any short exact sequence 
 $$
    0 \to \cO(m) \to E \to \cO(n) \to 0
 $$
   splits. 
 \end{lemma} 
  
  \begin{proof} (i) Let $\varphi$ be a nonzero regular section of $E$ over $\mathbb A^1$ 
  (it exists since $E|_{\mathbb A^1}$ is trivial as noted above). Then 
  $\varphi$ extends to a rational section of $E$ over $\mathbb P^1$ 
  of some order $m$ at $\infty$. Thus $\varphi$ defines 
  a regular map $\varphi\colon \cO(m)\to E$. 
  
  (ii) If $\varphi$ vanishes only at points $p_1, \ldots, p_s$ to orders
  $r_1, \ldots, r_s$, then it defines a nonvanishing map 
  \[ \widehat\varphi\colon \cO(m) \otimes \cO(\sum_{i=1}^s r_i p_i) = \cO(m + \sum_{i=1}^s r_i)\to E. \].
  
  (iii) Such short exact sequences are classified by
 \begin{equation}\label{spl} \Ext^1_{\bP^1}(\cO(n), \cO(m)) = H^1(\bP^1, \cO(m-n)) \cong H^0(\bP^1, \cO(n-m-2))^* 
 \end{equation}
  where $\cong$ is Serre duality using the canonical bundle $\cK =
  \cO(-2)$. Hence if $m \ge n-1$, this Ext group vanishes
  and thus the sequence splits. 
  \end{proof} 
  
  Now fix a rank $2$ vector
  bundle $E$ on $\bP^1$. By Lemma \ref{l1}(i),(ii) there is a 
  nonvanishing map $\varphi\colon \cO(m) \to E$ for some $m\in \mathbb Z$. 
 Hence we obtain a short exact sequence
  \begin{equation} \label{eq:P1-rank-2-bundle-SES}
    0 \to \cO(m) \xrightarrow{\varphi} E \to \cO(n) \to 0
  \end{equation}
  where $n=\deg E-m$. (The cokernel of $\varphi$ is a line bundle since $\varphi$ is nonvanishing, so it has the form $\cO(n)$ for some
  $n \in \bZ$, and it is clear that $n=\deg E-m$.) 
 Let $R$ be the maximum of integers $r$ such that  $\Hom(\cO(r), E) \neq 0$. It exists and satisfies $R\le \max(m,n)$ by applying $\Hom(\cO(r), -)$
  to \eqref{eq:P1-rank-2-bundle-SES}; namely, for $r>\max(m,n)$, both
  $\Hom(\cO(r), \cO(m))$ and $\Hom(\cO(r), \cO(n))$ vanish, so
  $\Hom(\cO(r), E) = 0$ as well. 
  
  Recall that $E$ is realized by the
  clutching function
  \[ g(z) = \begin{pmatrix} z^m & f(z) \\ 0 & z^n \end{pmatrix} \]
  where $f(z)$ is a Laurent polynomial. We can modify $g$ by $g\mapsto h_1 \circ g
  \circ h_2^{-1}$ where
  \[ h_1 = \begin{pmatrix} 1 & \varphi(z) \\ 0 & 1 \end{pmatrix}, \qquad h_2 = \begin{pmatrix} 1 & \psi(z^{-1}) \\ 0 & 1 \end{pmatrix}, \]
and $\varphi,\psi\in \mathbf k[u]$. Using such a modification, we can reduce uniquely
  to the case where $f$ has no monomials $z^s$ except where $m < s <
  n$. A map $\cO(r) \to E$ is a section of the bundle $\cO(-r) \otimes E$ with
  transition map
  \[ \begin{pmatrix} z^{m-r} & z^{-r} f(z)  \\ 0 & z^{n-r} \end{pmatrix}. \]
  Hence a regular section of this bundle is a pair $\binom{x_0(z)}{y_0(z)}, \binom{x_\infty(z^{-1})}{
  y_\infty(z^{-1})}$ of vector-valued polynomials, such that
  \[ \begin{pmatrix} z^r x_0(z) \\ z^r y_0(z) \end{pmatrix} = \begin{pmatrix} z^m & f(z) \\ 0 & z^n \end{pmatrix} \begin{pmatrix} x_\infty(z^{-1}) \\ y_\infty(z^{-1}) \end{pmatrix}. \]
  In particular, the left hand sides $z^r y_0$ and $z^r x_0$ have no
  monomials of degree $< r$. So we just need to find $x_\infty$ and
  $y_\infty$ so that the right hand sides also have no monomials of
  degree $< r$. This means that $\deg y_\infty \le n-r$ and $f(z)
  y_\infty(z)$ has no terms of degree between $m+1$ and $r$. This
  gives  $r-m$ homogeneous linear equations in $n-r+1$ unknowns. Hence non-zero
  solutions definitely exist if $r - m < n-r+1$, or, equivalently, $r \le
  \frac{m+n}{2}$. So we have $R
  \ge \frac{m+n}{2}$. Moreover, any nonzero 
  map $\cO(R) \to E$ must be an embedding,
  otherwise $R$ is not maximal by Lemma \ref{l1}(ii). So
  there is a short exact sequence
  \[ 0 \to \cO(R) \to E \to \cO(m+n-R) \to 0 \]
  Since $R\ge m+n-R$, this short exact sequence splits by Lemma \ref{l1}(iii), as desired.
\end{proof}

\begin{corollary}[Grothendieck, \cite{Gr}]\label{grothen}
  A rank $n$ vector bundle on $\bP^1$ is uniquely of the form
  \[ \cO(m_1) \oplus \cdots \oplus \cO(m_n) \]
  for integers $m_1 \ge \cdots \ge m_n$.
\end{corollary}

\begin{proof}
  The proof is by induction on the rank $n$. Let us prove the statement for rank $n+1$ assuming
  it is known for smaller ranks. Let $m_0$ be the maximal possible integer $r$ for
  which $\Hom(\cO(r), E) \neq 0$. By the induction hypothesis and Lemma \ref{l1}(i),(ii), we have
  an extension
\begin{equation}\label{ses}
0 \to \cO(m_{0}) \to E \to \cO(m_1) \oplus \cdots \oplus \cO(m_n) \to 0. 
\end{equation}
  For $j = 1, \ldots, n$, consider the subbundle $E_j \subset E$
  which is the preimage of $\mathcal O(m_j)$. Then we have a short
  exact sequence
  \[ 0 \to \cO(m_0) \to E_j \to \cO(m_j) \to 0, \]
  and $m_{0}$ is the maximal possible integer $r$ for which
  $\Hom(\cO(r), E_j) \neq 0$. So by Theorem \ref{grrank2}, $E_j = 
  \cO(m_{0})\oplus \cO(m_j)$ and $m_0 \ge m_j$. Since this is true for
  any $j = 1, \ldots, n$, the sequence \eqref{ses}
  splits by Lemma \ref{l1}(iii) and $E \cong \cO(m_0) \oplus \cO(m_1)\oplus \cdots \oplus \cO(m_n)$, as
  desired.
\end{proof}

\begin{problem} \label{PGLSL}
  Classify $\PGL_n$-bundles and $\SL_n$-bundles on $\bP^1$.
\end{problem}

\begin{problem} \label{hankel}
(i)  Let $E$ be the rank $2$ vector bundle on $\bP^1$ given by the
  short exact sequence
  \[ 0 \to \cO(0) \to E \to \cO(m) \to 0 \]
  for $m \ge 2$, with transition function (from $U_\infty$ to $U_0$)
  \[ g(z) \coloneqq \begin{pmatrix} 1 & f(z) \\ 0 & z^{m} \end{pmatrix} \]
  for some polynomial $f(z) = a_1 z + \cdots + a_{m-1}
  z^{m-1}$. Given $\frac{m}{2}\le n\le m$, 
  find the condition on $a_1, \ldots, a_{m-1} \in \bC$
  such that 
 $E \cong \cO(k) \oplus \cO(m-k)$
for some $\frac{m}{2}\le k\le n$.   
  
(ii) For even $m$ and $n=\frac{m}{2}$, find a polynomial $H(a_1,\ldots,a_{2n-1})$ 
such that $E\cong \cO(n)\oplus \cO(n)$ iff $H(a_1,\ldots,a_{2n-1})\ne 0$.  
\end{problem}

\begin{problem} \label{Bstr}
 Let $B \subset \GL_n$ be the subgroup of
  upper-triangular matrices. 
  Show that any $\GL_n$-bundle $E$ on a smooth projective curve $X$
  admits a $B$-structure, i.e. is associated to a (non-unique)
  $B$-bundle.\footnote{This actually holds for any reductive group, see \cite{DS} and references therein.}
\end{problem}

\subsection{Bundles with level structure and local presentations of \texorpdfstring{${\rm Bun}_{\PGL_2}(\mathbb P^1)$}{Bun\_PGL2(P1)}} 

For $G = \PGL_2$, Theorem~\ref{grrank2} (and the solution to Problem~\ref{PGLSL}) implies that every $G$-bundle $E$ on $\mathbb P^1$ is of the form $\mathcal O(n)\oplus \mathcal O(0)$ for a unique $n=N(E)\ge 0$. This allows us to describe in more detail the structure of the stack ${\rm Bun}_G(\mathbb P^1)$.

To do so, consider the nested sequence of open substacks ${\rm Bun}_{G}(\mathbb P^1)_n\subset {\rm Bun}_{G}(\mathbb P^1)$
of bundles $E$ with $N(E)\le n$; i.e., for a commutative algebra $A$, ${\rm Bun}_{G}(\mathbb P^1)_n(A)$ 
is the set of isomorphism classes of $G$-bundles on ${\rm Spec}A\times \mathbb P^1$ which at every geometric point of ${\rm Spec}A$ give a bundle on $\mathbb P^1$ with $N(E)\le n$. We would like to present 
${\rm Bun}_{G}(\mathbb P^1)_n$ as the quotient $[Y_n/H_n]$ of a smooth variety $Y_n$ by an algebraic group $H_n$. There are, in fact, many different (albeit equivalent) ways to choose this presentation, which correspond to various ways 
to ``rigidify" the classification problem for $G$-bundles, so that the automorphism groups of all bundles with $N(E)\le n$ are killed.\footnote{Recall from Example \ref{pgl} that if $N(E)>0$ then ${\rm Aut}(E)$ has dimension 
$N(E)+2$, so these automorphism groups get arbitrarily large. Thus for large $n$ the group $H_n$ must be large as well, as it must contain these automorphism groups for all $E$ with $N(E)\le n$.}  

One convenient way to do so is to consider bundles $E$ equipped with {\it level $n$ structure} at $0$, i.e., with a trivialization over the {\it $n$-th 
infinitesimal neighborhood} $\widehat 0_n\coloneqq
{\rm Spec}(\mathbf k[z]/z^{n+1})$ of $0\in \mathbb P^1$. If we 
present bundles $E$ by trivializing them on $U_0,U_\infty$ and writing down the clutching function as above, 
then such a trivialization of $E$ over $\widehat 0_n$ differs from its trivialization 
over $U_0$ restricted to $\widehat 0_n$ by an element $h$ of the group $H_n\coloneqq G(\mathbf k[z]/z^{n+1})$. On the other hand, in Example \ref{pgl} we saw that an automorphism of $\cO (n)\oplus \cO (0)$ for $n>0$ is given in the chart $U_0$ by a matrix $(a_{ij})$ with $a_{22}=1$, $a_{11}={\rm const}$, $a_{21}=0$ and $a_{12}=p(z)$, where $p$ is a polynomial of degree $n$. On the other hand, if $n=0$ then an automorphism is given by a constant matrix. This implies 
that if $N(E)\le n$ then any non-trivial automorphism of $E$ changes the element $h$, so ${\rm Aut}(E,h)=1$. From this one can deduce that the  
stack $Y_n\coloneqq\widetilde{\rm Bun}_{G}(\mathbb P^1)_n$ classifying $G$-bundles $(E,h)$ on $\mathbb P^1$ with level $n$ structure at $0$ such that $N(E)\le n$ is actually a smooth algebraic variety with an action of $H_n$ (by changing $h$), and 
$$
{\rm Bun}_{G}(\mathbb P^1)_n=[Y_n/H_n].
$$
Note that the action of $H_n$ on $Y_n$ is not free: the stabilizer of $(E,h)$ is isomorphic to ${\rm Aut}(E)$. 

More generally, one can define $Y_n$ by taking $h$ to be a collection of level $n_i$ structures $h_i$ at any given distinct points $z_i\in \mathbb P^1$, so that $\sum_i (n_i+1)=n+1$. In this case $H_n$ will be $\prod_i G(\mathbf k[z]/z^{n_i+1})$. 
The above definition then corresponds to the case when all $z_i$ coalesce at $0$. 

Now consider the set ${\rm Bun}_{G}(\mathbb P^1)(\mathbf k)\cong \mathbb Z_{\ge 0}$ (where the bijection sends $E$ to $N(E)$), and let us ask what topology on $\mathbb Z_{\ge 0}$ is induced by the stack structure. First of all, note that ${\rm Bun}_{G}(\mathbb P^1)$ has two disjoint parts 
${\rm Bun}_{G}(\mathbb P^1)^{\rm even}$, ${\rm Bun}_{G}(\mathbb P^1)^{\rm odd}$, corresponding to bundles of even and odd degree. Thus we have a decomposition of topological spaces
$\mathbb Z_{\ge 0}=\mathbb Z_{\ge 0}^{\rm even}
\sqcup \mathbb Z_{\ge 0}^{\rm odd}$. Also since the sets $\Bun_{G}(\mathbb P^1)_n$ are open, we see that the subsets 
$$
U_r^{\rm even}=\lbrace 0,2,\ldots,2r-2\rbrace\subset \mathbb Z_{\ge 0}^{\rm even},\ U_r^{\rm odd}=\lbrace 1,3,\ldots,2r-1\rbrace\subset \mathbb Z_{\ge 0}^{\rm odd}
$$ 
are open for all $r\ge 0$. 

\begin{proposition} $U_r^{\rm even},U_r^{\rm odd}$ are the only proper open subsets of $\mathbb Z_{\ge 0}^{\rm even}, \mathbb Z_{\ge 0}^{\rm odd}$. In particular, the latter spaces are connected. 
\end{proposition} 

\begin{proof} It suffices to show that the closure of any $k\in \mathbb Z_{\ge 0}$ contains $k+2$. To this end, consider the space of short exact sequences 
$$
0\to \cO(0)\to E\to \cO(k+2)\to 0;
$$
as we have seen, it is the vector space $V_k=\Ext^1_{\mathbb P^1}(\cO(k+2),\cO(0))=H^0(\mathbb P^1,\cO(k))^*$ of dimension $k+1$. 
For $v\in V_k$ let $E_v$ denote the corresponding bundle. 
It is clear that $E_0\cong \cO(k+2)\oplus \cO(0)$, and it is easy to show that 
if $E_v=\cO(k+2)\oplus \cO(0)$ then $v=0$. 
On the other hand, 
there is $w\in V_k, w\ne 0$ such that $E_w\cong \cO(k+1)\oplus \cO(1)$, which is isomorphic to $\cO(k)\oplus \cO(0)$ as a $G$-bundle (see Problem \ref{hankel}). 
Hence $E_{\lambda w}\cong \cO(k)\oplus \cO(0)$ as a $G$-bundle for any 
nonzero scalar $\lambda$. Thus the family of bundles $E_{\lambda w}$, $\lambda\in \mathbb A^1$ has 
$$
N(E_{\lambda w})=\begin{cases} k+2,\ \lambda=0\\ k,\ \lambda\ne 0\end{cases}
$$ 
It follows that $k+2$ lies in the closure of $k$ and we are done. 
\end{proof} 

In particular, it follows that $\Bun_{\SL_2}(\mathbb P^1)(\mathbf k)=\mathbb Z_{\ge 0}^{\rm even}$ with the above topology, and $\Bun_{\SL_2}(\mathbb P^1)$ is the nested union of 
open substacks $\Bun_{\SL_2}(\mathbb P^1)_n$, comprising bundles with 
$\Hom(\cO(n+1),E)=0$, which correspond to the open subsets $U_{n+1}^+$. 
Moreover, 
$$
\Bun_{\SL_2}(\mathbb P^1)_n\cong \Bun_{\PGL_2}(\mathbb P^1)_{2n}^+=\widetilde{\Bun}_{\PGL_2}(\mathbb P^1)_{2n}^+/H_{2n},
$$ 
where $\widetilde{\Bun}_{\PGL_2}(\mathbb P^1)_{2n}^+$
is the variety classifying pairs $(E,h)$ where $E$ is a $GL_2$-bundle of degree $2n$ with $N(E)\le 2n$ and $h$ is a level structure on $E$ at $0$ of order $2n$. 

In conclusion, let us give a concrete realization of the variety $\widetilde{\Bun}_{\PGL_2}(\mathbb P^1)_{2n}^+$, by embedding it into the Grassmannian 
${\rm Gr}(2n,4n+2)$. To this end, let $(E,h)\in \widetilde{\Bun}_{\PGL_2}(\mathbb P^1)_{2n}^+$ and consider the space of sections 
$H^0(\mathbb P^1,E\otimes \cO(-1))$, which has dimension $2n$. 
For each $s\in H^0(\mathbb P^1,E\otimes \cO(-1))$, the trivialization $h$ 
gives rise to its $2n$-th Taylor approximation $\nu_h(s)=\binom{s_{1}}{s_{2}}$, 
where $s_{j}\in \mathbf k[z]/z^{2n+1}$. This defines a linear map 
$$
\nu_h\colon H^0(\mathbb P^1,E\otimes \cO(-1))\to \mathbf k^2[z]/z^{2n+1}\cong \mathbf k^{4n+2}.
$$
It follows from Theorem \ref{grrank2} that $\nu_h$ is injective, so we obtain an $H_{2n}$-equivariant map 
$$
\nu\colon \widetilde{\Bun}_{\PGL_2}(\mathbb P^1)_{2n}^+\to {\rm Gr}(2n,\mathbf k^2[z]/z^{2n+1})={\rm Gr}(2n,4n+2).
$$
To describe its image, note that $\widetilde{\Bun}_{\PGL_2}(\mathbb P^1)_{2n}^+$ has $n+1$ $H_{2n}$-orbits, i.e., those of the bundles $E_j=\cO(n+j)\oplus \cO(n-j)$, $0\le j\le n$, with the standard trivializations $h^j$ at $0$. For these bundles, we have 
$$
\nu(E_j,h^j)=V_j\coloneqq\lbrace \binom {s_1}{s_2}\colon \deg(s_1)\le n+j-1,\ \deg(s_2)\le n-j-1\rbrace.
$$
Thus $\im\nu=\sqcup_{0\le j\le n}H_{2n}\cdot V_j$. But it is easy to check that 
$\overline{H_{2n}\cdot V_0}\supset V_j$, so we get that 
$$
\im\nu\subset \overline{H_{2n}\cdot V_0}
$$
is an open dense $H_{2n}$-invariant subset. 

One can also show that $\nu$ is injective; in fact it is an embedding. This gives rise to an explicit realization of $\widetilde{\Bun}_{\PGL_2}(\mathbb P^1)_{2n}^+$ as a locally closed subvariety of dimension $6n$ in ${\rm Gr}(2n,4n+2)$, as desired. 
 
\begin{remark} One can do the same taking $h$ to be a collection of level $n_i$ structures $h_i$ at any given distinct points $z_i\in \mathbb P^1$, so that $\sum_i (n_i+1)=2n+1$, as noted above. The least degenerate case is $n_i=0$ 
for $i\in [1,2n+1]$. In this case $H_{2n}=G^{2n+1}$ which acts naturally 
on ${\rm Gr}(2n,(\mathbf k^2)^{2n+1})={\rm Gr}(2n,4n+2)$, and $\nu_h(s)$ is defined by taking the values of $s$ at $z_1,\ldots,z_{2n+1}$. In this case 
$Y_{2n}\subset {\rm Gr}(2n,4n+2)$ is an open subset of $\overline{G^{2n+1}\cdot V_0}$, where $V_0\subset (\mathbf k^2)^{2n+1}$ is the space of values of pairs of polynomials $(s_1(z),s_2(z))$ of degree $n-1$ at $z_1,\ldots,z_{2n+1}$. 

All the other cases are degenerations of this, when some of the points $z_i$ coalesce, 
with the original setting corresponding to all the points $z_i$ coalescing at $0$. 
\end{remark} 

\begin{remark} Similar methods can be used to locally represent $\Bun_G(X)$ as a quotient for general $X$ and $G$. 
\end{remark} 

\subsection{Principal \texorpdfstring{$G$}{G}-bundles on \texorpdfstring{$\bP^1$}{P1}}

We will now explain how to generalize Grothendieck's theorem from $\GL_n$ to an
arbitrary split connected reductive group $G$. For this, we first need to reformulate it. 

Recall that rank $n$ vector bundles are the same as $\GL_n$-bundles,
and $\GL_n$ is the group of invertible $n \times n$ matrices, so it
contains a maximal torus $T$ consisting of diagonal matrices with
non-zero entries on the diagonal: $T = \GL_1^n$. 
Grothendieck's theorem says that every $\GL_n$-bundle
on $\bP^1$ admits a $T$-structure, meaning that it is associated to some $T$-bundle $E$ by sending $E$ to $E\times_T \GL_n$. Recall that $T$-bundles on $\mathbb P^1$ correspond to $n$-tuples of integers $\mathbf m=(m_1,\ldots,m_n)$: the $T$-bundle $E(\mathbf m)$ corresponding to such a tuple is $\cO(m_1)^\times\times\cdots\times \cO(m_n)^\times$, where $\cO(m)^\times$ is the bundle of nonzero vectors in 
$\cO(m)$.  
This shows that the same $\GL_n$-bundle can
come from many different $T$-bundles, because the data of a $T$-bundle
is sensitive to the ordering of the integers $m_1, \ldots, m_n$.
Precisely, if $E(\mathbf m), E(\mathbf n)$ are $T$-bundles on $\bP^1$, then
by Grothendieck's theorem 
\[ E(\mathbf m) \times_T \GL_n \cong E(\mathbf n) \times_T \GL_n \]
if and only if there exists a permutation $w \in S_n$ such that
$\mathbf m = w\mathbf n$. 

Now let $G$ be a split connected reductive group, 
$T \subset G$ a maximal torus, $N(T) \subset G$ the
normalizer of $T$, and $W\coloneqq N(T)/T$ the Weyl group of $(G,T)$.
Then $T$-bundles on $\mathbb P^1$ 
are canonically parametrized by the cocharacter lattice 
$\mathbf X_*(T)\coloneqq\Hom(\mathbb G_m,T)$ (namely, the bundle $E(\mu)$ attached to the cocharacter 
$\mu$ has this cocharacter as its clutching map). Note that $W$ acts 
naturally on $T$, hence on $\mathbf X_*(T)$. 

\begin{theorem}[\cite{Gr}] \label{grgen}
  Any $G$-bundle on $\bP^1$ is associated to a $T$-bundle $E$, and
  given two $T$-bundles $E(\mu)$ and $E(\nu)$, we have 
  \[ E(\mu) \times_T G \cong E(\nu) \times_T G \]
  if and only if $\mu = w \nu$ for some element $w \in W$.
\end{theorem}

\begin{proof}[Proof idea.]
  Reduce to the case of vector bundles by considering representations
  of $G$. For more detail, see \cite{Gr,MS}.
\end{proof}

Thus Theorem \ref{grgen} says that isomorphism classes of $G$-bundles on $\bP^1$ are labeled by $\mathbf X_*(T)/W$.

\begin{remark}
  The cocharacter lattice  $\mathbf X_*(T)$ is the {\it character} or {\it weight lattice}
  $\Lambda^\vee$ of the {\it Langlands dual group $G^\vee$}. Thus 
  Theorem \ref{grgen} says that isomorphism classes of 
  $G$-bundles on $\bP^1$ are labeled by
  \[ \Lambda^\vee/W \cong \Lambda_+^\vee, \]
  the set of dominant integral weights for $G^\vee$, which labels its
  irreducible representations. This is one of the simplest instances of {\it Langlands 
  duality}.
\end{remark}

\section{The double quotient realization of \texorpdfstring{$\Bun_G(X)$}{BunG(X)}}

\subsection{The double quotient construction} 
As before, let $X$ be a smooth irreducible projective curve and $G$ a split
connected reductive group over a field $\mathbf k$. In Section~\ref{sec:BunG}, we
attached to this pair $(X, G)$ the moduli stack $\Bun_G(X)$ of
principal $G$-bundles on $X$. In general, $\Bun_G(X)$ is a very
complicated object, but most of these complications will not be
relevant for us. We defined $\Bun_G(X)$ via its functor of points, but
now we would like to describe $\Bun_G(X)$ in a slightly more explicit
way.

For simplicity, let's assume first that $G$ is semisimple and $\mathbf k$ is
algebraically closed. By Theorem \ref{Har}, every $E
\in \Bun_G(X)$ trivializes once any chosen point is removed from $X$.
So pick a point $x \in X$. Cover $X$ by two charts: a disk around
$x$, and $X \setminus x$. In algebraic geometry, we do not have small
disks, but we can take a {\it formal} disk $D_x$ around $x$ instead.
To describe $G$-bundles on $X$ using these two charts, it suffices to study the transition function on the intersection
\begin{equation}\label{inter}
 (X \setminus x) \cap D_x = D_x^\times 
 \end{equation} 
 where $D_x^\times$ is the {\it punctured formal
  disk}. 
  
Let us explain the precise meaning of this equality. 
Let $\cO_{(x)}$ be the local ring of the point $x$, 
  $$
  \mathbf k[D_x]=\cO_x\coloneqq\underleftarrow{\lim}_{n\to \infty}\cO_{(x)}/\mathfrak{m}_x^n,
  $$
  where $\mathfrak{m}_x\subset \cO_{(x)}$ is the maximal ideal, and 
  $$
  \mathbf k[D_x^\times] \eqqcolon K_x
  $$ 
 be the field of fractions of $\cO_x$. 
More concretely, if $t$ is a formal coordinate on $X$ near $x$, then we obtain identifications\footnote{  In complex analysis, holomorphic functions on a disk are given by convergent Taylor series, and holomorphic functions on a punctured
  disk which are meromorphic at the puncture 
  are given by convergent Laurent series which are finite in the
  negative direction. We make these series formal by removing the
  convergence assumption, and then they make sense over any field.}
$$
 \cO_x \cong \mathbf k[[t]],\ K_x \cong \mathbf k((t)).
 $$
 Also let $R_x \coloneqq \cO(X \setminus x)$ be the
ring of regular functions on the affine curve $X \setminus x$.
Then we have an inclusion 
$R_x\hookrightarrow K_x$ defined by the Laurent series expansion at $x$, 
(it is easy to show that it does not depend on the choice of the formal coordinate), and algebraically equality \eqref{inter} just means that 
$$
R_x\cdot \cO_x=K_x. 
$$

Now, $G$-bundles $E$ on $X$ are defined by transition
maps $g(z)$ from $D_x$ to $X \setminus x$, or equivalently, elements
$g \in G(K_x)$, up to $g \mapsto h_1 \circ g \circ h_2^{-1}$ where $h_1 \in G(R_x)$
and $h_2 \in G(\cO_x)$. We have thus arrived at the following proposition.

\begin{proposition}\label{doubquot}
  $\Bun_G(X)(\mathbf k) = G(R_x) \backslash G(K_x) / G(\cO_x)$.
\end{proposition}

It is useful to take this double quotient in two steps. 
Namely, for any reductive $G$ consider the one-sided quotient
$\Gr_G \coloneqq G(\mathbf k((t))) / G(\mathbf k[[t]])$,
called the {\it affine Grassmannian}.\footnote{Recall that the usual Grassmannian is the quotient of 
$\GL_n$ by a maximal parabolic subgroup. Similarly, ${\rm Gr}_G$ 
is the quotient of the affine Kac-Moody group $G(\mathbf k((t)))$ 
by the maximal parabolic subgroup $G(\mathbf k[[t]])$. This 
explains the origin of the term 
``affine Grassmannian''.}  It is an {\it ind-variety} ---
infinite-dimensional, but a nested union of projective varieties of
increasing dimension. We then obtain that for $G$ semisimple, 
$G$-bundles on $X$
correspond to orbits of $G(R_x)$ on $\Gr_G$. 

\begin{example} If $G=\GL_n$ then geometric points of $\Gr_G$ are 
lattices $L\subset \mathbf k((t))^n$, i.e. finitely generated spanning 
$\mathbf k[[t]]$-submodules (which are necessarily free of rank $n$). Note that for any such $L$ there exists $N$ such that 
\begin{equation}\label{squeeze} 
t^{-N}\mathbf k[[t]]^n\supset L \supset t^N\mathbf k[[t]]^n.
\end{equation} 
Thus ${\rm Gr}_G$ is the nested union of the sets ${\rm Gr}_{G,N}$
of lattices satisfying  \eqref{squeeze}. Note that each 
${\rm Gr}_{G,N}$ can be realized as a closed subvariety 
of the Grassmannian ${\rm Gr}(\mathbf k[[t]]^n/t^{2N})={\rm Gr}(2nN)$ 
(disjoint union of Grassmannians of subspaces of all dimensions in $\mathbf k^{2nN}$) 
by sending $L$ to $t^NL/t^{2N}\mathbf k[[t]]^n$ (namely, it is the locus of all subspaces which are invariant under multiplication by $t$), i.e., it has a 
natural structure of a projective variety. Thus 
${\rm Gr}_G$ is a nested union of projective varieties. 
\end{example}  

We can generalize this construction by removing multiple points from
$X$ instead of just one. Namely, let $S \subset X$ be a non-empty
finite subset, and take the two charts $U_1 \coloneqq X \setminus S$
and $U_2 \coloneqq \bigsqcup_{x \in S} D_x$. Then
\[ U_1 \cap U_2 = \bigsqcup_{x \in S} D_x^\times, \]
and therefore, by the same reasoning as above,
\begin{equation} \label{eq:BunG-semisimple}
  \Bun_G(X)(\mathbf k) = G(R_S) \Big\backslash \prod_{s \in S} G(K_x) \Big/ \prod_{x \in S} G(\cO_x),
\end{equation}
where $R_S\coloneqq\mathbf k[X\setminus S]$.

This, however, does not quite work for non-semisimple reductive groups
$G$, e.g. $G=\GL_1 = \bG_m$, since in this case there is no finite set
$S$ such that all $G$-bundles are trivialized on $X \setminus S$ (as
explained at the end of Subsection~\ref{subsec:etale-vs-zariski-charts}).
So, to generalize the above construction to such groups, we will
remove {\it all} geometric points of $X$. This sounds like then there
will be nothing left, but in fact this is not the case; the
``Grothendieck generic point'' still remains! Indeed, removing a
single point in algebraic geometry means considering rational
functions which are allowed to have a pole at that point. So, removing
all points means considering rational functions which are allowed to
have poles anywhere, i.e. simply all rational functions. We thus
obtain the presentation
\begin{equation} \label{eq:BunG-general}
  \Bun_G(X)(\mathbf k) = G(\mathbf k(X)) \Big\backslash \sideset{}{'}\prod_{x \in X} G(K_x) \Big/ \prod_{x \in X} G(\cO_x)
\end{equation}
which is now valid for all reductive groups. 

The prime in the product is a technical but important detail. It denotes the {\it
  restricted product} consisting of elements with only finitely many
coordinates having poles, i.e. not lying in $G(\cO_x)$. The restricted
product arises because we are taking a colimit of
\eqref{eq:BunG-semisimple} over {\it finite} sets $S$ in order to
obtain \eqref{eq:BunG-general}.

Finally, if $\mathbf k$ is not algebraically closed, we can perform the same
construction using finite subsets $S \subset X(\overline{\mathbf k})$ which are
Galois-invariant, where $\overline{\mathbf k} \supset \mathbf k$ is the algebraic closure of $\mathbf k$.
Namely, let $\Gamma \coloneqq \Gal(\overline{\mathbf k}/{\mathbf k})$, then
\[ \Bun_G(X)(\mathbf k) = G(\mathbf k(X)) \Big\backslash \sideset{}{'}\prod_{x \in X(\overline{\mathbf k})/\Gamma} G(K_x) \Big/ \prod_{x \in X(\overline{\mathbf k})/\Gamma} G(\cO_x). \]
For example, if $\mathbf k$ is finite, then all the valuations $v$ of
$K=\mathbf k(X)$ come from points $x\in X(\overline{\mathbf k})/\Gamma$,
and the corresponding completions $K_v=K_x$ of $K$ with respect to $v$
are locally compact non-discrete topological fields (also called {\it
  local fields}). Such fields $K$ are called {\it global fields}. We
get
\[ \Bun_G(X)(\mathbf k) = G(K) \bigg\backslash G\bigg(\sideset{}{'}\prod_{v \in \Val(K)} K_v\bigg) \bigg/ G\bigg(\prod_{v \in \Val(K)} \cO_v\bigg), \]
where $\Val(K)$ is the set of valuations of $K$. This is called an
{\it arithmetic quotient}. 

\subsection{Analogy with number theory} 
Similar quotients arise in number theory. This is because while global
fields of characteristic $p>0$ all have the form $\mathbf k(X)$ for finite fields $\mathbf k$, in characteristic $0$ they are {\it number fields}, i.e. finite extensions of $\bQ$.

\begin{definition}
  If $K$ is a global field, the {\it ring of ad\`eles} of $K$ is
  \[ \bA \coloneqq \bA_K \coloneqq \sideset{}{'}\prod_{v \in \Val(K)} K_v. \]
\end{definition} 
  
 Note that while for $K=\mathbf k(X)$, all valuations are 
 non-archimedean (discrete), for a number field $K$, there are two kinds of valuations: archimedean (embed $K$ into $\bC$ and take the absolute value) and
  non-archimedean (discrete, or $p$-adic valuations). Rings of integers $\cO_v
  \subset K_v$ make sense only for non-archimedean valuations. Let
  \[ \cO_{\bA} \coloneqq \prod_{v \in \Val_{n.a.}(K)} \cO_v \]
  where $\Val_{n.a.}(K)$ is the set of non-archimedean valuations of $K$.
  Then we can consider the {\it arithmetic quotient} 
  \[ \cM \coloneqq G(K) \backslash G(\bA) / G(\cO_\bA). \]

This generalizes $\Bun_G(X)$, because if $K = \mathbf k(X)$, then $\mathcal M =
\Bun_G(X)(\mathbf k)$.

\begin{example}
  Let $K = \bQ$. Then there are $p$-adic (non-archimedean) valuations of $K$ with respect to all primes $p$, with completion $\mathbb Q_p$, and also the usual archimedean valuation $x\mapsto |x|$
  corresponding to $p=\infty$, with completion $\mathbb R$. Thus
  \begin{align*}
    \bA &= \bR \times \sideset{}{'}\prod_{p \text{ prime}} \bQ_p \\
    \cO_\bA &= \prod_{p \text{ prime}} \bZ_p,
  \end{align*}
  and
  \[ \cM = G(\bQ) \Big\backslash \bigg(G(\bR) \times \sideset{}{'}\prod_{p \text{ prime}} G(\bQ_p)\bigg) \Big/ \prod_{p \text{ prime}}G(\bZ_p) 
  \]
  Moreover, by the weak approximation theorem (\cite{PR}, p.402), 
  $$
  G(\mathbb Q)
  \cdot \prod_{p \text{ prime}}G(\bZ_p)=\sideset{}{'}\prod_{p \text{ prime}} G(\bQ_p),
  $$
  and 
  $$
  G(\mathbb Q)
  \cap \prod_{p \text{ prime}}G(\bZ_p)=G(\mathbb Z), 
  $$
  hence
  \[\cM= G(\bZ) \backslash G(\bR). \]
  For instance, if $G = \Sp_{2n}$, then $\cM = \Sp(2n, \bZ) \backslash
  \Sp(2n, \bR)$, and taking a quotient by $U(n)$ on the right gives
  the moduli space of $n$-dimensional abelian varieties
  \[ \cA_n = \Sp(2n, \bZ) \backslash \Sp(2n, \bR) / U(n). \]
  In particular, if $G = \SL_2 = \Sp_2$, then
  \[ \cA_1 = \SL_2(\bZ) \backslash \SL_2(\bR) / U(1) \]
  is the moduli space of elliptic curves, which is the home of classical modular forms. 
 \end{example}
 
\begin{example}  
  Let $G = \GL_1$ and $\mathbf k$ be
  any field. Then we obtain
  \[ \Pic(X)(\mathbf k) = \GL_1(\mathbf k(X))\backslash \GL_1(\bA) / \GL_1(\cO_{\bA})=\mathbf k(X)^\times \backslash \bA^\times / \cO_{\bA}^\times. \]
  The quotient of $\bA^\times/\cO_{\bA}^\times$ by $\mathbf k^\times$ is
  the group of divisors ${\rm Div}(X)$ on $X$, and $\mathbf k(X)^\times/\mathbf k^\times$ 
  is its subgroup of principal divisors, ${\rm PDiv}(X)$. (Here, $\mathbf k^\times$ is the
  intersection $\mathbf k(X)^\times \cap \cO_{\bA}^\times$.) Thus we recover the 
  classical definition of the Picard group of $X$: 
  $$
  \Pic(X)(\mathbf k)={\rm Div}(X)/{\rm PDiv}(X).
 $$ 
\end{example}

\section{Stable bundles and Higgs fields}

\subsection{Stable bundles}\label{stabu}
Hitchin systems are integrable Hamiltonian systems, and Hamiltonian systems live on
symplectic manifolds (or varieties), so we need to come up with one. A natural way to create a symplectic variety in our setting is to consider the cotangent bundle $T^*\Bun_G(X)$. But, as explained in Section 3, $\Bun_G(X)$ is not a variety, or even a scheme. Rather, it is a stack, and the cotangent bundle of a stack is well-defined only as another stack. We will avoid these difficulties, however, and work only with a
particular open set in $T^*\Bun_G(X)$ --- the cotangent bundle 
$T^*\Bun_G^\circ(X)$ of the subset $\Bun_G^\circ(X)$ of suitably defined ``generic" bundles, a smooth open subvariety of $\Bun_G(X)$ which is nonempty and dense if the genus $g$ of $X$ is $\ge 2$ (see Remark \ref{genus>2}).

Assume $g \ge 2$ and first suppose that $G$ is simple and adjoint. Then a ``generic'' 
bundle in $\Bun_G(X)$ has trivial automorphism group, so in the local
presentation of $\Bun_G(X)$ as a quotient of an algebraic variety by a
group, the group acts freely at such a bundle, and, moreover, the orbit of a ``generic'' bundle is closed. The locus of
``generic'' bundles therefore forms a smooth algebraic variety
$\Bun_G^\circ(X)$.

There are many ways to specify what ``generic'' means. We will use a
{\it stability condition}. 

For instance, let $G = \PGL_n$. Recall that $G$-bundles are rank $n$
vector bundles modulo tensoring with line bundles. Recall also that if $E$ is a vector bundle on $X$, then there are two integers attached to it: the degree $d(E)$ (given by the first Chern
class), and the rank $r(E)$.

\begin{definition} If $E\ne 0$ then 
  the {\it slope} of $E$ is
  \[ \mu(E) \coloneqq \frac{d(E)}{r(E)}. \]
  We say $E$ is {\it stable} if for every subbundle $0 \neq E'
  \subsetneq E$,
  \[ \mu(E') < \mu(E). \]
\end{definition}

 There is a more technical definition, due to Ramanathan, 
 for other reductive groups $G$,
  which we will not state. We refer the reader to \cite{R} for details. 

It is easy to see that if $L$ is a line bundle and $E$ a vector bundle, then $E$ is
  stable if and only if $E \otimes L$ is stable. This implies that stability is well-defined for $\PGL_n$-bundles. 
  
\begin{theorem}[\cite{R}]
  Stable bundles have the trivial group of automorphisms, and form a
  smooth variety which is an open subset $\Bun_G^\circ(X) \subset
  \Bun_G(X)$.
\end{theorem}

\begin{example}
  The converse is not true, i.e. unstable bundles may also have
  trivial automorphism groups. Take a line bundle $L$ on $X$ of degree
  $1$, and consider the group $\Ext^1(\cO_X, L)$ which parameterizes
  extensions $0 \to L \to E \to \cO_X \to 0$. If $X$ has genus $\ge
  3$, then this group is non-zero by Riemann--Roch. Any such extension
  $E$ is unstable since $\mu(L) = 1 > 1/2 = \mu(E)$. However, for any
  such non-trivial extension, the automorphism group of $E$ as a
  $\PGL_2$-bundle is isomorphic to $H^0(L)$, which is trivial for
  generic $L$ of degree $1$.
\end{example}

\begin{remark}\label{hms} Let $\cM_G^\circ(X) \coloneqq T^*\Bun_G^\circ(X)$. The Hitchin system
will initially live on $\cM_G^\circ(X)$, but actually there is a
partial compactification $\cM_G(X)$ of $\cM_G^\circ(X)$ called the
{\it Hitchin moduli space}, which is still symplectic, to which the
Hitchin system naturally extends (see \cite{Hl}). For the simplest example of this, 
see Subsection \ref{lastsec}. 
This sort of extension to a partial
compactification is a common phenomenon in integrable systems.
\end{remark}

For general semisimple $G$, not necessarily of adjoint type, there is
a straightforward extension of this story. Namely, a bundle is ``generic'' if
it is {\it regularly stable}, meaning that it is stable and its group
of automorphisms reduces to the center $Z(G)$ (which is the smallest
it can be). The resulting $\Bun_G^\circ(X)$ is still a stack with
stabilizer $Z(G)$ at every point, but because this stabilizer is the
same everywhere, we can rigidify like we did for ${\rm Pic}(X)$ and obtain a variety. In
other words, we may ignore the stackiness and just consider the
underlying variety. \footnote{Stability is still important here: in general, the locus consisting of bundles whose group of automorphisms reduces to $Z(G)$ is not a scheme, since the orbit of such a bundle may fail to be closed.}

\subsection{Higgs fields} 
Before we go further, we should compute $\dim \Bun_G^\circ(X)$, or,
equivalently since it is a smooth variety, the dimension $\dim
T_E\Bun_G^\circ(X)$ of the tangent space at a point $E \in
\Bun_G^\circ(X)$. This tangent space is just the deformation space of
the bundle $E$, classifying its first-order deformations.

Let $\lie{g}$ be the Lie algebra of $G$. For a $G$-bundle $E$ on $X$, let $\ad E$ be the {\it adjoint bundle} of $E$, i.e., the vector bundle $E_\lie{g}$ associated to $E$ via the adjoint representation $\lie{g}$ of $G$.

\begin{problem}\label{adE}  Show, using \v Cech
  $1$-cocycles, that the tangent space to $\Bun_G^\circ(X)$ at $E$ (i.e., the deformation space of $E$) is
  $H^1(X, \ad E)$.
\end{problem}

\begin{example}
  Let $G = \GL_n$. First order deformations of a vector bundle $E$ are classified
  by $\Ext^1(E, E)$; this is unsurprising because, affine-locally, we
  are just deforming modules over the coordinate rings $\mathbf k[U]$, $U\subset X$. Since
  \[ \Ext^1(E, E) = \Ext^1(\cO, E^* \otimes E) = H^1(X, \ad E), \]
  this agrees with our claim that deformations of $E$ are classified by
  $H^1(X, \ad E)$.
\end{example}

The invariant pairing on
$\lie{g}$ may be used to identify $(\ad E)^* \cong \ad E$. Using this identification 
and Serre duality, it follows that
\begin{align*}
  T_E^*\Bun_G^\circ(X) = H^1(X, \ad E)^*
  &\cong H^0(X, \cK_X \otimes (\ad E)^*) \\
  &= H^0(X, \cK_X \otimes \ad E).
\end{align*}
Elements of $H^0(X, \cK_X \otimes \ad E)$ have a very concrete geometric
interpretation: they are $1$-forms on $X$ with coefficients in $\ad
E$. From physics, they have the name {\it Higgs fields} on $E$.
The space of Higgs fields on $E$ is usually denoted by $\Omega^1(X,\ad E)$. 
A pair $(E,\phi)$ where $\phi\in \Omega^1(X,\ad E)$ is called a {\it Higgs pair}.

\subsection{The dimension of the variety of stable bundles}
It remains to compute $\dim H^0(X, \cK_X \otimes \ad E)$. To this end, we will use the following lemma. 

\begin{lemma}\label{euchar} 
Let $V$ be a vector bundle on $X$ of degree $d$ and rank $r$. 
Then the Euler characteristic $\chi(X,V)$ equals 
$d-(g-1)r$. 
\end{lemma} 

\begin{proof} By the Riemann-Roch theorem, if $L$ is a line bundle on $X$ 
then 
$$
\chi(X,L)=\deg L-(g-1).
$$ 
By Problem \ref{Bstr}, 
$V$ has a $B$-structure, i.e., a filtration $0=V_0\subset V_1\subset...\subset V_r=V$ 
such that $V_{i}/V_{i-1}$ are line bundles. Thus 
$$
\chi(X,V)=\sum_{i=1}^r\chi(X,V_i/V_{i-1})=\sum_{i=1}^r(\deg(V_i/V_{i-1})-(g-1))=d-(g-1)r.
$$
\end{proof} 

Now note that since the vector bundle ${\rm ad} E$ is self-dual, it has degree $0$. Thus 
$\cK_X\otimes {\rm ad} E$ has degree $2(g-1)\dim G$. So it follows from Lemma \ref{euchar} 
that 
$$
\chi(X,\cK_X\otimes {\rm ad} E)=(g-1)\dim G.
$$ 

On the other hand, by Serre duality $\dim H^1(X, \cK_X \otimes \ad E) = \dim H^0(X, \ad E)$, and $H^0(X, \ad E)=0$  
because stable bundles have no (infinitesimal)
automorphisms.\footnote{This is where we use that $G$ is semisimple,
because otherwise generic $G$-bundles still have infinitesimal automorphisms, and so
$\dim H^1$ will not vanish in the Euler characteristic.} Thus 
$$
\dim H^0(X,\cK_X\otimes {\rm ad}E)=\dim H^0(X,\cK_X\otimes {\rm ad}E)-\dim H^1(X, \cK_X \otimes \ad E)
$$
$$
=\chi(X,\cK_X\otimes {\rm ad}E)=(g-1)\dim G.
$$
Hence we get 

\begin{proposition} If $G$ is semisimple then $\dim {\rm Bun}^\circ_G(X)=(g-1)\dim G$.
\end{proposition} 

For $G = \bG_m$, we know that $\Bun_G(X) = \Pic(X)$ has dimension $g$. So
for general reductive $G$, we have 
\begin{proposition}
\[ \dim \Bun_G^\circ(X) = (g - 1) \dim G + \dim Z(G). \]
\end{proposition}

\begin{example}
  For $G = \GL_n$, we have 
  $$
  \dim \Bun_G^\circ(X)=(g - 1)n^2 + 1.
  $$
\end{example}

\section{The classical Hitchin integrable system}

\subsection{The classical Hitchin system for \texorpdfstring{$\GL_n$}{GLn}, \texorpdfstring{$\SL_n$}{SLn} and \texorpdfstring{$\PGL_n$}{PGLn}}
We are now ready to introduce the main character of our story --- the Hitchin integrable system. We will take $\mathbf k=\mathbb C$. As a warm-up, we
begin with $G = \SL_n$ or $\PGL_n$. In this case, as we've just shown, 
\[ \dim \Bun_G^\circ(X) = (n^2 - 1)(g - 1), \]
and elements in $\Bun_G^\circ(X)$ are pairs $(E, \phi)$ where $E$ is a
stable bundle and $\phi \in \Omega^1(X, \End E)$ is
a Higgs field with trace zero.

\begin{definition}[Hitchin system for $G = \SL_n$]
The {\it Hitchin base} is the vector space 
\[\cB=\cB_{X,G}\coloneqq\bigoplus_{i=1}^{n-1} H^0(X, \cK_X^{\otimes (i+1)}).\]

  The {\it Hitchin map} is
  $$
    p\colon T^*\Bun_G^\circ(X) \to \cB, \quad
   p (E, \phi)\coloneqq(\tr \wedge^2 \phi, -\tr \wedge^3 \phi, \ldots, (-1)^n \tr \wedge^n \phi).
 $$
 \end{definition}
 
\begin{remark}
  Note that the fibers of $\ad E$ are isomorphic to $\lie{sl}_n$ 
    non-canonically; namely, the isomorphism is well defined only up to inner
  automorphisms. However, since the functions $\tr\wedge^i A$ of
  a matrix $A$ are conjugation-invariant, $\tr\wedge ^i \phi$ is
  still well defined for any Higgs field $\phi$.
\end{remark}
 
  By the Riemann--Roch theorem, $\dim
  H^0(X, \cK_X^{\otimes (i+1)}) = (2i+1)(g-1)$, and so the dimension of
  the Hitchin base is
  \[ \dim \cB=\sum_{i=1}^{n-1} (2i+1)(g-1) = (n^2 - 1)(g - 1) = \dim \Bun_G^\circ(X). \]

\begin{theorem}[Hitchin, \cite{H}]\label{hisl_n}
The map $p$ is generically a Lagrangian fibration\footnote{A Lagrangian fibration is a fibration with Lagrangian fibers.}, i.e., defines an integrable system.
\end{theorem}

This means that coordinate functions on $\cB$, pulled back by $p$, are
Poisson-commuting and functionally independent. Explicitly, if we
choose a basis $b_1, \ldots, b_d \in \cB$ (say, compatible with the direct sum decomposition) and write $$p(E, \phi) =
\sum_{j=1}^d H_j(E, \phi) b_j,$$ then $H_j(E,\phi)$ are the Poisson-commuting 
{\it classical Hitchin Hamiltonians}, forming an integrable system. 

We will prove Theorem \ref{hisl_n} in Subsections \ref{st1},\ref{st2} and \ref{st3}.

Note that Theorem \ref{hisl_n} extends verbatim to the equivalent case
$G=\GL_n$. In this case, we just need to include a linear term into
the Hitchin map. Namely, if we set
\[\cB=\cB_{X,G}\coloneqq\bigoplus_{i=0}^{n-1} H^0(X, \cK_X^{\otimes (i+1)})\]
and 
$$
    p\colon T^*\Bun_G^\circ(X) \to \cB, \quad
   p (E, \phi)\coloneqq(-\tr \phi, \tr \wedge^2 \phi, -\tr \wedge^3 \phi, \ldots, (-1)^n \tr \wedge^n \phi).
$$
(the full collection of coefficients of the characteristic polynomial 
$\det(\lambda-\phi)$ of $\phi$), then Theorem \ref{hisl_n} with the same formulation (and proof) holds for $\GL_n$. 

\begin{remark} Another (equivalent) choice for Hitchin Hamiltonians 
is the collection of components of $(\tr\phi,\tr\phi^2,\ldots,\tr\phi^n)$. 
\end{remark} 

\begin{example}\label{gl1hit} 
If $G=\GL_1$ then the Hitchin Hamiltonians (say, for degree zero bundles) 
are just the momenta $p_i:=({\rm tr}\phi)_i$, $i=1,\ldots,g$, on $T^*{\rm Jac}(X)={\rm Jac}(X)\times H^0(\cK_X)$.
\end{example} 

\subsection{Classical Hitchin system for general \texorpdfstring{$G$}{G}}
To generalize the classical Hitchin system to an arbitrary reductive
group $G$, we first need to recall {\it Chevalley's theorem} for the Lie algebra $\lie{g}\coloneqq{\rm Lie}G$:

\begin{theorem}[see {\cite[Section 10]{E}}] The algebra $\bC[\lie{g}]^G$ 
of $G$-invariant polynomials on $\lie{g}$ is a polynomial algebra $\bC[Q_1, \ldots, Q_r]$,
where $r \coloneqq \rank G$ and $Q_i$ are homogeneous polynomials.
\end{theorem} 

Let $d_i \coloneqq \deg Q_i$, so that $d_1\le d_2\le\cdots\le d_r$. 
These numbers do not depend on the choice of $Q_i$ and are called 
the {\it degrees} of $G$ (or $\lie{g}$), and it is clear that 
the number of $i$ such that $d_i=1$ is $\dim Z(G)$. Moreover, 
if $e\coloneqq\sum_{i=1}^r e_i\in \lie{g}$ is the regular nilpotent then $\ad e$ 
is a direct sum of Jordan blocks of sizes $2d_i-1$, $1\le i\le r$ (\cite[Lemma 17.1]{E}). 
Thus 
\begin{equation}\label{sumformula}
\sum_{i=1}^r (2d_i-1)=\dim G.
\end{equation} 

\begin{example}
  $\bC[\lie{gl}_n]^{\GL_n} = \bC[Q_1,Q_2,\ldots,Q_n]$, where 
  $Q_m(A)\coloneqq(-1)^m\tr \wedge^m A$. Thus the degrees 
  of $\GL_n$ are $1,2,\ldots,n$. 
  Similarly, $\bC[\lie{sl}_n]^{\SL_n} = \bC[Q_2,\ldots,Q_n]$. Thus the degrees 
  of $\SL_n$ or $\PGL_n$ are $2,\ldots,n$.
  \end{example}

Let $Q \in \bC[\lie{g}]^G$ be a homogeneous polynomial of degree $m$, and
let $(E, \phi)$ be a Higgs pair. A conjugation-invariant function like
$Q$ may be evaluated fiberwise on $\phi \in \Omega^1(X,
\ad E)$ to produce elements
\[ Q(\phi) \in H^0(X, \cK_X^{\otimes m}). \]
The Hitchin base should therefore be 
$$
\mathcal B=\mathcal B_{X,G}\coloneqq\bigoplus_{i=1}^r H^0(X,
\cK_X^{\otimes d_i}).
$$

\begin{definition}[Hitchin system for general $G$]
  The {\it Hitchin map} is
  \[ p\colon T^*\Bun_G^\circ(X) \to \cB_{X,G}, \quad p(E, \phi) \coloneqq (Q_1(\phi), \ldots, Q_r(\phi)). \]
\end{definition} 
  
 Recall that $\dim H^0(X,\cK_X^{\otimes m})$ equals $g$ if $m=1$ and 
$(2m-1)(g-1)$ if $m>1$. Hence 
$$
\dim \cB=\sum_{i=1}^r (2d_i-1)(g-1)+|\lbrace i: d_i=1\rbrace|=(g-1)\dim G+\dim Z(G)=\dim \Bun_G^\circ(X).
$$
Thus, we might hope that $p$ is an integrable system. 
And this indeed turns out to be the case. 
 
\begin{theorem}[Hitchin \cite{H}, Faltings \cite{Fa}, Ginzburg \cite{Gi}; see a discussion in \cite{BD}, 2.2.4 and 2.10] \label{hitgen}
Theorem \ref{hisl_n} holds for any connected reductive group $G$. 
\end{theorem}

The restriction of the Hitchin map to the locus where it is a Lagrangian fibration is called the {\it Hitchin fibration}. 

Hitchin proved Theorem \ref{hitgen} for classical groups $G$, and then Faltings and later Ginzburg proved it for general $G$. 
The proof has two parts:
\begin{enumerate}
\item showing that the coordinates of $Q_j(\phi)$ are in involution
  (i.e., Poisson-commute);
\item showing that they are functionally independent.
\end{enumerate}
Functional independence is equivalent to $p$ being a dominant map, meaning that
the image of $p$ contains an open dense subset.

Below we will discuss part 1 for general $G$ and part 2 for $G=\GL_n$. 

\begin{remark}\label{encoun} 
It is known (\cite{BD}) that Hitchin Hamiltonians are the only global regular functions 
on the cotangent bundle $T^*{\rm Bun}_G(X)$ to the stack ${\rm Bun}_G(X)$, and 
usually the same is true for $T^*{\rm Bun}_G^\circ(X)$ (as unstable bundles occur in codimension $\ge 2$); for instance, this is obvious for $G=GL_1$ (Example \ref{gl1hit}). For this reason, we inevitably encounter the Hitchin system as soon as we start doing geometry or analysis on ${\rm Bun}_G(X)$. 
\end{remark} 

\subsection{Hamiltonian reduction} 
For the proof of Hitchin's theorem, we must first review
the {\it Marsden--Weinstein symplectic (or Hamiltonian) reduction} (see \cite{MR} for more details).

 Let $Y$ be a manifold (or
variety), and $H$ be a Lie group (or algebraic group) acting on $Y$ on
the right. In this case, $H$ acts by Hamiltonian automorphisms on
$T^*Y$, and so there is a moment map
\[ \mu\colon T^*Y \to \lie{h}^* \]
where $\lie{h}$ is the Lie algebra of $H$. This is defined to be dual
to the infinitesimal action map
\[ a\colon \mathfrak h \to \Vect(Y) = \Gamma(Y, T_Y), \]
meaning that
\[ \mu(x, p)(b) \coloneqq \inner{p, a(b)_x}, \qquad \forall (x, p) \in T^*Y, \; b \in \lie{h}. \]

\begin{theorem}[Marsden--Weinstein reduction] If $H$ acts freely on $Y$ then 
$0$ is a regular value of $\mu$ (so $\mu^{-1}(0)$ is smooth) and the quotient $\mu^{-1}(0)/H$ has a natural symplectic structure. Furthermore, there is a natural
  isomorphism of symplectic manifolds
  \[ \mu^{-1}(0)/H \cong T^*(Y/H). \]
\end{theorem}

This can be used to construct integrable systems as follows. Suppose
$\dim Y/H = n$, and $F_1, \ldots, F_n$ are $H$-invariant functions on
$T^*Y$ which are in involution ($\lbrace F_i,F_j\rbrace=0$). Then they define functions $\overline{
F}_i$ on the quotient $T^*(Y/H)$, by first restricting $F_i$ to
the zero fiber $\mu^{-1}(0) \subset T^*Y$ of the moment map, and then descending to the quotient by $H$. 
It is easy to check
that $\{\overline F_i, \overline F_j\} = 0$. There is already the right number
of them to form an integrable system; if in addition they are
functionally independent, then $\overline F_1, \ldots, \overline F_n$ form an
integrable system on $T^*(Y/H)$.\footnote{Note however that there are too few functions to form an integrable system on $T^*Y$, so it was necessary to descend to the quotient to obtain an integrable system.}

\begin{remark} In general, especially when there is no effective way of writing the
functions explicitly, it is difficult to check whether a set of
functions are in involution. But in this method of constructing
integrable systems, sometimes the functions $F_i$ are in involution on
$T^*Y$ for silly reasons. This happens, for instance, when $Y$ is a vector space and the $F_i$ all depend only on the momentum coordinates on $T^*Y = Y \times Y^*$, or, more generally, if $Y$ is a Lie group and $F_i$ depend only on momenta on $T^*Y=Y\times Y^*$ and are conjugation invariant. In fact, this is exactly what is going to happen in the case of Hitchin systems. 
\end{remark}

\subsection{Proof for Hitchin's theorem, part 1}\label{st1}

The proof is based on an infinite dimensional version of Hamiltonian reduction, which involves additional subtleties. While this can be done with full rigor, here we will ignore these subtleties and pretend that we are in the finite dimensional case. 

Our job is to show that the coordinates of $Q_j(\phi)$ Poisson-commute. 
Without loss of generality assume that $G$ is semisimple, and let $x\in X$.
Then by Proposition~\ref{doubquot}
$$
\Bun_G(X) = G(R) \backslash G(K) / G(\cO),
$$
where $R\coloneqq R_x$, $K\coloneqq K_x$, $\cO\coloneqq\cO_x$. 
Denote the preimage of $\Bun_G^\circ(X)$ in $G(K)$
by $G^\circ(K)$. Then the group $G(R) \times G(\cO)$ acts
on $G^\circ(K)$ with stabilizer $Z(G)$. So we may express
$T^*\Bun_G^\circ(X)$ as the Hamiltonian reduction of $T^*G^\circ(K)$
by the free action of $(G(R) \times G(\cO))/Z(G)$. 

We will now construct some $(G(R) \times G(\cO))/Z(G)$-invariant functions on $T^*\Bun_G^\circ(X)$ which manifestly Poisson commute, and then descend them to the Hamiltonian reduction to get the Hitchin Hamiltonians.  
To this end, let us trivialize the
cotangent bundle of $T^*G^\circ(K)$ by left (or right) translations.
Thus we obtain an isomorphism $T^*G^\circ(K)\cong G^\circ(K)\times \lie{g}(K)^*$. 
Observe that the Lie
algebra $\lie{g}(K)$ carries a $G(K)$-invariant pairing, given by
\[ \inner{a(t), b(t)}_{\lie{g}(K)} \coloneqq \Res_{t=0} \inner{a(t), b(t)}_{\lie{g}} \, dt, \]
where $t$ is a formal coordinate near $x$. 
This allows us to identify $\lie{g}(K)^*$ with $\lie{g}((t)) \, dt$.
Hence we get an isomorphism
\[ T^*G^\circ(K) \cong G^\circ(K) \times \lie{g}((t)) \, dt. \]
Thus points of $T^*G^\circ(K)$ can be viewed as pairs $(g, \widetilde \phi)$, 
where $g\in G^\circ(K)$ and $\widetilde \phi$ is a Higgs field on $D_x^\times$. 
So, as before, we can apply invariant functions $\{Q_i\}_{i=1}^r$ on 
$\lie{g}$ to $\widetilde\phi$ to get the functions
\[ H_{i,n} \coloneqq c_n(Q_i(\widetilde \phi)), \quad 1 \le i \le r,\ n \in \bZ \]
(the $n$-th Laurent coefficient of $Q_i(\widetilde \phi)$) 
on $T^*G^\circ(K)$. 

Observe that by construction the functions $H_{i,n}$
depend only on the momentum coordinates on $T^*G^\circ(K)$, and that
the Poisson bracket of momenta on the cotangent bundle to a Lie group coincides with the commutator of the corresponding 
Lie algebra elements. Since in addition the $H_{i,n}$ are invariant, we obtain that 
\[ \{H_{i,n}, H_{j,m}\} = 0. \]

Now, it is easy to check that the following square commutes, with $H_i:=Q_i(\widetilde \phi)$:
\begin{equation} \label{comdia}
  \begin{tikzcd}
    T^*G^\circ(K) \ar{r}{(H_1, \ldots, H_r)} & \bigoplus_{i=1}^r \mathbf k((t)) \, (dt)^{d_i} \\
    \mu^{-1}(0) \ar[hookrightarrow]{u} \ar[twoheadrightarrow]{d} & {} \\
    T^*\Bun_G^\circ(X) \ar{r}{p} & \bigoplus_i H^0(X, \cK_X^{\otimes d_i}) \ar[hookrightarrow]{uu}
  \end{tikzcd}
\end{equation}
where the vertical arrow on the right is just the Taylor
expansion at $x$ of global  differentials on $X$ of degree $d_i$.
Thus, since the Taylor expansion map is injective, 
Poisson-commutativity for the coordinates of the Hitchin map $p$
follows from the Poisson-commutativity of $H_{i,n}$, as desired. 

\begin{remark}\label{desce} We see that the Hamiltonians $H_{i,n}$ descend to $0$ 
for $n<0$ and to $c_n(Q_i(\phi))$ for $n\ge 0$. 
\end{remark} 

\subsection{Spectral curves}\label{st2}

Now we would like to prove functional independence of the Hitchin Hamiltonians for $G = \GL_n$ (the case $n=1$ is trivial, so we will assume that $n\ge 2$). This will use an important technique which appears across the field of integrable
systems: the theory of {\it spectral curves}. 

Take a point
\[ b = (b_1, \ldots, b_n) \in \cB = \bigoplus_{i=0}^{n-1} H^0(X, \cK_X^{\otimes (i+1)}). \]
 Consider the
polynomial
\[ \lambda^n + b_1 \lambda^{n-1} + \cdots + b_n \eqqcolon \prod_{i=1}^n (\lambda - \lambda_i). \]
Since $b_m\in H^0(X, \cK_X^{\otimes m})$, the quantities $\lambda_i$ are all $1$-forms on
$X$ (more precisely, they are branches of a multivalued 1-form). So
\[ \{\lambda_1(x), \ldots, \lambda_n(x)\} \subset T^*_x X. \]
Varying $x \in X$ produces a closed subset $C_b \subset T^*X$ associated to
$b$ - the graph, or Riemann surface, of the multi-valued 1-form
$\lambda(x)$ defined by the equation 
$$
\lambda^n+b_1\lambda^{n-1}+\cdots+b_n=0.
$$ 
In fact, it is clear that $C_b$ is a projective algebraic curve inside $T^*X$ defined
by the equation
\[ y^n + b_1(x)y^{n-1} + \cdots + b_n(x) = 0, \]
where $y$ is the coordinate along the cotangent fibers, and that the
natural projection $\pi\colon C_b \to X$ has degree $n$.

\begin{definition}
$C_b$ is called the {\it spectral curve} of $b$ and the map $\pi$ is called the {\it spectral cover}.
\end{definition}

Now let $(E, \phi) \in T^*\Bun_G^\circ(X)$. Applying $p$ produces $p(E,
\phi) = (b_1, \ldots, b_n)$ where the $b_i$ are just the coefficients of
the characteristic polynomial of $\phi$. In other words,
\[ \lambda^n + b_1 \lambda^{n-1} + \cdots + b_n = \det(\lambda - \phi), \]
and the $\lambda_i(x)$ are just the eigenvalues of $\phi(x)$, for $x \in
X$. So the curve $C(E,\phi)\coloneqq C_{p(E, \phi)}$ is traced out in $T^*X$ by the spectrum of
$\phi(x)$ as $x \in X$ varies, which explains the term ``spectral curve". (It is important to keep in mind that
the eigenvalues are $1$-forms, so they naturally live in $T^*X$).

\begin{theorem}[\cite{H}] \label{smirr} The spectral curve 
$C_b$ is smooth and irreducible for generic $b \in \cB$. 
\end{theorem}

\begin{proof} Note that smoothness and irreducibility of $C_b$ are open conditions 
with respect to $b$. Thus it suffices to show that there exists at least one $b$ for which $C_b$ is smooth and irreducible. 

\begin{lemma}\label{simzer}
  \begin{enumerate}[label=(\roman*)]
  \item If $L$ is a line bundle on $X$ of degree $d\ge 2g$ then a
    generic section of $L$ has only simple zeros.

  \item A generic section of $\cK_X^{\otimes n}$ has only simple zeros.
  \end{enumerate}
\end{lemma} 

\begin{proof} (i) The Riemann-Roch theorem implies that 
if $M$ is a line bundle on $X$ of degree $m\ge 2g-2$ then 
$\dim H^0(X,M)=m-g+1$ unless $m=2g-2$, $M=\cK_X$, in which case 
$\dim H^0(X,M)=m-g+2=g$. This means that the variety $Y$ 
of sections of $L$ with a double zero 
has dimension $d-g$. Indeed, such a section is determined by 
the position $x$ of the double zero on $X$ and a section 
of the line bundle $M\coloneqq L\otimes \cO(-2x)$, 
and $H^0(X,L\otimes \cO(-2x))$ has dimension 
$d-g-1$ unless $L\otimes \cK_X^{-1}=\cO(2x)$. But this 
can only happen for finitely many $x$ and in this case 
$H^0(X,L\otimes \cO(-2x))=d-g$. Thus $Y\ne H^0(X,L)\cong \mathbb C^{d-g+1}$, 
which implies the statement.

(ii) follows from (i) and the fact that the degree of $\cK_X^{\otimes n}$
is $2n(g-1)$, which is $\ge 2g$ for $g\ge 2$.
\end{proof} 

By Lemma~\ref{simzer}(ii), a generic section $s$ of $\cK_X^{\otimes n}$ has only simple zeros. Thus for $b\coloneqq (0,\ldots,0,s)$, the spectral curve $C_b$ given by the equation $y^n=s(x)$ in $T^*X$ 
is smooth and irreducible, as desired. 
\end{proof}

\begin{theorem}[\cite{H}] The genus of the spectral curve for generic $b$ is given by the formula 
  $$
  g(C_b) = n^2(g-1) + 1.
  $$
\end{theorem}

\begin{proof}
  Since $g(C_b)$ is deformation-invariant,  we can compute it at the
  point 
  $$
  b_1 = b_2 = \cdots = b_{n-1} = 0,\ b_n=s,
  $$ 
  where $s$ is a section 
  of $\cK_X^{\otimes n}$ with simple zeros, whose existence is guaranteed by Lemma \ref{simzer}(ii). Then the degree $n$ map $\pi\colon C_b\to X$ 
  is unramified except at the zeros of the section $s$. 
   So $g(C_b)$ may be computed by
  the Riemann--Hurwitz formula. Namely, since
  $s \in H^0(X, \cK_X^{\otimes n})$
  is a section of a degree $2n(g-1)$ bundle, it has 
  $2n(g-1)$ zeros. Thus the Riemann--Hurwitz formula 
  for the Euler characteristic of $C_b$ gives
  $$
    \chi(C_b) = \left(\chi(X) - 2n(g-1)\right) n + 2n(g-1)  = -n^2 (2g - 2).
  $$
  (as $\chi(X) = 2 - 2g$). Thus
  \[ g(C_b) =\frac{2-\chi(C_b)}{2}=  n^2(g-1)+1. \qedhere \]
\end{proof}

\subsection{Proof for Hitchin's theorem, part 2, for \texorpdfstring{$G = \GL_n$}{G = SLn}}\label{st3}
Let us now proceed with part 2 of the proof of Hitchin's theorem. 
If we have a Higgs pair $(E,\phi)$ such that the spectral curve of $C(E,\phi)$ is $C$, then
there is an {\it eigenline bundle} $L_\phi$ on $C$. Namely, the fiber of
$L_\phi$ at a point $\lambda \in C$ lying over $x \in X$ is the
eigenline of $\phi(x)$ with eigenvalue $\lambda$, in the generic
situation where all the eigenvalues are distinct. By analyzing the ramification points, it is easy to show that $L_\phi$ extends naturally to the whole $C$. 

The vector bundle $E$ is then reconstructed from $C$ and $L_\phi$ as the pushforward of $L_\phi$: 
\[ E \cong \pi_* L_\phi. \]
(Algebraically, this means that Zariski locally on $X$ we regard the $\cO_C$-module $L_\phi$ as 
an $\cO_X$-module using the map $\pi^*\colon \cO_X\to \cO_C$.)
This is because the fiber $E_x$ is the direct sum
$\bigoplus_{\lambda \in \pi^{-1}(x)} (L_\phi)_\lambda$ when
all the eigenvalues are distinct. 

Moreover, we can also recover
$\phi$ from $L_\phi$, because the action of $\phi$ on $E$ is just
multiplication by the cotangent coordinate on $L_\phi$.

Theorem~\ref{smirr} says that the spectral curve $C_b$ for generic $b
\in \cB$ is smooth and irreducible. In fact, for such $C_b$ and a
generic line bundle $L$ on $C_b$, the reconstructed vector bundle $E$
is stable. Thus, $C_b$ is smooth and irreducible for generic $b \in
\cB$ in the image of the Hitchin map $p$.

The line bundle $L_\phi$ on $C_b$ has degree $d =
  n(n-1)(1-g) + \deg E$ \cite{H}. So, if $\deg E$ is fixed, then $\deg
  L_\phi$ is constant on the component containing $E$. Then
\[ L_\phi \in \Pic_d(C_b) \cong \Jac(C_b) \]
and so $p^{-1}(b) \subset T^*\Bun_G^\circ(X)$ gets identified with a
subset of $\Jac(C_b)$. Thus for generic $b$, 
\[ \dim p^{-1}(b) \le  n^2(g-1)+1 \]
(the dimension of $\Bun_G^\circ(X)$). So it follows from part 1 
that $p$ is generically a Lagrangian fibration, as desired. 

\begin{remark} We see that for $G=\GL_n$ a generic fiber $p^{-1}(b)$ of the Hitchin map $p$ is an open set in an abelian variety --- the Jacobian of the spectral curve $C_b$. Similarly, for $G=\SL_n$,  $p^{-1}(b)$ is generically an open set in the kernel of $\pi_\bullet\colon \Jac(C_b)\to \Jac(X)$, where $\pi_\bullet$ is the {\it norm map} of Jacobians induced by $\pi$ (this kernel is also an abelian variety but not a Jacobian, in general; it is called the {\it Prym variety} of the map $\pi$). The missing points of these abelian varieties are added back when one extends $p$ to 
the Hitchin moduli space $\mathcal M_G(X)$ (see Remark \ref{hms}), as the extended Hitchin map 
$p: \mathcal M_G(X)\to \mathcal B$ is proper (\cite{H}). This is consistent with the Liouville--Arnold theorem (\cite{A}), which says that fibers of a Lagrangian fibration (i.e., common level sets of an integrable system), when compact and connected, are Lagrangian tori, and the Hamiltonian dynamics is a linear flow on each of them. 
\end{remark}

\section{Classical Hitchin systems for \texorpdfstring{$G$}{G}-bundles with parabolic structures and twisted classical Hitchin systems}
\label{sec:hitchin-with-parabolic-and-twists}

\subsection{Principal bundles with parabolic structures} 

Let $G$ be a connected reductive group. Recall that a closed subgroup
$P \subset G$ is called {\it parabolic} if it contains a Borel
subgroup, or, equivalently, if the quotient $G/P$ is a projective
variety (called the {\it partial flag variety} corresponding to $P$).
For a subdiagram $D$ of the Dynkin diagram $D_G$ of $G$, there is a
parabolic subgroup $P(D)\subset G$, the connected subgroup whose Lie
algebra is generated by the Chevalley generators $e_i,h_i$, $i\in D_G$
and $f_i,i\in D$, and every parabolic subgroup is conjugate to $P(D)$
for a unique $D$. Borel subgroups $B$ are the smallest parabolics ---
those conjugate to $P(\emptyset)$ --- and then $G/P=G/B$ is the (full)
{\it flag variety}. For instance, if $G=\GL_n$, then subdiagrams of
the Dynkin diagram correspond to compositions $n = n_1 + \cdots +
n_r$, and the corresponding parabolic subgroup $P(n_1,\ldots,n_r)$
consists of upper block-diagonal matrices with blocks of size $n_1,
\ldots, n_r$. The partial flag variety $G/P$ is then the variety
$\mathcal F_{n_1,\ldots,n_r}(\mathbf k^n)$ of partial flags $0=V_0\subset
V_1\subset\ldots\subset V_r=\mathbf k^n$ such that $\dim V_i/V_{i-1}=n_i$
for all $i$ (note that this description depends only on $n_1, \ldots,
n_r$, and not the choice of $P$). The smallest parabolics are the
Borel subgroups, conjugate to the group $P(1,\ldots,1)$ of upper
triangular matrices, and then $G/B$ is the variety $\mathcal
F(n)=\mathcal F_{1,\ldots,1}(\mathbf k^n)$ of full flags in $\mathbf k^n$.

Now let $X$ be a smooth irreducible projective curve, $t_1,\ldots,t_N\in X$, and 
$P_1,\ldots,P_N\subset G$ be parabolic subgroups. 
Denote by $\Bun_G(X, t_1, \ldots, t_N, P_1, \ldots, P_N)$ the moduli
stack of $G$-bundles on $X$ with $P_i$-structures at $t_i$ for $i =
 1, \ldots, N$. Points of this stack are often called {\it parabolic bundles}.

It is clear that we have a fibration
\[ \theta\colon \Bun_G(X, t_1, \ldots, t_N, P_1, \ldots, P_N) \to \Bun_G(X) \]
(forgetting the parabolic structures) with fiber $G/P_1 \times \cdots \times G/P_N$.

\begin{example}\label{examPstr}
  Let $\cE$ be a $\GL_n$-bundle on $X$ and $E$ denote the associated
  rank $n$ vector bundle. Canonically, the fiber of $\cE$ at a point
  $x \in X$ is $\cE_x = \{\text{bases in } E_x\}$. In particular, if
  $P=P(n_1,\ldots,n_r)$ then $P$ is the stabilizer of a partial flag
  \[ 0 \subset V_1 \subset \cdots \subset V_r = \mathbf k^n \]
  where the quotients $V_i/V_{i-1}$ are vector spaces of dimensions
  $n_i$, for $i = 1, \ldots, r$. So a $P$-structure on $E$ at $x$ is the set of bases
  compatible with this flag, i.e. such that there is a nested sequence of
  subsets of the basis which are bases of the $V_i$. Thus, choosing a
  $P$-structure on $E$ at $x$ is equivalent to fixing a partial flag in $E_x$
  of type $(n_1,\ldots,n_r)$. 
  
It follows that the fiber of $\theta$ at $E$ is canonically 
  $\theta^{-1}(E)=\prod_{j=1}^N \mathcal F_{n_{j1},\ldots,n_{jr_j}}(E_{t_j})$, 
  where $(n_{j1},\ldots,n_{jr_j})$ is the type of $P_j$. 

  The basic example we will consider is $G = \GL_2$ and $P_i = B$ is the Borel subgroup. A $B$-structure at $t_i$ is then a choice of a line $\ell_i
  \subset E_{t_i}$, so the fiber of $\theta$ over $E$ in this case is 
  $\theta^{-1}(E)=\prod_{j=1}^N\mathbb P E_{t_j}$. 
\end{example}

One reason to consider parabolic bundles is that it gives rise to interesting problems already for $g = 0$ and $g = 1$, which leads to nice explicit formulas. Indeed, without parabolic structures, there are no stable bundles for $g < 2$; all bundles have non-trivial automorphism groups, even if $G$ is adjoint. However in presence of parabolic structures, automorphisms must preserve it, so the automorphism group shrinks. In particular, for a sufficient number of marked points there will be a lot of bundles with trivial automorphism group.

For example, if $N \ge 3$ and $G$ is adjoint, then a generic $G$-bundle on
$\bP^1$ with parabolic structures has trivial automorphism
group. For example, consider $G = \PGL_2$ and let $E$ be the trivial
$G$-bundle on $X$. Then $\Aut(E) = \PGL_2$. But we have parabolic
structures $\ell_1, \ldots, \ell_N$, where $\ell_i \in \bP E_{t_i} =
\bP^1$ for each $i$. So the set $\Bun_G^{\text{triv}}(X, t_1, \ldots,
t_N, P_1, \ldots, P_N)$ of parabolic structures on the trivial bundle is
just $[(\bP^1)^N / \PGL_2]$. This is still stacky, because when the $N$
points $\ell_1,\ldots,\ell_N$ in $\bP^1$ coincide, there is still a non-trivial automorphism
group. But we can consider the smaller open set given by points 
$(\ell_1, \ldots, \ell_N) \in (\bP^1)^N$ such that $\ell_{N-2},\ell_{N-1},\ell_N$ are distinct. It is
well-known that $\PGL_2$ acts simply $3$-transitively on $\bP^1$, i.e. given
three distinct points on $\bP^1$, there is a unique element in
$\PGL_2$ which sends them to $(0, 1, \infty)$. Hence the open set parameterizing $N$-tuples of points with the last three points distinct is 
\[ (\bP^1)^{N-3} \subset [(\bP^1)^N / \PGL_2]. \]
Note that it is a smooth projective variety.

In general,
\[ \Bun_G^{\text{triv}}(X, t_1, \ldots, t_N, P_1, \ldots, P_N) \cong \bigg(\prod_i G/P_i\bigg) / G \]
where the $G$-action is diagonal.

\begin{remark} Like before, the Hitchin system for parabolic bundles 
should be defined on the cotangent bundle of the variety of {\it stable parabolic bundles.} There are many notions of stability for bundles with parabolic structures, depending on parameters called {\it weights} (see \cite{MSe}), but for us it will not matter which notion we use, since it does not matter much on which 
open dense subvariety on the stack of bundles we initially define the system. 
In any case, in the examples we consider, the trivial bundle with generic parabolic structures will be stable. 
\end{remark} 

\subsection{Classical Hitchin systems with parabolic structures}

Recall that we have constructed the Hitchin system by realizing
$\Bun_G(X)$ as a double quotient of the loop group and then descending invariant
functions from the cotangent bundle to 
the loop group to the double quotient. We may do the same when there are
parabolic structures. Namely, recall that
\[ \Bun_G(X) = G(R_{t_1, \ldots, t_N}) \big\backslash \prod_i G(K_{t_i}) \big/ \prod_i G(\cO_{t_i}), \]
where $R_{t_1,\ldots,t_N}\coloneqq\mathbf k[X\setminus \lbrace
  t_1,\ldots,t_N\rbrace]$, and parabolic structures are local at each
of the $t_i$, so we should modify the right quotient. Let $\ev\colon
G(\mathbf k[[t]]) \to G$, $g(t) \mapsto g(0)$ be the evaluation function
at $t=0$, and let $\widetilde P_i \coloneqq \ev^{-1}(P_i)$. In other
words, $\widetilde P_i$ consists of Taylor series whose constant term
lies in $P_i \subset G$. Then
\[ \Bun_G(X, t_1, \ldots, t_N, P_1, \ldots, P_N) = G(R_{t_1, \ldots, t_N}) \big\backslash \prod_i G(K_{t_i}) \big/ \prod_i \widetilde P_i. \]
The discrepancy between $\Bun_G(X, t_1, \ldots, t_N, P_1, \ldots,
P_N)$ and $\Bun_G(X)$ is therefore $\prod_i G/P_i$, exactly as stated earlier.

Now consider the space $T^*(\prod_{j=1}^N G(K_{t_j}))$. 
This is the set of $(g_1,\ldots,g_N,\widetilde \phi_1,\ldots,\widetilde\phi_N)$ 
where $g_j\in G(K_{t_j})$ and $\widetilde \phi_j$ are Higgs fields 
on $D_{t_j}^\times$. So we may take the usual Hamiltonians 
$$
H_{i,j,n}\coloneqq
c_n(Q_i(\widetilde \phi_j))
$$
on this space (using some formal coordinates $z_j$ near $t_j$)
and do the same reduction as before but now with respect to the
subgroup $G(R_{t_1, \ldots, t_N}) \times \prod_i \widetilde
P_i$. The result is an integrable system on $T^*\Bun_G^\circ(X, t_1, \ldots,
t_N, P_1, \ldots, P_N)$. Points of this space are pairs $(E, \phi)$
where $E$ is a bundle with parabolic structures, and $\phi \in
\Omega^1(X \setminus \{t_1, \ldots, t_N\}, \ad E)$ is a Higgs field
      {\it with tame singularities and nilpotent residues}. Namely, it is easy to check that the condition at the singularities is that 
  $\phi$ has at most first-order poles only at the points $t_1,
  \ldots, t_N$, and the residue $\Res_{t_i} \phi$ {\it strictly
    preserves} the flag $\lbrace V_j\rbrace$ specified by the parabolic structure
  at $t_i$.
Here, ``strictly'' means that $\Res_{t_i} \phi$ lies in the nilpotent
radical of the Lie algebra $\mathfrak p_i\coloneqq{\rm Lie}P_i$. For instance, in the $G = \GL_n$ case, this means that the
  residue preserves the flag and acts by $0$ on the associated graded
  of $\mathbf k^n$ under the flag filtration.

\subsection{Garnier system} 
  Let us compute the classical Hitchin system for $\PGL_2$ in genus $g = 0$. For
  this purpose, we will assume for convenience that $t_1, \ldots, t_N
  \in \bA^1 \subset \bP^1$ and that the parabolic structure at $t_j$ is
  given by $y_j \in \bA^1\subset \mathbb P^1=\PGL_2/B$. 
  
  Let $z$ be a coordinate on $\bA^1 \subset \bP^1$. The Higgs field
  $\phi$ is a $1$-form with simple poles at $t_j$, valued in
  $\lie{sl}_2$. So
  \[ \phi = \sum_{j=1}^N \frac{A_j}{z - t_j} \, dz, \qquad A_j=\begin{pmatrix} a_j & b_j \\ c_j & -a_j \end{pmatrix} \in \lie{sl}_2, \]
  satisfying the following conditions. First, $\phi$ must be regular at
  $\infty \in \bP^1$ because there is no marking/puncture there. This is
  the case if and only if 
 \begin{equation}\label{sumai0}
 \sum_{j=1}^N A_j = 0.
 \end{equation} 
 Second, $A_j \begin{pmatrix} y_j \\ 1 \end{pmatrix} = 0$ (in particular, since $A_j$ has trace $0$, it must
  be nilpotent).
  This is the condition
  \[ \begin{pmatrix} a_j & b_j \\ c_j & -a_j \end{pmatrix} \begin{pmatrix} y_j \\ 1 \end{pmatrix} = 0, \]
  which says that $a_j= c_jy_j$ and $b_j = -a_jy_j$, so $b_j=-c_jy_j^2$. Hence
    \begin{equation}\label{Aimat}
  A_j = c_j \begin{pmatrix} y_j & - y_j^2 \\ 1 & -y_j \end{pmatrix}.
  \end{equation}
  One may check that $c_j=p_j$ are the momentum coordinates, i.e., 
  the symplectic form is $\omega=\sum_j dp_j \wedge dy_j$.
  
  Let us now compute the Hitchin Hamiltonians. The generating function for them is  
  \[ H(z) = \frac{1}{2}\tr \phi^2 = \frac{1}{2}\tr\left(\sum_{i,j} \frac{A_i}{z - t_i} \frac{A_j}{z - t_j}\right) \]
  (we drop the factor $(dz)^2$ for brevity). 
  The $i = j$ terms drop out, because $A_i$ is nilpotent and thus $A_i^2
  = 0$. Thus we get
  \[ H(z) = \sum_{i<j} \frac{\tr A_iA_j}{(z - t_i)(z - t_j)}. \]
  Using the identity
  \[ \frac{1}{(z - a)(z - b)} = \frac{1}{a - b} \left(\frac{1}{z - a} - \frac{1}{z - b}\right), \]
  this can be rewritten as
  \[ H(z) = \sum_{i \neq j} \frac{\tr A_iA_j}{(t_i - t_j)(z - t_i)}. \]
  It remains to compute the trace. Using \eqref{Aimat}, we have 
  \[ \tr A_iA_j = p_i p_j \tr \begin{pmatrix} y_i & -y_i^2 \\ 1 & -y_i \end{pmatrix} \begin{pmatrix} y_j & -y_j^2 \\ 1 & -y_j \end{pmatrix} = -p_i p_j (y_i - y_j)^2. \]
  Finally, let's take residues:
  \[ G_i \coloneqq \Res_{t_i} H(z) = \sum_{j \neq i}\frac{p_ip_j (y_i - y_j)^2}{t_j - t_i}. \]
  These functions are called the {\it Garnier Hamiltonians}. By construction, we have 
  $$
  \lbrace G_i,G_j\rbrace=0.
  $$
  on  $T^*\bC^N$, which can also be checked directly.  
  
  Unfortunately, the Garnier Hamiltonians do not quite define an integrable system 
  on $T^*\bC^N$, since they are functionally (in fact, linearly) dependent. 
  Namely, by \eqref{sumai0} and \eqref{Aimat} we have 
  \begin{equation}\label{sums0}
  \sum_i p_i = \sum_i p_i y_i =
  \sum_i p_i y_i^2 = 0,
  \end{equation} 
  so 
  $$
  \sum_i G_i=\sum_i t_iG_i=\sum_i t_i^2G_i=0;
  $$
in fact, there are only $N-3$ independent Hamiltonians among the $G_i$.   
However, by \eqref{sums0}, 
  $(y, p)$ belongs to $\mu^{-1}(0) \subset
  T^*\bC^N$, where $\mu\colon T^*\bC^N\to \mathfrak{sl}_2^*$ 
  is the moment map for the $\PGL_2$-action. 
Thus $\lbrace G_i\rbrace$ define $N-3$ independent Poisson-commuting
Hamiltonians on the Hamiltonian reduction $\mathcal
M_N\coloneqq\mu^{-1}(0)/\PGL_2$, which has dimension $2(N-3)$. Thus we
obtained an integrable system on this variety, called the {\it Garnier
  system}.

\begin{problem}\label{garn}
  Find the spectral curve of the Garnier system and compute its genus. Compute the genus explicitly for $N = 4$. (Hint: let $a, b$ be coprime polynomials of $x$ of
  degrees $n_a$ and $n_b$, and with simple roots. What is the genus of
  the normalization of the affine curve $y^2 = a(z)/b(z)$?)
\end{problem}

\subsection{Twisted classical Hitchin systems}

It turns out that another added benefit of Hitchin systems for bundles with parabolic
structures is that they have a twisted generalization, which allows us to produce
more general integrable systems depending on many parameters. To introduce them, we first need to define
a more general version of Hamiltonian reduction called 
{\it Hamiltonian reduction along an orbit} (see e.g. \cite{E2}, Subsection 1.4). 

Let $M$ be a symplectic variety,
with Hamiltonian action by a group $H$. Let $$\mu\colon M \to
\lie{h}^*$$ be a moment map. Previously, we considered the Hamiltonian reduction $\mu^{-1}(0)/H$,
but more generally, we may consider
\[ \mu^{-1}(O)/H \]
for any {\it coadjoint $H$-orbit} $O \subset \lie{h}^*$. 
If the $H$-action is sufficiently nice, this quotient also has a
canonical symplectic structure, and we can run the same construction
of integrable systems as before: if $F_i$ are $H$-invariant functions
in involution on $M$, then they descend to functions $\overline F_i$ on
$\mu^{-1}(O)/H$, which are also in involution.

In particular, returning to our setting, recall that the group
$G(R_{t_1, \ldots, t_N}) \times \prod_i G(\cO_{t_i})$
acts on $\prod_i G(K_{t_i})$, and let ${\mathbf K}$ denote the
kernel of the evaluation map $\prod_i G(\cO_{t_i}) \to G^N$ at $(t_1,
\ldots, t_N)$. We reduce first by ${\mathbf K}$, after which there is a
residual action of $G^N$, and then for a coadjoint orbit $O \subset
(\lie{g}^*)^N$, we can descend the Hitchin Hamiltonians to
$\mu^{-1}(O)/G^N$. The result is called a {\it twisted Hitchin system}. In particular, the case of parabolic structures arises from specific
choices of $O$ (so-called {\it Richardson nilpotent orbits}).

\subsection{Twisted Garnier system} \label{tgar}
As before, points in $\mu^{-1}(O)/G^N$ are Higgs pairs $(E, \phi)$
where $\phi$ must satisfy some conditions. To illustrate, take $G =
\PGL_2$. At $t_i$, take the coadjoint orbit $O_i$ in $\lie{sl}_2^* \cong
\lie{sl}_2$ of regular elements with eigenvalues $\pm \lambda_i$, and let $O\coloneqq\prod_i O_i$. Then generically $E$ is a trivial bundle with parabolic structures 
$\ell_i\in \mathbb P^1$ at $t_i$, and $\phi$ has
simple poles at each $t_i$ with
\[ \Res_{t_i} \phi|_{\ell_i} = \lambda_i \cdot \id \]
(The usual story with parabolic structures
corresponds to the case $\lambda_i=0$.)

As before, writing the Higgs field as 
$$
\phi = \sum_i \frac{A_i}{z - t_i} \, dz,
$$ 
we get
\[ A_i \begin{pmatrix} y_i \\ 1 \end{pmatrix} = \lambda_i \begin{pmatrix} y_i \\ 1 \end{pmatrix}. \]
Solving this equation, we obtain
\[ A_i = \begin{pmatrix} -\lambda_i + p_iy_i & 2\lambda_i y_i - p_iy_i^2 \\ p_i & \lambda_i - p_iy_i \end{pmatrix}. \]
The trace of $A_iA_j$ then becomes
\[ \tr(A_iA_j) = -(y_i - y_j)^2 p_i p_j + 2(\lambda_i p_j - \lambda_j p_i)(y_i - y_j) + 2\lambda_i\lambda_j. \]
The resulting Hamiltonians
\[ G_i(\lambda_1, \ldots, \lambda_N) = \sum_{j \neq i} \frac{(y_i-y_j)^2 p_ip_j - 2(\lambda_i p_j - \lambda_j p_i)(y_i - y_j) - 2\lambda_i\lambda_j}{t_j - t_i} \]
define the {\it deformed} or {\it twisted} Garnier system. The
ordinary Garnier system is a specialization of this, when $\lambda_i
=0$.

 \subsection{Classical elliptic Calogero-Moser system} \label{ellcm}
Another known integrable system which is a special case of the (twisted) Hitchin system is the {\it elliptic Calogero-Moser system}. Let us explain how it arises in this way. 

  Let $X$ be an elliptic curve with zero denoted $0 \in X$, and
  consider a generic vector bundle of degree $0$ and rank $n$ on $X$. Line bundles
  of degree $0$ on $X$ are all of the form
  \[ L_q = \cO(q) \otimes \cO(0)^{-1} \]
  for a point $q \in X$, with a meromorphic section given by
  $\frac{\theta(z-q)}{\theta(z)}$, where $\theta(z)$ is the Jacobi theta-function of $X$. Atiyah (\cite{At}) showed that a generic rank $n$
vector bundle of degree zero on $X$ has the form
  \[ E = L_{q_1} \oplus \cdots \oplus L_{q_n}, \]
  say with $q_i \neq q_j$. Consider $G = \GL_n$, put one puncture at
  $0$, and perform the twisted reduction procedure for the orbit
  $O_c=\langle cT\rangle\subset \lie{sl}_n^*\cong \lie{sl}_n$ at this puncture, where
  $$
  T\coloneqq 1-nE_{nn}=
  \diag(1, \ldots, 1, 1-n) . 
  $$
  As $c \to 0$, this orbit degenerates into the closure of the orbit $O_0$ 
  of a rank $1$ nilpotent
  matrix, corresponding to the parabolic subgroup $P$ with blocks of
  size $(n-1) \times (n-1)$ and $1 \times 1$, i.e. $G/P = \bP^{n-1}$.
 
 So we will consider the twisted Hitchin system for $G=\PGL_n$ 
 with a $P$-structure at $0$. Note that the group $\Aut(E) = (\bC^\times)^{n-1}$ acts on $\bP^{n-1}$ with an open orbit generated by the vector $(1, 1, \ldots, 1)$, so generically we may assume that the $P$-structure at $0$ is defined by this vector. 
  
The
  components $\phi_{ij}$ of the Higgs field $\phi$ are sections of $L_{q_i} \otimes
  L_{q_j}^{-1}$ with at most a first-order pole at $z=0$. Hence, for $i \neq j$ 
  \[ \phi_{ij} = a_{ij} \frac{\theta'(0)\theta(z + q_j - q_i)}{\theta(z) \theta(q_j - q_i)}, \]
  $a_{ij}\in \mathbb C$, and $\phi_{ii} = p_i$ are the momenta, i.e. the symplectic form is 
  $\omega=\sum_j dp_j\wedge dq_j$. 
  
 What is the condition for the matrix $A = (a_{ij})$? Since we are considering the twisted Hitchin system corresponding to the orbit $O_c$, $A$  should act on the vector $(1, \ldots, 1)$ with eigenvalue
  $(1-n)c$:
  \[ A \begin{pmatrix} 1 \\ 1 \\ \vdots \\ 1 \end{pmatrix} = (1-n) c \begin{pmatrix} 1 \\ 1 \\ \vdots \\ 1 \end{pmatrix} \]
  and be conjugate to $cT$. This means that all
  off-diagonal entries $a_{ij}$ are equal to some constant $c$, so
  \[ \phi_{ij} = c\frac{\theta'(0)\theta(z + q_j - q_i)}{\theta(z) \theta(q_j - q_i)}, \ i\ne j;\ \phi_{ii}=p_i. \]
  This $\phi$ is (up to conjugation) {\it Krichever's Lax matrix} for the {\it elliptic
    Calogero--Moser (CM) system}, \cite{KrCM}. 
    
  Thus the generating function for the quadratic Hitchin Hamiltonians is  
    \[ \tr \phi^2 = \sum_i p_i^2 - c^2\sum_{j \neq i} \frac{\theta'(0)^2\theta(z - q_i + q_j) \theta(z - q_j + q_i)}{\theta(z)^2 \theta(q_i - q_j)^2}. \]
  The second summand on the right hand side is $c^2(\wp(z) - \wp(q_i -
  q_j))$, where $\wp(z)$ is the Weierstrass function, since both
  functions have the same zeros and poles and the same coefficient of $\frac{1}{z^2}$. 
  Thus the constant term of $\tr\phi^2$ has the form
  \begin{equation}\label{h2cm}
    \begin{aligned}
      H_2\coloneqq \text{constant term of }\tr\phi^2
      &=\tr\phi^2-n(n-1)c^2\wp(z) \\
      &=\sum_i p_i^2 - c^2 \sum_{i \neq j} \wp(q_i - q_j).
    \end{aligned}
   \end{equation}
  This is exactly the {\it quadratic Hamiltonian in the elliptic CM system.}\footnote{Note that $c$ is not an essential parameter here since it can be rescaled by renormalizing time. Thus physically there are just two distinct cases, corresponding to positive or negative $c^2$ (attracting vs. repelling potential).}
   The constant terms of the traces $\tr \phi^k$, $1\le k\le n$ then yield the first integrals of the elliptic CM system. They have the form 
  $$
  H_k=\sum_i p_i^k+\text{lower order terms with respect to $p_i$};
  $$ 
  for example, $H_1=\sum_i p_i$ and $H_2$ is given by \eqref{h2cm}. 
  Thus $H_k$ are independent (as so are their leading terms, the power sum functions $\sum_i p_i^k$), which 
  establishes complete integrability of the elliptic CM system.

\begin{exercise} 
  Show that for $n = 4$ the Garnier system is equivalent to the
  elliptic CM flow for $2$ particles. What is a geometric reason for
  it? Solve the Hamilton equation of the flow.
\end{exercise}

\begin{exercise}\label{cubic}
  Calculate the first integral $H_3$ for the elliptic CM system for $c=1$.
Namely, show that it has the form 
$$
H_3=\sum_{i=1}^{n} p_i^3+\sum_{i=1}^{n} a_i(q_1,\ldots,q_{n})p_i+b(q_1,\ldots,q_{n})
$$  
and compute the functions $a_i$ and $b$. 
\end{exercise}

\begin{exercise}
  Let $L = \partial_z^2 - 2\sum_{i=1}^{n} \wp(z-q_i)$ and
    assume the $q_1, \ldots, q_n$ are distinct.
  \begin{enumerate}[label=(\roman*)]
  \item Show that $L$ commutes with an odd-order
    differential operator if and only if $(q_1, \ldots, q_n)$ is a
    critical point of the Calogero--Moser potential $\sum_{1\le i<j \le n}
    \wp(q_i - q_j)$.
  \item Show that this differential operator can be
    chosen to be third-order if and only if, in addition, $\sum_{j\neq i}
    \wp(q_i - q_j)$ is independent of $i$. In this case, provide a
    complete classification of solutions using the ODE for the
    potential.
  \end{enumerate}
  A good reference for this exercise is \cite{GW1}.
\end{exercise}

\section{Quantization of Hitchin systems} 

\subsection{Quantization of symplectic manifolds}

What is a quantum integrable system, and what does it mean to quantize
a classical integrable system? This is not even the most basic
question, since classical integrable systems live on symplectic
manifolds, so we should first say what it means to quantize a symplectic
manifold. This story can be told for smooth manifolds, complex analytic
manifolds, or algebraic varieties over any field.

Let $M$ be a symplectic variety. We will pretend that $M$ is affine
so that there is no need to think about sheaves. Then the algebra $\cO(M)$ of regular functions on $M$ is a {\it Poisson algebra}, meaning that there is a Lie bracket 
$\{-, -\}$ on $\cO(M)$ which is a derivation in the first (hence second) argument. 
Such Lie bracket is called a {\it Poisson bracket}. In classical
mechanics, $M$ is the phase space, and $\cO(M)$ is the algebra of
observables. Quantization means that observables are replaced by
operators which may no longer commute. More precisely, we make the following definition. 

\begin{definition}
  A {\it quantization} of $\cO(M)$ over $\mathbf k[\hbar]$ or $\mathbf k[[\hbar]]$ 
  is an associative algebra $(A,*)$ over this ring
   equipped with an algebra isomorphism 
  $\kappa\colon A/\hbar A \cong \cO(M)$ so that $A$ is a flat deformation of
  $A/\hbar A$ and 
  \[ \lim_{\hbar \to 0} \frac{f * g - g * f}{\hbar} = \{f, g\}. \]
\end{definition}

Here ``flat deformation'' means that $A\cong \cO(M)[\hbar]$ as a 
$\mathbf k[\hbar]$-module (respectively  $A\cong \cO(M)[[\hbar]]$ as a 
$\mathbf k[[\hbar]]$-module) compatibly with $\kappa$. 

Recall that if $A=\cup_{i\ge 0}F_iA$ is a $\mathbb Z_+$-filtered $\mathbf k$-algebra 
then the associated graded algebra $\gr(A)$ is 
$$
\gr(A)=\bigoplus_{i\ge 0} \gr(A)_i,\quad \gr(A)_i\coloneqq F_iA/F_{i-1}A.
$$
For $a\in F_iA$, let $p_i(a)$ be its image in $\gr(A)_i$. 
If $\gr(A)$ is commutative then it carries a Poisson bracket of degree $-1$ 
such that if $a\in F_iA$ and $b\in F_jA$ then 
$$
\lbrace p_i(a),p_j(b)\rbrace=p_{i+j-1}([a,b]).
$$
In this case $A$ is said to be a {\it filtered quantization} of 
$(\gr(A),\lbrace,\rbrace)$. 

Recall that the {\it Rees algebra} of $A$ is 
$$
{\rm Rees}(A)=\sum_{i\ge 0}\hbar^i F_i(A)\subset A[\hbar].
$$
It is easy to see that ${\rm Rees}(A)$ is a free $\mathbf k[\hbar]$-module, 
${\rm Rees}(A)/\hbar {\rm Rees}(A)=\gr(A)$, ${\rm Rees}(A)/(\hbar-\hbar_0){\rm Rees}(A)\cong A$ for any $\hbar_0\in \mathbf k^\times$, and that if 
$\gr(A)$ is commutative then ${\rm Rees}(A)$ is a quantization of 
$(\gr(A),\ \lbrace, \rbrace)$ over $\mathbf k[\hbar]$. 

\begin{example}
  Suppose $M = T^*Y$ for a smooth affine variety $Y$ over $\mathbf k$ (${\rm char}(\mathbf k)=0$).  Let $\cD(Y)$   be the algebra of differential operators on $Y$, i.e., operators on $\cO(Y)$ which in local coordinates look like   
  \begin{equation}\label{dop}
  D=\sum_\alpha f_\alpha(x)\partial^\alpha,
  \end{equation} where $\alpha=(\alpha_1,\ldots,\alpha_n)$, $\alpha_i\in \mathbb N$ is a multi-index and 
  $\partial^\alpha\coloneqq\prod_i \partial_i^{\alpha_i}$. We have a filtration $\cD(Y)=\cup_{N\ge 0}\cD_{\le N}(Y)$, where $\cD_{\le N}(Y)$ is the space of operators of order $\le N$, i.e., those given by 
 \eqref{dop} such that all terms have $\sum_i \alpha_i\le N$. 
  Then $\cD(Y)$ is a filtered quantization of $\cO(T^*Y)$ with filtration by order, so 
  the algebra of {\it quasiclassical differential operators}
  $$
  \cD_\hbar(Y)\coloneqq{\rm Rees}(\cD(Y))
  $$ 
  is a quantization of $T^*Y$ over $\mathbf k[\hbar]$. 
  
  It is easy to see that 
  $\cD_\hbar(Y)$ is generated  
 over $\mathbf k[\hbar]$ by $\cO(Y)$ and elements $\hbar v$ where $v$ is a vector field on $Y$. For instance, if $Y=\mathbf k^n$ then 
 $\cD_\hbar(Y)$ is generated over $\mathbf k[\hbar]$ by $x_i$ and $\widehat p_i \coloneqq \hbar \di_i$, with defining relations the {\it Heisenberg uncertainty relations}
  \[ [\widehat p_i, x_j] = \hbar \delta_{ij}, \quad [\widehat p_i, \widehat p_j] = [x_i, x_j] = 0.\]
\end{example} 

\begin{exercise} \label{doex}
  \begin{enumerate}[label=(\roman*)]
  \item (Grothendieck's definition of differential operators) Show
    that an operator $D\colon \cO(Y)\to \cO(Y)$ is a differential
    operator of order $\le N$ if and only if for any
    $f_1,\ldots,f_{N+1}\in \cO(Y)$ one has
    $$
    \ad(f_1)\cdots\ad(f_{N+1})(D)=0,
    $$
    where $f_i$ denotes the operator of multiplying by $f_i$.
  \item Let $Y$ be a smooth affine algebraic variety over $\mathbb C$.
    Show that the algebra of differential operators $\cD(Y)$ on $Y$ is
    generated by $\cO(Y)$ and elements $\nabla_v$ attached
    $\cO(Y)$-linearly to vector fields $v$ on $Y$ with defining
    relations
    $$
    [\nabla_v,f]=v(f),\, f\in \cO(Y); \quad [\nabla_v,\nabla_u]=\nabla_{[v,u]},
    $$ and the filtration by order is defined by setting $\deg
    \cO(Y)=0$ and $\deg(\nabla_v)=1$.
  \end{enumerate}
\end{exercise} 

\subsection{Quantization of classical integrable systems}
Let $n \coloneqq \dim Y$. Recall that a classical
integrable system on $T^*Y$ consists of functionally (or, in the algebraic case,
algebraically) independent functions $H_1, \ldots, H_n$ on a dense open
subset of $T^*Y$ such
that $\{H_i, H_j\} = 0$. These functions define a map
\[ p\colon T^*Y \to \bA^n \]
whose pullback $p^*\colon \cO(\bA^n) = \mathbf k[T_1, \ldots, T_n] \to
\cO(T^*Y)$ sends $T_i$ to $H_i$ and is therefore an inclusion of a 
Poisson-commutative subalgebra $\mathbf k[H_1, \ldots, H_n] \hookrightarrow T^*Y$.
Then it is easy to prove the following lemma. 

\begin{lemma}\label{claint} Any regular function on 
a dense open set of $T^*Y$ which Poisson-commutes with all $H_i$ is
 functionally (respectively, algebraically) dependent on them.
 \end{lemma} 
 
Thus if $Y$ is a smooth algebraic variety, then classical integrable
systems on $T^*Y$ correspond to Poisson-commutative polynomial subalgebras in
$\cO(T^*Y)$ with $\dim Y$ generators. This motivates the
following definition.

\begin{definition}
  A {\it quantization} of a classical integrable system $\lbrace H_1,\ldots,H_n\rbrace $ on $Y$ (also called a {\it quasiclassical quantum integrable system} quantizing $\lbrace H_1,\ldots,H_n\rbrace $) is a
  non-commutative algebra $A$ quantizing $\cO(T^*Y)$, with an
  injection
  \[ \mathbf k[T_1, \ldots, T_n] \hookrightarrow A, \qquad T_i \mapsto \widehat H_i \]
  such that $\widehat H_i$ maps to $H_i$ in $A/\hbar A=\cO(T^*Y)$. 
\end{definition}

In particular, if $A=\cD_\hbar(Y)$, then we can specialize $\hbar$ to a nonzero numerical value and obtain a commuting system of operators $\widehat H_1,\ldots,\widehat H_n$ in $\cD(Y)$. In general, such a system is called 
a {\it quantum integrable system} in $\cD(Y)$, provided that these operators 
are algebraically independent. We thus see that 
any quantization of a classical integrable system on $T^*Y$ in $\cD_\hbar(Y)$ 
specializes to a family of quantum integrable systems in $\cD(Y)$ 
parametrized by (generic) $\hbar\in \mathbf k^\times$. 

\begin{example} Let $Y=\mathbb A^1$, and $H_1=H=p^2+U(x)$. 
This system can be quantized by setting $\widehat H=\hbar^2 \partial^2+U(x)$. 
\end{example} 

\begin{remark}\label{qis} Note that a {\it single} quantum integrable system in $\cD(Y)$ need not have a classical limit, i.e., need not be a specialization of a quasiclassical integrable system (=a member of a 1-parameter family of integrable systems as above). A vivid example of this is the deformed rational Calogero-Moser system, see e.g. \cite{FV}. 
\end{remark} 

The following non-trivial theorem is a quantum analog of Lemma \ref{claint}. 

\begin{theorem}[Makar--Limanov, \cite{ML}]
Let $B_1,\ldots,B_n$ be a commuting algebraically independent family of elements 
of $\cD(Y)$. If $B\in \cD(Y)$ and $[B, B_i] = 0$ for all $i$ then $B$ is algebraically
  dependent on $B_1, \ldots, B_n$.
\end{theorem}

Thus a quantum integrable system is a maximal (up to algebraic extensions) commutative subalgebra in $\cD(Y)$, and a quasiclassical quantum integrable system is one in $\cD_\hbar(Y)$. This quasiclassical system quantizes a given classical system if it converges
to it when $\hbar \to 0$ and $\widehat p_i \mapsto p_i$.

This gives rise to a naive quantization procedure: just replace all
instances of $p_j$ with $\hbar \di_j$.\footnote{In actual quantum mechanics one replaces $p_j$ with $-i\hbar \partial_j$, but since our considerations 
are purely algebraic, we drop the factor $-i$.} But this is not a good thing to do
in general due to ordering issues: unlike $p_i$, the partial
derivatives $\di_i$ do not commute with coordinates, so there is
ambiguity as to whether, say, $x_i p_i$ should be replaced by $\hbar
x_i \di_i$ or $\hbar \di_i x_i = \hbar x_i \di_i + \hbar$. However this
naive procedure does sometimes work, e.g. for the twisted Garnier system from
Subsection \ref{tgar}.

\begin{example}
  Recall the twisted Garnier system, given by the Hamiltonians
  \[ G_i = \sum_{j \neq i} \frac{(x_i - x_j)^2  p_i p_j - 2(x_i - x_j)(\lambda_i p_j - \lambda_j p_i) - 2\lambda_i\lambda_j}{t_j - t_i} \]
  where $x_i$ and $p_j$ are the standard coordinates and momenta. The
  naive quantization procedure produces
  \[\widehat G_i \coloneqq \hbar^2\sum_{j \neq i} \frac{(x_i - x_j)^2 \di_i \di_j - 2(x_i - x_j)(\frac{\lambda_i}{\hbar} \di_j - \frac{\lambda_j}{\hbar} \di_i) - 2\frac{\lambda_i}{\hbar} \frac{\lambda_j}{\hbar}}{t_j - t_i}. \]
  It is convenient to introduce $\Lambda_i \coloneqq
  2\lambda_i/\hbar$, so that this can be rewritten as
  \[ \widehat G_i=\hbar^2\sum_{j \neq i} \frac{(x_i - x_j)^2 \di_i \di_j - (x_i - x_j)(\Lambda_i \di_j - \Lambda_j  \di_i) - \frac{1}{2}\Lambda_i\Lambda_j}{t_j - t_i}. \]
\end{example}
  
  There is actually a more insightful way to write
  this formula using Lie theory, which in particular allows us to see without computation that 
  the operators $\widehat G_i$ commute. Namely, let $\lie{sl}_2 = \inner{e, f, h}$, and recall that there is
  an action of $U(\lie{sl}_2)$ by differential operators on $\bA^1$
  given by
 \begin{equation}\label{twidi}
 f \mapsto -\di_x, \qquad h \mapsto 2x \di_x + \Lambda, \qquad e \mapsto x^2 \di_x + \Lambda x. 
 \end{equation}
  The Casimir tensor is the unique element $\Omega \in (\lie{sl}_2
  \otimes \lie{sl}_2)^{\lie{sl}_2}$ up to scaling, and is given by the
  formula
  \[ \Omega = e \otimes f + f \otimes e + \frac{1}{2} h \otimes h. \]
  Then we have
  \[ \widehat G_i = \hbar^{2}\sum_{j \neq i} \frac{\Omega_{i,j}}{t_i - t_j}. \]
  Now note that this formula in fact makes sense for any 
  simple finite dimensional Lie algebra $\lie{g}$, with 
  $\Omega$ being the Casimir tensor (unique up to scaling nonzero element in 
  $(S^2\lie{g})^{\lie{g}}$. Moreover, it is obvious that 
  $$[\Omega_{12}, \Omega_{13}+\Omega_{23}] = 0,
  $$
  and therefore 
  $$
  [\widehat G_i, \widehat G_j] = 0.
  $$
  Thus we obtain a collection of commuting elements in $U(\lie{g})^{\otimes n}$
  called the {\it Gaudin Hamiltonians} for the Lie algebra $\lie{g}$.\footnote{Usually the factor $\hbar^2$ is dropped.} It is easy to check that these quantize second-order Hitchin Hamiltonians for $\lie{g}$ on $\mathbb P^1$ with parabolic structures.
  
  If we pick representations $V_1, \ldots, V_n$ of $\lie{g}$, then we
  get commuting operators
  \[ \widehat G_i \in \End(V_1 \otimes \cdots \otimes V_n) \]
  which also commute with $\lie{g}$ and therefore act on $(V_1 \otimes
  \cdots \otimes V_n)^{\lie{g}}$ and more generally on $\Hom_{\lie{g}}(V,V_1 \otimes \cdots \otimes V_n)$ for any representation $V$ of $\lie{g}$. This produces
 many interesting families of commuting operators.

  However, for simple Lie algebras $\lie{g}$ of rank $>1$, this does not immediately
  produce an integrable system, because we are missing higher-order
  operators; for example, for $\lie{g}=\mathfrak{sl}_n$ with $n>2$ 
  we are missing the components of $\tr \wedge^j\phi$ with $j>2$. We might 
  hope that our naive quantization procedure could help here, but 
  the following example shows that unfortunately this is not quite the case. 

\begin{example}[Elliptic Calogero--Moser system]
  Recall that the classical elliptic Calogero--Moser system has
  Hamiltonian $H_2 = \sum_i p_i^2 - \sum_{j \neq i} \wp(q_i - q_j)$.
  So naively the quantized Hamiltonian should be
  \[ \widehat H_2 \coloneqq \hbar^2\left(\sum_i \di_i^2 - \frac{1}{\hbar^2} \sum_{j \neq i} \wp(q_i - q_j)\right). \]

\begin{theorem}\label{quantCM} $\widehat H_2$ defines a quasiclassical quantum integrable
  system $\bC[\widehat H_1, \ldots, \widehat H_n]$, which is the
  centralizer of $\widehat H_2$ in the algebra of quasiclassical differential operators
  in $n$ variables.\footnote{The trigonometric or rational quantum Calogero--Moser systems
  can be obtained as limits from the elliptic one.}
\end{theorem}  

While this is encouraging, Theorem \ref{quantCM} (integrability of the quantum CM system) unfortunately does not follow from the integrability of the classical CM system. For instance, according to Exercise \ref{cubic}, we have a classical integral
  \[ H_3 = \sum_i p_i^3 + \sum_i a_i(\mathbf q)p_i + b(\mathbf q) \]
  for some functions $a_i(\mathbf q)$ and $b(\mathbf q)$. And already in the middle term
  we have the ambiguity in ordering: in the quantized Hamiltonian
  $\widehat H_3$, do we put $\di_i  a_i(q)$ or $a_i(q) \di_i$, or something
  else? Clearly, we need a more systematic approach!
\end{example}

In fact there is no uniform explicit\footnote{Abstractly, one can find quantizations by deformation-theoretic methods, in both the analytic setting \cite{BD:defquant} and the algebraic setting \cite{Ar}, but these methods do not provide explicit formulas.} way to quantize an arbitrary integrable
system. In each specific case one usually needs to go back to the definition of the classical system
and see if one can quantize the way in which it is obtained. This is exactly
what we'll do in the case of Hitchin systems.

\begin{exercise}
  \begin{enumerate}[label=(\roman*)]
  \item For $N = 4$ the Gaudin system for $\lie{sl}_2$ reduces to a
    second order differential operator $L$ in $1$ variable with $4$
    singularities. Compute this operator after sending $(t_1, t_2,
    t_3, t_4) \mapsto (0, 1, \infty, t)$. (Hint: you can get the
    general shape of $L$ by using that it has $4$ regular
    singularities.)

  \item Show that for $\Lambda_i = -1$ one obtains (up to adding a
    constant) the {\it Lam\'e operator} (with parameter
    $-\frac{1}{2}$)
    \[ L = \di \circ x (x -1)(x - t) \circ \di + x. \]
    
  \item Let $E$ be the elliptic curve $y^2 = x(x-1)(x-t)$. Then the
    function $x$ defines a double cover $E\to \mathbb P^1$ branched at
    $(0,1,\infty,t)$. Let us lift the operator $L$ from (i) to this
    double cover. Show that the lift $\widetilde L$ is the {\it
      Darboux operator}
    \[ \widetilde L = \di_z^2 - \sum_{i=0}^3 \frac{\Lambda_i(\Lambda_i+2)}{4} \wp(z + \varepsilon_i; \tau) \]
    for $\varepsilon_1 = 0$, $\varepsilon_2 = \frac{1}{2}$,
    $\varepsilon_3 = \frac{\tau}{2}$, and $\varepsilon_4 =
    \frac{1+\tau}{2}$.
  \end{enumerate}
\end{exercise}

\subsection{Quantum Hamiltonian reduction}\label{qhr}

Recall our construction of the classical Hitchin system on $\Bun_G^\circ(X)$.
There were two steps.
\begin{enumerate}
\item Represent $\Bun_G(X)$ as a double quotient, e.g. $G(R) \backslash G(K) / G(\cO)$.
\item Construct some commuting Hamiltonians on $T^*G(K)$ which are
  invariant under the left and right actions of $G(K)$, and then
  descend them to $\Bun_G^\circ(X)$ using Hamiltonian reduction by $G(R) \times G(\cO)$.
\end{enumerate}
To retrace our steps, we must discuss a quantized version of
Hamiltonian reduction (along a choice of coadjoint orbit).

Classically, a group $H$ acts in a Hamiltonian manner on a symplectic
manifold $M$, with moment map $\mu\colon M \to \lie{h}^*$, and we
define the Hamiltonian reduction $\mu^{-1}(0)/H$ which is a symplectic
manifold. Note that $\mu$ can be viewed as a Poisson homomorphism
between Poisson algebras $\mu\colon S(\lie{h}) = \cO(\lie{h}^*) \to
\cO(M)$. In the quantum setting, we therefore must consider a group
$H$ acting on an algebra $A$, and the natural way to quantize $\mu$ is
to ask for an algebra homomorphism
\[ \mu\colon U(\lie{h}) \to A. \]
This is the input data that one must supply. The classical condition
that the moment map is $H$-equivariant and dual to the action map
becomes the condition that $\mu$ is $H$-equivariant and
\[ z \cdot a = [\mu(z), a], \qquad \forall z \in \lie{h}. \]
Finally, classically, we considered the quotient $M/H$, for which
$$
\cO(M/H) = \cO(M)^H \subset \cO(M)
$$ 
is a Poisson subalgebra, so a natural way to quantize the locus $\mu^{-1}(0)/H
\subset M/H$ cut out by the equation $\mu(m) = 0$ is given by the following definition.

\begin{definition}
  The {\it quantum Hamiltonian reduction} of $A$ by $H$ is the algebra
  \[ \overline{A}\coloneqq A^H / (A \mu(\lie{h}))^H. \]
  Note that $A \mu(\lie{h}) \subset A$ is only a left ideal, but one
  can check that, after taking $H$-invariants, $(A \mu(\lie{h}))^H
  \subset A^H$ is a two-sided ideal.
  \end{definition}

  Note that if $H$ is reductive and acts on $A$ locally finitely, then the operation of quotienting by $A
  \mu(\lie{h})$ commutes with the operation of taking $H$-invariants, 
  so 
  \[ \overline{A}=(A/ A \mu(\lie{h}))^H. 
  \]

  To replace $0$ with a coadjoint orbit $O \subset \lie{h}^*$, the
  equation $\mu(m) = 0$ is replaced by $\mu(m) \in O$. Thus in the
  quantum setting, we need to find a two-sided ideal $I \subset
  U(\lie{h})$ which quantizes the orbit $O$, in the sense that
  $U(\lie{h})/I$ is a quantization of $O$. Then the quantum
  Hamiltonian reduction is
  \[ \overline{A}=A^H / (A \mu(I))^H, \]
  which coincides with $(A / A \mu(I))^H$ in the reductive case. 
  
  When $O = 0$, the ideal $I$ is the augmentation ideal, i.e. the
  kernel of $U(\lie{h}) \to \bC$, so that $A \mu(I) = A \mu(\lie{h})$
  and we recover the previous case. 
  
  More generally, given a parabolic 
  subalgebra $\lie{p}\subset \lie{g}$ and a character $\lambda\colon \lie{p}\to \mathbb C$, 
  we have the two-sided ideal $I_\lambda\subset U(\lie{g})$ which is the annihilator of the parabolic Verma module $U(\lie{g})\otimes_{U(\lie{p})}\lambda$. The quantum Hamiltonian reduction using this ideal corresponds to the classical Hamiltonian reduction along the orbit $O_\lambda$ of $\lambda$ extended arbitrarily to an element of $\lie{g}^*$ (note that all such extensions are conjugate), which occurs in the construction of twisted Hitchin systems. 
 
\subsection{The quantum anomaly}\label{quanosec}
Let us now implement this for the Hitchin system (first without parabolic structures). We
should take the (huge!) algebra $A \coloneqq \cD(G(K))$ of
differential operators on $G(K)$, and the quantum Hitchin system should be
obtained from some $2$-sided-invariant differential operators.

This is an infinite-dimensional group and there is actually quite a
bit of technical trouble in talking about differential operators on
such a group. But here we will ignore such issues and pretend that $L \coloneqq G(K)$ is an ordinary Lie group; a fully rigorous treatment can be found in \cite{BD}.

What are $2$-sided-invariant differential operators on a Lie group $L$?
Left-invariant differential operators are well-known to be identified
with the universal enveloping algebra $U(\lie{l})$ of the Lie algebra $\lie{l}$
of $L$. Therefore $2$-sided-invariant differential operators are identified with
\[ U(\lie{l})^L = Z(U(\lie{l})), \]
the center of $U(\lie{l})$. 

And here we are in for an unpleasant surprise. In our case, $\lie{l} = \lie{g}((t))$, but
for $G$ semisimple, unfortunately the center of $U(\lie{l})$ is
trivial. (Technically, since $\lie{l}$ is infinite-dimensional, one
should take a suitable completion of the universal enveloping algebra, but,
even so, the center is still trivial.) Thus our first attempt to quantize the Hitchin system fails. 

We can explain what went wrong. 
Classically, recall that $H_2 = \frac{1}{2} (\phi,\phi)$, and the Higgs field $\phi$ has the form
$\phi(z) \, \frac{dz}{z}$ with $\phi \in \lie{g}((t))$. Let 
$$
\phi
= \sum_{n\in \mathbb Z} \phi_n z^{-n}.
$$ 
Also pick an orthonormal basis $\lbrace a_i\rbrace$ of
$\lie{g}$ and write 
$$
\phi_m \eqqcolon \sum_i \phi_m^i a_i.
$$
 Then we have (dropping $(dz)^2$ for brevity)
\begin{equation}\label{h2for}
  H_2 = \frac{1}{2} \sum_n z^{-n-2} \sum_m (\phi_m, \phi_{n-m}) 
  =  \sum_n T_nz^{-n-2},
\end{equation}
where
\begin{equation}\label{tnfor}
T_n\coloneqq H_{2,-n-2}=\frac{1}{2}\sum_{m,i} \phi_m^i \phi_{n-m}^i
\end{equation} 

We now want to generalize these formulas to the quantum case, replacing 
two-sided invariant functions on $T^*G((t))$ by two-sided invariant differential operators 
on $G((t))$ so that $[\phi_\ell^j, H_2] = 0$ for all $\ell$, i.e., 
\begin{equation}\label{comre}
[\phi_\ell^j,T_n]=0
\end{equation} 
for all $\ell,n$, where
$$
[\phi_m^i,\phi_\ell^j]=\sum_r C_{ij}^r \phi_{m+\ell}^r
$$
and $C_{ij}^r$ are the structure constants of the commutator in $\mathfrak{g}$
in the basis $\lbrace a_i\rbrace$.

To make this precise, consider a completion $\widehat U(\lie{g}((t)))$ of the universal enveloping algebra $U(\lie{g}((t)))$ defined as follows. Let $U_n(\lie{g}((t)))$ be the subspace of $U(\lie{g}((t)))$
  consisting of elements of PBW degree $\le n$, and let $\widehat
  U_n(\lie{g}((t)))$ be its completion in the topology whose base of neighborhoods of 
  $0$ consists of the subspaces $U_{n-1}(\lie{g}((t)))\cdot t^m\lie{g}[[t]]$, 
  $m\ge 0$. Now let $\widehat U(\lie{g}((t))) \coloneqq 
  \bigcup_n \widehat U_n(\lie{g}((t)))$. Then by left-invariant differential operators 
  on $G((t))$ we mean elements of $\widehat U(\lie{g}((t)))$, and by two-sided invariant 
  differential operators on $G((t))$ -- elements of the center of this algebra. 
  
Now we can try to define the quantum operators $T_n\in \widehat U(\lie{g}((t)))$ by the same formula 
\eqref{tnfor} (it is easy to see that this series converges in $\widehat U(\lie{g}((t)))$ since $[\phi_m^i,\phi_{n-m}^i]=0$). But then the commutation relation \eqref{comre} does not hold, as seen from the following exercise (physicists call this phenomenon {\it quantum anomaly}). 

\begin{exercise}\label{quano}
Show that if the inner product $(,)$ in the definition of $H_2$ is normalized so that long roots of $\lie{g}$ have squared length $2$ then 
$$
[\phi_\ell^j,T_{n}]= \ell h^\vee\phi_{\ell+n}^j,
$$
where $h^\vee$ is the dual Coxeter number of $\lie{g}$, defined by the formula 
$$
\tr(\ad(a)^2)=2h^\vee (a,a),\ a\in \lie{g}.
$$
\end{exercise}

We might hope that this can be remedied by modifying the formula for $T_n$, 
but this is, unfortunately, a dead end -- one can show that the center of 
the algebra $\widehat U(\lie{g}((t)))$ is trivial. 

\subsection{Twisted differential operators} 

In fact, we were doomed to fail from the start, since Beilinson and Drinfeld showed
(\cite{BD}) that every globally defined differential operator on $\Bun_G(X)$ is a
scalar. However, we can recall something from physics to save the day:
differential operators on a manifold are {\it not} the most natural
quantization of functions on its cotangent bundle. Recall from quantum
mechanics that classical observables on $M = T^*Y$ should quantize to
operators on $L^2(Y)$. But to define $L^2(Y)$, one must fix a measure
on $Y$, and there is no natural choice for the measure in general. The
solution is to take $L^2(Y, \Omega^{1/2})$, where $\Omega$ is the
bundle of densities, and the $L^2$-norm is given by
\[ \|f(y) |dy|^{1/2}\|^2 \coloneqq \int_Y |f(y)|^2 \, |dy|. \]
Thus, from this point of view, the most natural quantization of
functions on the cotangent bundle is the algebra of {\it twisted}
differential operators $\cD(Y, \cK_Y^{1/2})$ acting on a square root
of the canonical bundle on $Y$. Remarkably, this algebra is
independent of the choice of this square root, and moreover makes
sense even when such a square root does not exist at all.

\begin{definition}
  Let $Y$ be a smooth variety and $L$ be a line bundle on $Y$. An {\it
    $L$-twisted differential operator} on $Y$ is a differential
  operator acting on local sections of $L$. 
\end{definition}

Let $\cD(Y, L^{\otimes n})$ be
  the algebra of $L^{\otimes n}$-twisted differential operators. Similarly to Exercise \ref{doex}(ii), it is
  generated by $\cO(Y)$ and elements $\nabla_v$ attached $\cO(Y)$-linearly to vector fields $v$ on $Y$, with relations that deform the relations in Exercise \ref{doex}(ii). 
Namely, if one locally picks a connection on $L$
 with curvature denoted $\omega$, the only relation that changes is
  $[\nabla_v, \nabla_u] = \nabla_{[v,u]}$ for vector fields $v,u$, which is replaced by 
   \[ [\nabla_v, \nabla_u] = \nabla_{[v,u]} + n\omega(v, u). \]
  So it makes sense to take any $n \in \bC$ (in particular, $n=\frac{1}{2}$, as noted above).

\begin{exercise}
  For $\lambda\in \mathbb C$, let $U_\lambda$ be the quotient of
  $U(\mathfrak{sl}_2)$ by the relation
  $C=\frac{1}{2}\lambda(\lambda+2)$, where $C=ef+fe+\frac{1}{2}h^2$ is
  the Casimir element. Show that $U_\lambda$ is isomorphic to the
  algebra of twisted differential operators $\cD(\mathbb
  P^1,\cO(1)^{\otimes \lambda})$. This is the simplest instance of the
  {\it Beilinson-Bernstein localization}, \cite{BB}. (Hint: use the
  representation of $\mathfrak{sl}_2$ given by formula \eqref{twidi}).
\end{exercise} 

\subsection{Twisted differential operators on \texorpdfstring{${\rm Bun}_G(X)$}{Bun\_G(X)}}\label{twidif}

Motivated by this, we replace $\cD(\Bun_G(X))$ with the algebra $\cD(\Bun_G(X),
\cK^{1/2})$ of differential operators on the line bundle $\cK^{1/2}$ (which in this case actually exists, see \cite{BD}). An important feature of this line bundle is that it
can be obtained by descending a line bundle from $G(K)=G((t))$ to the double quotient. 
Namely, for brevity let $G$ be simply connected and simple and $\mathbf k=\mathbb C$.
Then there exists the {\it Kac--Moody central extension} 
\[ 1 \to \bC^\times \to \widehat G \to G((t)) \to 1 
\]
which gives rise to a $\bC^\times$-bundle (i.e., a line bundle) $\mathcal L$
on $G((t))$ with 
$$
c_1(\mathcal L)\in H^2(G((t)),\mathbb Z)\cong H^3(G,\mathbb Z)\cong \mathbb Z
$$ 
being the element $1$. The Lie algebra of the resulting {\it Kac-Moody group} 
$\widehat G$ is the {\it affine Kac--Moody algebra}
\[ \widehat{\lie{g}} \coloneqq \lie{g}((t)) \oplus \bC {\rm K} \]
with commutator given by 
$$
[a(t), b(t)] \coloneqq [a,b](t) + \Res_{t=0} (da(t),b(t)) {\rm K}
$$ 
where the inner product on $\lie{g}$ is normalized so that long
roots have squared length $2$.

\begin{theorem}[\cite{BD}] \label{BDtheo}
  The bundle $\cK_{\Bun_G(X)}$ is the descendant of the line bundle
  $\cL^{-2h^\vee}$ on ${\rm Bun}_G(X)$, where $h^\vee$ is the dual
  Coxeter number of $\mathfrak g$.
\end{theorem}

\subsection{The Sugawara construction}

So we need to work with two-sided invariant twisted differential operators on $G((t))$ 
acting on the line bundle $\cL^{-h^\vee}$. 

To make this precise, given a number $k\in \mathbb C$ (called the {\it level}), consider the algebra 
$U_k(\lie{g}((t))) \coloneqq U(\widehat{\lie{g}})/(K-k)$ and define its completion $\widehat U_k(\lie{g}((t)))$ 
in the same way as we defined 
$\widehat U(\lie{g}((t)))$. Then left-invariant 
twisted differential operators on $G((t))$ 
acting $\cL^{\otimes k}$ are just elements 
on this algebra, and two-sided invariant operators are its central elements. 
Remarkably, it turns out that the algebra $\widehat U_k(\lie{g}((t)))$
for $k\in \mathbb C$ has a nontrivial center if and only if $k=-h^\vee$ (the so-called {\it critical level}). 
In fact, some of these central elements are not difficult to construct. 

Namely, let us keep the notation of Subsection \ref{quanosec}, 
but now view $T_n$ as elements of $\widehat U_k(\lie{g}((t)))$. 
Some care must now be taken, because the infinite sum defining $T_0$ is not an element of the completion $\widehat U_k(\lie{g}((t)))$ unless $k=0$: the series no longer converges since $[\phi_m^i,\phi_{-m}^i]=mk$. The way to fix this is to use {\it normal ordering}, i.e. to define
$$
T_n \coloneqq \frac{1}{2}\sum_{m,i} \, \NO{\phi_m^i \phi_{n-m}^i},
$$
where the normally-ordered product $\NO{\phi_m^i \phi_\ell^i}$ is defined 
by the formula
$$
\NO{\phi_m^i \phi_\ell^i}=\begin{cases}\phi_m^i \phi_\ell^i\text{ if }m\le \ell\\ \phi_\ell^i \phi_m^i\text{ if }m>\ell.\end{cases}
$$
This does not affect $T_n$ for $n \neq 0$, but changes $T_0$ into a well-defined element of $\widehat U_k(\lie{g}((t)))$.

Let ${\rm W}$ be the (formal) {\it Witt algebra}, i.e., the Lie algebra of vector fields on $D^\times$, ${\rm W}=\mathbb C((t))\partial_t$, with bracket 
$$
[f\partial_t,g\partial_t]=(fg'-gf')\partial_t.
$$ 
Then ${\rm W}$ has a (topological) basis $L_n\coloneqq -t^{n+1}\partial_t, n\in \mathbb Z$, such that 
$$
[L_n, L_m] = (n-m) L_{m+n}.
$$
It is well known (see \cite{Ka}) that ${\rm W}$ has a unique non-trivial 1-dimensional central extension 
called the (formal) {\it Virasoro algebra} and denoted ${\rm Vir}$. It has a (topological) basis consisting of $L_n,n\in \mathbb Z$ and a central element $C$ such that 
$$
     [L_n, L_m] = (n-m) L_{m+n} + \frac{n^3-n}{12}\delta_{n+m,0}C.
$$
If $V$ is a representation of ${\rm Vir}$ with $C$ acting as $c\in \mathbb C$, then we say that 
$V$ has {\it central charge} $c$. 

Generalizing the computation of Exercise \ref{quano} to the case of
nonzero level, and taking the normal ordering into
  account, we obtain the following theorem.

\begin{theorem}[Sugawara construction, see \cite{Ka}] One has 
  \begin{align*}
    [\phi_\ell, T_n] &= (k + h^\vee) \ell\phi_{n+\ell}, \\
    [T_n, T_m] &= (k+ h^\vee) (n-m) T_{m+n} + \frac{n^3-n}{12} k(k + h^\vee) \dim (\lie{g}) \delta_{n+m,0}.
  \end{align*}
  Hence, when $k\neq -h^\vee$, the elements
  \[ L_n \coloneqq \frac{T_n}{k + h^\vee} \]
  define a representation of the {\it Virasoro algebra} with central charge $c \coloneqq \frac{k \dim
    \lie{g}}{k + h^\vee}$ compatible with $\phi_\ell$:
    $$
    [\phi_\ell, L_n] = \ell\phi_{n+\ell}, 
    $$
    $$
     [L_n, L_m] = (n-m) L_{m+n} + \frac{n^3-n}{12} \frac{k\dim \lie{g}}{k + h^\vee} \delta_{n+m,0}.
$$
On the other hand, when $k=-h^\vee$ (the critical level) then $T_n$ are central elements. 
\end{theorem}

The elements $T_n$ are called the {\it Sugawara elements}. 

\subsection{\texorpdfstring{$\mathfrak{sl}_2$}{sl2}-opers on \texorpdfstring{$D^\times$}{D*}}
Note that at the non-critical level $k\ne -h^\vee$ the commutation relations for $T_n$ can be written as 
$$
[L_n, T_m] = (n-m) T_{m+n} + \frac{n^3-n}{12} K\dim (\lie{g}) \delta_{n+m,0}.
$$
In particular, this has a well defined specialization at the critical level, to an action 
of ${\rm W}$ on $\mathbb C[T_n,n\in \mathbb Z]$: 
$$
L_n\circ T_m = (n-m) T_{m+n} - \frac{n^3-n}{12} h^\vee\dim (\lie{g}) \delta_{n+m,0}.
$$
This action deforms the action in the classical case given by 
$$
L_n\circ T_m = (n-m) T_{m+n},
$$
which expresses the fact that $T_n\in {\rm W}$ are linear functions on the space ${\rm W}^*$ of quadratic 
differentials $\mathbb C((t))(dt)^2$ on $D^\times$: for $\eta=\eta(t)(dt)^2\in \mathbb C((t))(dt)^2$, 
$$
T_n(\eta)=-{\rm Res} (t^{n+1}\partial_t\cdot \eta)=-{\rm Res}(t^{n+1}\eta(t)dt). 
$$
In the quantum case, however, the additional summand means that 
$T_n$ are rather {\it affine linear} functions on the hyperplane ${\rm Vir}_1^*\subset {\rm Vir}^*$
defined by the equation $C=1$ (or, equivalently, $C=c$ for any nonzero value of $c$). 
So we may ask for the geometric meaning of elements of ${\rm Vir}_1^*$; this is a certain deformation of the notion of a quadratic differential on $D^\times$. This question is answered by the following exercise. 

\begin{exercise}\label{operex}
  (i) Show that the coadjoint representation ${\rm Vir}^*$ of the Virasoro algebra ${\rm Vir}$
   is isomorphic as a ${\rm W}$-module (hence as an ${\rm Aut}(D)$-module) to the space of differential operators 
   $$
   L=\alpha \di_t^2 + u(t)\colon \cK^{-1/2} \to \cK^{3/2}
   $$ 
   on $D^\times$, for
  $\alpha \in \bC$, i.e., 
  $$
  L(f(t)(dt)^{-1/2})=(\alpha f''(t)+u(t)f(t))(dt)^{3/2},
  $$
  where $f,u\in \mathbb C((t))$. In particular, ${\rm Vir}_1^*$ is identified 
  with the space of {\it Hill operators} 
  $$
  L=\partial^2_t+u(t)\colon \cK^{-1/2}\to \cK^{3/2}.
  $$ 

(ii) Part (i) implies   that the space of Hill operators is invariant under the group ${\rm Aut}(D)$ of formal changes of the variable $t$. Give another proof of this fact by giving a coordinate-free definition of a Hill operator; namely, interpret Hill operators as formally self-adjoint second-order differential operators $\cK^{-1/2}\to \cK^{3/2}$ with symbol $1$.\footnote{If $L$ is a differential operator $\cK^a\to \cK^b$ on $D^\times$ then one can 
canonically define the formal adjoint
$L^*\colon \cK^{1-b}\to \cK^{1-a}$ with respect to the pairing $(f,g)={\rm Res}(fg)$ between $\cK^b$ and $\cK^{1-b}$, which is a differential operator of the same order. So if $a+b=1$, it makes sense to say that $L$ is self-adjoint: $L=L^*$. Also if the order of $L$ is $n$ then we can canonically define its symbol $\sigma(L)$, which is a section of $\cK^{b-a-n}$. So if $b-a=n$ then the symbol is a function (element of $\mathbb C((t))$).}

(iii) Compute how $L$ transforms
  under changes $t\mapsto s(t)$ of the formal coordinate $t$. Hint: you should see the  function 
  $$
  D(s)=\tfrac{s'''}{s'}-\tfrac{3}{2}(\tfrac{s''}{s'})^2
  $$ called the {\it Schwarzian derivative} of $s$.
\end{exercise}

\begin{definition}[\cite{BD,BD:opers}] A Hill operator $L=\partial^2_t+u(t)\colon \cK^{-1/2}\to \cK^{3/2}$, 
where $u\in \mathbb C((t))$, is called an {\it $\mathfrak{sl}_2$-oper} on $D^\times$. 
\end{definition} 

This is a special case of the general notion of a $\lie{g}$-oper for a finite dimensional simple Lie algebra $\lie{g}$ defined in \cite{BD}. The terminology 
is motivated by the fact that the problem of solving the differential 
equation $L\psi=0$ reduces to the problem of integrating the $\SL_2$-connection 
$$
\nabla=\partial+\begin{pmatrix} 0& 1\\ u & 0\end{pmatrix}.
$$

Thus we see that the subalgebra $Z_{\rm Sug}$ of the center $Z$ of $\widehat U_{-h^\vee}(\widehat{\lie{g}})$ generated (topologically) by the Sugawara elements $T_n$ may be interpreted as the algebra of polynomial functions on the space ${\rm Op}_{\mathfrak{sl}_2}(D^\times)\coloneqq{\rm Vir}_1^*$ of $\mathfrak{sl}_2$-opers 
on $D^\times$. 

\begin{exercise} Show that for a function $s=s(t)$ ($C^3$ on an interval or holomorphic on a disk), one has $D(s)=0$ if and only if $s$ is a M\"obius transformation $s(t)=\frac{at+b}{ct+d}$. 
Use this to show that under M\"obius transformations $t\mapsto \frac{at+b}{ct+d}$ (with $b$ being a formal variable), $\mathfrak{sl}_2$-opers transform as quadratic differentials. (This is a reflection of the fact that the 2-cocycle $\frac{n^3-n}{12}\delta_{n+m,0}$ defining the Virasoro algebra vanishes for $n=-1,0,1$, i.e., on the M\"obius Lie subalgebra $\langle L_{-1},L_0,L_1\rangle$ of the Witt algebra ${\rm W}$). 
\end{exercise}

\subsection{\texorpdfstring{$\mathfrak{sl}_2$}{sl2}-opers on curves} 

Now let $X$ be a smooth algebraic curve over $\mathbb C$. By analogy with the previous subsection we make the following definition. 

\begin{definition} An $\mathfrak{sl}_2$-oper on $X$ is a differential operator 
$\cK_X^{-1/2}\to \cK_X^{3/2}$ which in local coordinates looks like $\partial_t^2+u(t)$ where $u$ is a regular function. 
\end{definition} 

Similarly to Exercise \ref{operex}, we may alternatively define an $\mathfrak{sl}_2$-oper
as a formally self-adjoint second order differential operator $\cK_X^{-1/2}\to \cK_X^{3/2}$ with symbol $1$. As this definition is independent of the choice of coordinates, so is the original one. 

\begin{exercise} \label{operbun}
  Let $X$ be a smooth irreducible projective curve over $\mathbb C$ of
  genus $g \ge 2$ and fix a spin structure on $X$ (i.e., a square root
  $\cK_X^{1/2}$; these choices form a torsor over the group ${\rm
    Jac}_{X,2}\cong (\mathbb Z/2)^{2g}$ of points of order $2$ on
  ${\rm Jac}(X)$).
  \begin{enumerate}[label=(\roman*)]
  \item Show that $X$ admits a unique rank $2$ vector bundle $E_X$,
    called the {\it oper bundle}, such that there exists a non-split short exact sequence
    \[ 0 \to \cK_X^{1/2} \to E_X \to \cK_X^{-1/2} \to 0, \]
    and this sequence is unique up to scaling the arrows. 
  
  \item Show that $\lie{sl}_2$-opers on $X$ are in natural bijection
    with connections on $E_X$.
  
  \item Show that $\lie{sl}_2$-opers on $X$ exist and form a torsor over $H^0(X,
    \cK_X^{\otimes 2})$ (i.e. an affine space of dimension $3g-3$).
  
  \item Compute $\lie{sl}_2$-opers explicitly on a genus $2$ curve using the
    hyperelliptic realization.
  \end{enumerate}
\end{exercise}

\begin{exercise}\label{difgeo}
  \begin{enumerate}[label=(\roman*)]
  \item Let $X$ be a smooth irreducible projective curve over $\mathbb
    C$, viewed as a Riemann surface. Use the uniformization theorem to
    give an open cover of $X$ with transition maps being M\"obius
    transformations with real coefficients. Use this cover to identify
    the affine space ${\rm Op}_{\mathfrak{sl}_2}(X)$ with the vector
    space of quadratic differentials $H^0(X,\cK_X^{\otimes 2})$. In particular,
    this introduces an origin in ${\rm Op}_{\mathfrak{sl}_2}(X)$,
    denoted $L_{\rm un}$ and called the {\it uniformization oper} (a
    term motivated by part (ii)).

  \item Let $J\colon \mathbb C_+\to X$ be a uniformization map, and
    let $f=J^{-1}$ be the multivalued inverse function. Show that
    $f(z)=\frac{\psi_1(z)}{\psi_2(z)}$, where $\lbrace
    \psi_1,\psi_2\rbrace$ is a basis of the space of holomorphic
    solutions of the differential equation $L_{\rm un}\psi=0$ near some point $x\in X$ with
    Wronskian $1$ in which the monodromy matrices of this equation
     have real entries. Show that both
    uniformization maps and such bases (up to sign of $\psi_j$) form
    ${\rm PSL}_2(\mathbb R)$-torsors, and the above correspondence is
    an isomorphism between these torsors.

  \item Let
    $$
    \beta(z,\overline z)\coloneqq\im(\psi_1(z)\overline{\psi_2}(z))=\frac{\psi_1(z)\overline{\psi_2}(z)-\psi_2(z)\overline{\psi_1}(z)}{2i}.
    $$
    Show that $\beta$ is a positive $-1/2$-density on $X$ independent
    on the choice of the basis $\lbrace\psi_1,\psi_2\rbrace$.

  \item Recall that in a Riemann surface, conformal metrics are in
    natural bijection with positive densities (in local coordinates,
    $\rho (dx^2+dy^2)$ corresponds to $\rho \,dx\, dy$). Show that
    (upon suitable normalization) $\beta^{-2}$ is the positive density on $X$ corresponding to the
    Poincar\'e metric (with Gaussian curvature $-1$).
  \end{enumerate}
\end{exercise} 

\begin{exercise}[\cite{Fa2,Go}; see also \cite{EFK}]
  This is a generalization of Exercise \ref{difgeo}.

  \begin{enumerate}[label=(\roman*)]
  \item An $\mathfrak{sl}_2$-oper $L$ on $X$ is said to have {\it real
    monodromy} if there exists a basis $\lbrace\psi_1,\psi_2\rbrace$
    of holomorphic solutions of the equation $L\psi=0$ with Wronskian
    $1$ in which its monodromy matrices are real. Show that this is
    equivalent to the existence of a non-zero single-valued real
    $C^\infty$ solution $\beta$ of the equation $L\beta=0$, and such a
    solution, when exists, is unique up to scaling.

  \item Let $L$ be an $\mathfrak{sl}_2$-oper on $X$ with real
    monodromy, and $\beta$ be the corresponding real solution of
    the differential equation $L\beta=0$. Show that the system of equations $\beta=0$,
    $d\beta=0$ has no solutions on $X$, so the zero set $Z(L)$ of
    $\beta$ is a smooth $1$-dimensional real submanifold on $X$ (a
    collection of non-intersecting simple closed smooth curves).

  \item Show that the density $\beta^{-2}$ defines a complete
    Poincar\'e metric on $X\setminus Z(L)$ with logarithmic
    singularities on $Z(L)$ (i.e., near a point of $Z(L)$ it is
    isomorphic to the Poincar\'e metric on the upper half plane near
    the origin).

  \item A {\it real projective structure} on $X$ is an equivalence class of
    atlases of charts with transition maps being real M\"obius
    transformations, i.e., elements of ${\rm PSL}_2(\mathbb R)$ (two
    atlases are equivalent if they have a common refinement). Show
    that $\mathfrak{sl}_2$-opers on $X$ with real monodromy correspond
    to equivalence classes of real projective structures on $X$.
  \end{enumerate}
\end{exercise} 

\subsection{Quantization of the Hitchin system for \texorpdfstring{$G=\SL_2$}{G=SL2}}

We can now use the Sugawara elements $T_n$ to construct quantizations of quadratic Hitchin Hamiltonians, as outlined at the end of Subsection \ref{twidif}. Namely, according to Theorem \ref{BDtheo}, $T_n$ descend to second order twisted differential operators $\overline T_n\in \mathcal D({\rm Bun}_G^\circ(X),\cK^{1/2})$. 
By considering the semiclassical limit, it is easy to see that $\lbrace\overline T_n \rbrace$ span a vector space of dimension $3g-3$ and are algebraically independent. 
So for $G=\SL_2$ they provide a quantization of the Hitchin system. 

Let us consider the case $G=\SL_2$ in more detail. Recall that in this case 
$T_n$ can be interpreted as functions on the space ${\rm Op}_{\mathfrak{sl_2}}(D^\times)$. 
Namely, given $L\in {\rm Op}_{\mathfrak{sl_2}}(D^\times)$, $L=\partial_t^2+u(t)$, 
where $u(t)=\sum_{n\in \mathbb Z} u_nt^n$, we have $T_n(L)=u_{-n-2}$. 
It is clear that the same should apply to $\overline T_n$, except 
that now $L\in {\rm Op}_{\mathfrak{sl_2}}(X)$ is an oper on $X$; 
this is just the quantum analog of the commutative diagram \eqref{comdia}. 
Namely, similarly to Remark \ref{desce}, for $n>-1$ the element $T_n$ descends to $0$, 
while for $n\le -2$ it descends to the Taylor coefficient $u_{-n-2}$ of the oper $L$ on $X$. 

Moreover, as in the classical case, this story extends straightforwardly to the ramified case, with parabolic structures and, more generally, twistings (see Section \ref{sec:hitchin-with-parabolic-and-twists}), using the notion of quantum Hamiltonian reduction with respect to an ideal in $U(\lie{g})$ (cf. Subsection \ref{qhr}). 

\begin{exercise} For $\lambda\in \mathbb C$ let $I_\lambda\subset U(\mathfrak{sl}_2)$ be the ideal generated by $C-\frac{\lambda(\lambda+2)}{2}$, where $C=ef+fe+\frac{h^2}{2}$ is the Casimir element. 
Show that the twisted quantum Hitchin system for $X=\mathbb P^1$ 
with marked points $t_1,\ldots,t_N$ and twistings by the ideals $I_{\lambda_j}$ 
at $t_j$ is the Gaudin system with parameters $\lambda_j$. 
\end{exercise} 

Finally, in view of Remark~\ref{encoun}, by considering the semiclassical limit, we obtain the following theorem in the unramified case. 

\begin{theorem}[\cite{BD}]
  Let $A$ be the algebra of twisted differential operators on ${\rm
    Bun}_G(X)$ for $G=\SL_2$ generated by the elements $\overline
  T_n$. Then
  \begin{enumerate}[label=(\roman*)]
  \item $A$ is a polynomial algebra in $3g-3$ generators, and ${\rm
    Spec}A$ is naturally identified with the affine space of opers
    ${\rm Op}_{\lie{sl}_2}(X)$;

  \item $A$ coincides with the algebra of all global twisted
    differential operators $\cD({\rm Bun}_G(X),\cK^{1/2})$.
  \end{enumerate}
\end{theorem} 

Thus we see that quantization of the Hitchin system is given by a canonical map 
$\mathbb C[{\rm Op}_{\lie{sl}_2}(X)]\to \cD({\rm Bun}_G(X),\cK^{1/2})$, which happens to be an isomorphism. This shows that the affine space ${\rm Op}_{\lie{sl}_2}(X)$ 
plays the role of quantization of the Hitchin base $\mathcal B$, which 
in the $\SL_2$ case is the space of quadratic differentials on $X$. 

\begin{remark} In the ramified case the spectrum of the Hitchin system should be 
interpreted as the space of opers with appropriate singularities. We will not discuss this here and refer the reader to \cite{Fr2,BD:opers}.
\end{remark} 

\subsection{The Feigin--Frenkel theorem and quantization of Hitchin systems in higher rank}

We would now like to generalize quantization of Hitchin systems to
simple groups $G$ of rank $r>1$. And here we encounter a serious
difficulty: the elements $\overline T_n$ still span a space of
dimension $3g-3$, so no longer suffice for a quantum integrable
system. To complete this collection of operators to an integrable
system using our approach, we must construct more central elements of
$\widehat U_{-h^\vee}(\lie{g}((t)))$, independent of
the $T_n$. Remarkably, the needed elements do exist, but the proof of
their existence is quite non-trivial. This is the celebrated theorem
of Feigin and Frenkel, which is one of the most fundamental facts
about affine Lie algebras.

\begin{theorem}[Feigin--Frenkel, \cite{FF}]
  The center of $\widehat U_{-h^\vee}(\lie{g}((t)))$ is 
  generated (topologically) by the elements $\widehat H_{i,n}$, $i=1,\ldots,r, n\in \mathbb Z$, which quantize the classical 
  Hamiltonians $H_{i,n}$, and such that $\widehat H_{2,n}=T_{-n-2}$. 
\end{theorem}

The difficult part of the theorem is to show that every $H_{i,n}$ for $i>2$ 
admits a quantization $\widehat H_{i,n}$; the fact that these quantizations generate the center 
is then proved by a deformation argument.\footnote{For classical groups, there are 
explicit formulas for $H_{i,n}$, but for exceptional groups  of type $E$ and $F$ 
no nice formulas are known.}

Using this result, Beilinson and Drinfeld proved the following theorem in the unramified case.

\begin{theorem}[Beilinson--Drinfeld, \cite{BD}] \label{bdmain}
  Let $A$ be the algebra of twisted differential operators generated
  by the descendants $\overline{\widehat{H}}_{i,n}$ of the elements
  $\widehat H_{i,n}$. Then
  \begin{enumerate}[label=(\roman*)]
  \item $A$ is a polynomial algebra in $(g-1)\dim \lie{g}$ generators,
    which provides a quantization of the quantum Hitchin system for
    $\lie{g}$.
 
  \item $A$ coincides with the entire algebra $\cD(\Bun_G(X),
    \cK^{1/2})$.
  \end{enumerate}
\end{theorem}

In the ramified case (with parabolic structures or, more generally, twistings), 
the story is similar --- the quantum Hitchin Hamiltonians are still constructed 
as descendants of $\widehat H_{i,n}$, although part (ii) of Theorem \ref{bdmain} no longer holds. 

\begin{exercise}[cf. \cite{E3}]
  Construct a quantization of the Calogero--Moser system as a special
  case of the quantum twisted Hitchin system. Namely, quantize the
  procedure of Subsection \ref{ellcm}.
\end{exercise} 

\subsection{Opers in higher rank and Langlands duality} 

Finally, it remains to extend to higher rank the notion of an oper, i.e., 
generalize the statement that $A$ is the algebra of regular functions on the space of opers. 
This can be done as follows. 

Let $G$ be an adjoint simple group with Lie algebra $\lie{g}$, and $B\subset G$ a Borel subgroup with Lie algebra $\lie{b}$. Fix a principal $\lie{sl}_2$-triple $e,h,f\in \lie{g}$ such that $e,h\in \mathfrak{b}$. 

\begin{definition}[\cite{BD,BD:opers}] \label{operde}
  A $\lie{g}$-{\bf oper} on $X$ is a triple
  $(E,E_B,\nabla)$, where $E$ is a $G$-bundle on $X$,
  $E_B$ is a $B$-reduction of $E$, and
  $\nabla$ is a connection on $E$ which has the form 
 $$
 \nabla=d+(f+b(t))dt,\ b\in \mathfrak{b}[[t]]
 $$ 
 for any trivialization of $E_B$ (and hence $E$) on the formal
 neighborhood of any
 point $x$ of $X$ (where $t$ is a formal coordinate at
 $x$).\footnote{This definition doesn't actually depend on the choice of the principal $\lie{sl}_2$-triple since all principal $\lie{sl}_2$-triples $e,h,f\in \lie{g}$ such that $e,h\in \mathfrak{b}$ are conjugate by $B$.}  \end{definition}

In this definition, $X$ could be any smooth curve, a formal disk, or a punctured formal disk. 
But if $X$ is a projective curve then there is another, equivalent definition of an oper. Namely, 
let $\rho\colon SL_2\to G$ be the homomorphism corresponding to the triple $e,h,f$.
Then we can consider the associated $G$-bundle $E_{X,\rho}$ 
to the $SL_2$ oper bundle $E_X$ of Exercise \ref{operbun} via $\rho$. 

\begin{proposition}[\cite{BD,BD:opers}] A $\lie{g}$-oper is the same thing as a connection on the bundle $E_{X,\rho}$. In other words, the underlying bundle of a $\lie{g}$-oper 
is always isomorphic to $E_{X,\rho}$, and any connection on $E_{X,\rho}$ 
admits a unique oper structure (i.e., a $B$-structure satisfying Definition \ref{operde}). 
\end{proposition} 

\begin{exercise} Show that $\lie{sl}_n$-opers on $X$ bijectively correspond 
to differential operators $L\colon \cK_X^{\frac{-n+1}{2}}\to \cK_X^{\frac{n+1}{2}}$
such that in local coordinates 
$$
L=\partial_t^n+a_2(t)\partial_t^{n-2}+\cdots+a_n(t).
$$
\end{exercise}

\begin{theorem}[\cite{BD}] \label{bdmain1} 
  \begin{enumerate}[label=(\roman*)]
  \item $\lie{g}$-opers form an affine space ${\rm
    Op}_{\mathfrak{g}}(X)$ whose underlying vector space is the
    Hitchin base $\mathcal B_{X,G}$.

  \item The algebra $A$ of Theorem \ref{bdmain} is naturally
    isomorphic to $\mathbb C[{\rm Op}_{\mathfrak{g}^\vee}(X)]$, where
    $\mathfrak{g}^\vee$ is the {\it Langlands dual} Lie algebra of
    $\lie{g}$ (i.e., the Lie algebra with the dual root system to the
    root system of $\lie{g}$).
  \end{enumerate}
\end{theorem} 

\begin{remark}
  \begin{enumerate}
  \item Note that $\mathcal B_{X,G}$ is canonically
    isomorphic to $\mathcal B_{X,G^\vee}$.

  \item In fact, a more precise version of the Feigin-Frenkel theorem
    provides a natural identification of the center of $\widehat U_{-h^\vee}(\lie{g}((t)))$
     with the algebra of
    regular functions on the space ${\rm
      Op}_{\mathfrak{g}^\vee}(D^\times)$. Using this identification,
    the procedure of descending central elements to Hitchin
    Hamiltonians corresponds to the inclusion ${\rm
      Op}_{\mathfrak{g}^\vee}(X)\hookrightarrow {\rm
      Op}_{\mathfrak{g}^\vee}(D^\times)$.
  \end{enumerate}
\end{remark} 

Thus Theorem \ref{bdmain1} should be viewed as an instance of {\it Langlands duality}. 
It is, in fact, a starting point of the {\it geometric Langlands program}. But this is already beyond the scope of this paper.  

\section{Solutions of problems}

\subsection{Problem \ref{bafun}}

  (i) This follows since 
  $\Pic(X\setminus 0)\cong \Pic(X)/\langle \cO(0)\rangle=\Pic_0(X)$
  and the class of $L$ in $\Pic_0(X)$ is nontrivial. 

  (ii) We realize $X$ as $\mathbb C/\langle 1,\tau\rangle$, where
  $\im\tau>0$. Let $\theta(z)\coloneqq\theta(z,\tau)$ be the
  theta-function of $X$, which is an entire function with simple zeros
  on the period lattice of $X$. This function is periodic with period
  $1$ and
  \[ \theta(z + \tau) = -e^{-2\pi i z}\theta(z). \]
  Thus $\frac{\theta'}{\theta}$ is periodic with period $1$ and 
  $\frac{\theta'}{\theta}(z + \tau) = \frac{\theta'}{\theta}(z) - 2\pi i$.
  Hence for $a\notin \langle 1,\tau\rangle$, the {\it Lame--Hermite function}
  \[ H(z,a) \coloneqq e^{a\frac{\theta'}{\theta}(z)} \frac{\theta(z-a)}{\theta(z)} \]
  is doubly-periodic, i.e. is a holomorphic function on $X \setminus
  0$. It has a simple zero at $a$ and no other zeros and poles, but it
  has an essential singularity at $0$: $H(z)\sim Cz^{-1}e^{\frac{a}{z}}$, $z\to 0$. Thus $H$ may be viewed as a
  non-vanishing holomorphic section of the analytic line bundle
  $\cO(a)^\vee$ over $X \setminus 0$. Hence $\cO(a)$ is trivial on $X
  \setminus 0$. But the given line bundle $L$ is of the form $L =
  \cO(a) \otimes \cO(0)^\vee$ for some $0 \neq a \in X$, so we are
  done.\footnote{The Lam\'e--Hermite function arises as the Baker--Akhiezer function 
for elliptic curve in the theory of integrable systems, see \cite{Kr}.}

\subsection{Problem \ref{usegaga}}

 Pick $r>0$ such that 
$A(z)$ and $A^{-1}(z)$ are regular on the circle $|z|=r$ (this can be done 
since these functions have countably many poles, while the set of choices for $r$ is uncountable). By rescaling $z$, we may assume without loss of generality that $r=1$. 

Consider the elliptic curve $X\coloneqq\mathbb C^\times/q^{\mathbb Z}$.
Cover $X$ by two open sets: 
$$
U_1=\lbrace z\in \mathbb C: |q|<|z|<1\rbrace,\ U_2=\lbrace z\in \mathbb C: 1-\varepsilon<|z|<1+\varepsilon\rbrace,
$$ 
for sufficiently small 
$\varepsilon>0$. The intersection $U_1\cap U_2$ has two connected components 
$W_\pm$, where 
$$
W_+\coloneqq\lbrace z\in \mathbb C: 1<|z|<1+\varepsilon\rbrace,\ W_-\coloneqq\lbrace z\in \mathbb C: 1-\varepsilon<|z|<1\rbrace.
$$ 
Define a holomorphic vector bundle $E$ on $X$ by the transition function $g(z)$ which equals 
$1$ on $W_-$ and $A(z)$ on $W_+$ (this is well defined since for 
small enough $\varepsilon$, the matrix functions $A,A^{-1}$ are regular on $W_+$). It is easy to see that then vector solutions of the difference equation 
$f(qz)=A(z)f(z)$ correspond to meromorphic sections of $E$. 
By the GAGA theorem, $E$ has an algebraic structure, hence trivializes algebraically 
on some open set $X\setminus \lbrace x_1,\ldots,x_m\rbrace$ by Hilbert's Theorem 90. 
Thus $E$ has a basis of meromorphic sections $f_1,\ldots,f_n$ with no poles outside $x_1,\ldots,x_m$. Arranging them into a matrix $f=(f_1,\ldots,f_n)$, we obtain the desired solution. 

\subsection{Problem \ref{bunelcur}}

  Let $n \coloneqq \deg \cL$. Without loss of generality, we may
  assume $n > 0$; otherwise, replace $\cL$ with $\cL^\vee$. Then $\cL
  = \cO(x_1 + \cdots + x_n)$ for some points $x_1, \ldots, x_n \in X$.
  Let $y \in X$ be such that $x_1 + \cdots + x_n = ny$. We can
  construct such a $y$ by lifting from $X = \bC/\Gamma$ to $\bC$,
  taking $y$ to be the average of the $x_i$, and then projecting back
  to $\bC/\Gamma$. Then $\cL = \cO(y)^{\otimes n}$, so $\cL$ is
  trivialized by deleting $y$.

  The same is not true on a genus $2$ curve $X$ purely by dimension
  counting: the Jacobian has dimension two, but the choice of a point
  $x \in X$ is only one dimension's worth of freedom.

\subsection{Problem \ref{PGLSL}}

  By Corollary \ref{grothen}, $\GL_n$-bundles on $\mathbb P^1$ have the form $\cO(m_1) \oplus
  \cdots\oplus \cO(m_n)$ for unique $m_1 \ge m_2\ge\cdots\ge m_n$, $m_i\in \mathbb Z$. 
  Since $\PGL_n$-bundles are
  equivalence classes of $\GL_n$-bundles under tensor product with
  line bundles, and line bundles on $\bP^1$ are all of the form
  $\cO(k)$ for some $k \in \bZ$, it follows that $\PGL_n$-bundles are
  classified by the differences $\ell_i\coloneqq m_i - m_n$, which form 
  a sequence of integers $\ell_1\ge\cdots\ge \ell_{n-1}\ge 0$.
  
  Similarly, $\SL_n$-bundles are $\GL_n$-bundles with trivial determinant, so they  
  have the form $\cO(m_1) \oplus
  \cdots\oplus \cO(m_n)$ for unique $m_1 \ge m_2\ge\cdots\ge m_n$, $m_i\in \mathbb Z$, $\sum_{i=1}^n m_i=0$. 
  
\subsection{Problem \ref{hankel}}

(i) By Theorem \ref{grrank2}, $E \cong \cO(k) \oplus \cO(m-k)$ for 
  unique integer $\frac{m}{2}\le k\le m$.  It follows that for an integer $\frac{m}{2}\le n\le m$
  \[ E \cong \cO(k) \oplus \cO(m-k),\ \frac{m}{2}\le k\le n \iff \Hom(\cO(n+1), E) = 0. \]
  In other words, we want to find $f(z)$ such that $E \otimes
  \cO(-n-1)$ has no non-zero section. The transition function for this
  bundle is
  \[ \begin{pmatrix} z^{-n-1} & z^{-n-1} f(z) \\ 0 & z^{m-n-1} \end{pmatrix} \]
  and a section is a pair $(\binom{x_0(z)}{y_0(z)}, \binom{x_\infty(z^{-1})}{
  y_\infty(z^{-1})})$ of vector-valued polynomials such that
  \[ \begin{pmatrix} x_0(z) \\ y_0(z) \end{pmatrix} = \begin{pmatrix} z^{-n-1} & z^{-n-1} f(z) \\ 0 & z^{m-n-1} \end{pmatrix} \begin{pmatrix} x_\infty(z^{-1}) \\ y_\infty(z^{-1}) \end{pmatrix}. \]
  Let's try to construct such a section. The equation for the second entry implies
  \[ y_\infty(z^{-1}) = c_0 + c_{-1} z^{-1} + \cdots + c_{n-m+1} z^{n-m+1} \]
  for some scalars $\{c_{-i}\}_{i=0}^{m-n-1}$. Plugging this into the
  first equation, we get
  \[ z^{-n-1} x_\infty(z^{-1}) + z^{-n-1} (a_1 z + \cdots + a_{m-1} z^{m-1}) (c_0 + \cdots + c_{n-m+1} z^{n-m+1}) \]
  must be a polynomial. We can use the freedom to choose
  $x_\infty(z^{-1})$ to cancel all terms of degree $< -n$. The
  remaining terms of degree in $[-n, -1]$ must vanish. This is the
  condition that the equations
  \begin{align*}
    a_1 c_0 + a_2 c_{-1} + \cdots + a_{m-n} c_{n-m+1} &= 0 \\
    a_2 c_0 + a_3 c_{-1} + \cdots + a_{m-n+1} c_{n-m+1} &= 0 \\
    &\vdots \\
    a_n c_0 + a_{n+1} c_{-1} + \cdots + a_{m-1} c_{n-m+1} &= 0
  \end{align*}
  have a non-zero solution $(c_0, \ldots, c_{n-m+1})$, i.e. that the
  matrix
  \[ A \coloneqq \begin{pmatrix} a_1 & a_2 & \cdots & a_{m-n} \\ a_2 & a_3 &\cdots & a_{m-n+1} \\ & & \cdots \\ a_n & a_{n+1} & \cdots & a_{m-1}\end{pmatrix} \]
  has rank $<m-n$. So the desired condition on $a_1, \ldots, a_{m-1}$ is that the matrix $A$ has full rank (equal to $m-n$). 
  
(ii) In the special case $m=2n$ we get that $E\cong \cO(n)\oplus \cO(n)$
iff 
$$
H(a_1,\ldots,a_{2n-1})\coloneqq\det \begin{pmatrix} a_1 & a_2 & \cdots & a_{n} \\ a_2 & a_3 &\cdots & a_{n+1} \\ & & \ddots \\ a_n & a_{n+1} & \cdots & a_{2n-1}\end{pmatrix}\ne 0. 
$$
This determinant is called the {\it Hankel determinant}.

\begin{remark} An isomorphism $E\cong \cO(n)\oplus \cO(n)$ is given by 
  polynomials $g_0(z)$ and $g_\infty(z^{-1})$ valued in $\GL_2$ such
  that
  \[ \begin{pmatrix} 1 & f(z) \\ 0 & z^{2n} \end{pmatrix} = g_0(z) \begin{pmatrix} z^n & 0 \\ 0 & z^n \end{pmatrix} g_\infty(z^{-1})^{-1}. \]
  This can be rewritten as
  \[ \begin{pmatrix} z^{-n} & z^{-n} f(z) \\ 0 & z^{n} \end{pmatrix} = g_0(z) g_\infty(z^{-1})^{-1}. \]
  Finding such $g_0(z)$ and $g_\infty(z^{-1})$ is a {\it Birkhoff
    factorization} problem. Birkhoff factorizations exist only for
  sufficiently generic operators, and part (ii) of the problem asks to find the
  appropriate condition on $f(z)$ so that this Birkhoff factorization
  exists.
    \end{remark}

\subsection{Problem \ref{Bstr}}

  Recall that $\GL_n$-bundles $E$ are equivalently vector bundles. A
  $B$-structure on the vector bundle $E$ is equivalently a filtration
  by subbundles
  \[ 0 = E_0 \subset E_1 \subset \cdots \subset E_n = E \]
  such that $E_{i+1}/E_i$ are line bundles. We construct such a
  filtration by induction on the rank $n$. By Lemma \ref{l1}, there is a line subbundle $L\subset E$, i.e., a short exact sequence 
    \[ 0 \to L\to E \to E' \to 0. \]
  Since $\rank E' < \rank E$, by the induction hypothesis, $E'$ has a
  filtration by subbundles
  \[ 0 = E_0' \subset E_1' \subset \cdots\subset  E_{n-1}' = E'. \]
  Let $E_{i+1}$ be the preimage of $E_i'$ in $E$ and $E_1=L$. 
  This gives the desired filtration of
  $E$.

\subsection{Problem \ref{adE}}
  Pick a finite cover $X = \bigcup_{i \in I} U_i$ such that the
  principal $G$-bundle $E$ is trivialized on each $U_i$. (By Theorem \ref{BSS}, we may take a  Zariski cover.) Recall that $E$ is
  determined by transition functions, which are regular functions
  \[ g_{ij}\colon U_i \cap U_j \to G\]
  satisfying $g_{ii}=\id,g_{ij} \circ g_{ji} = \id$, and the $1$-cocycle condition
  $g_{ij} \circ g_{jk} \circ g_{ki} = \id$. Therefore, an
  infinitesimal deformation of $E$ is given by the modification
  \[ g_{ij} \mapsto \widetilde g_{ij} \coloneqq g_{ij} \cdot \exp(\varepsilon \xi_{ij}) \]
  where $\varepsilon^2 = 0$, for a choice of a regular function
  $\xi_{ij}\colon U_i \cap U_j \to \lie{g}$ for each $i,j \in I$. The
  conditions that $\{\widetilde g_{ij}\}_{i,j \in I}$ is still a set of
  valid transition functions, namely that $\widetilde g_{ii}=\id,\ \widetilde g_{ij} \circ \widetilde
  g_{ji} = \id$ and $\widetilde g_{ij} \circ \widetilde g_{jk} \circ \widetilde
  g_{ki} = \id$, hold if and only if
  \begin{align*} \xi_{ii}=0,\
    g_{ij} \circ \xi_{ij} \circ g_{ij}^{-1} &= -\xi_{ji}, \\
    g_{ij} \circ \xi_{ij} \circ g_{ij}^{-1} + g_{ik} \circ \xi_{jk} \circ g_{ik}^{-1} + \xi_{ki} &= 0.
  \end{align*}
  These equations say precisely that the element
  \[ (\xi_{ij})_{i,j \in I} \in \bigoplus_{i,j \in I} H^0(U_i \cap U_j, \ad E) \]
  lies in the kernel of the \v Cech differential. Similarly, recall
  that two sets $\{g_{ij}\}_{i,j \in I}$ and $\{g_{ij}'\}_{i,j \in I}$
  of transition functions are equivalent (i.e., define isomorphic bundles) if and only if there exist
  regular functions $\{h_i\colon U_i \to G\}_{i \in I}$ such that
  $g_{ij}' = h_i \circ g_{ij} \circ h_j^{-1}$. A similar reasoning shows that the
  deformed transition functions defined by two different
  $(\xi_{ij})_{i,j \in I}$ and $(\xi_{ij}')_{i,j \in I}$ are
  equivalent if and only if they differ by the image, under the \v
  Cech differential, of an element
  \[ (\eta_i)_{i \in I} \in \bigoplus_{i \in I} H^0(U_i, \ad E). \]
  Putting it all together, we find that $T_E\Bun_G^\circ(X)$ is the cohomology at
  the middle term of the \v Cech complex
  \[ \bigoplus_{i \in I} H^0(U_i, \ad E) \to \bigoplus_{i,j \in I} H^0(U_i \cap U_j, \ad E) \to \bigoplus_{i,j,k \in I} H^0(U_i \cap U_j \cap U_k, \ad E), \]
  which by definition is $H^1(X, \ad E)$.
  
\subsection{Problem \ref{garn}}\label{lastsec}

The Garnier system is the Hitchin system for parabolic
  $\PGL_2$-bundles on $\bP^1$ with marked points $t_1, \ldots, t_N \in
  \bC$, parabolic structures $y_1, \ldots, y_N$, and Higgs field
  \[ \phi = \sum_{i=1}^N \frac{\begin{pmatrix} p_iy_i & -p_i y_i^2 \\ p_i & -p_i y_i \end{pmatrix}}{z - t_i}dz. \] 
Thus the determinant $\det \phi$ 
is a rational function of $z$ with poles only at $t_1,\ldots,t_N$ (we drop the factor $(dz)^2$ for brevity). Moreover, since the residues of $\phi$ are nilpotent, 
these poles are simple. Thus $\det \phi=\frac{a(z)}{b(z)}$ where 
$b(z)\coloneqq(z-t_1)\cdots(z-t_N)$, hence $\deg b=N$. Also since $\phi$ is regular at infinity, the corresponding matrix function vanishes there to second order, implying that $\det \phi$, viewed as a function, generically vanishes to order $4$ at infinity. It follows that 
$\deg a=N-4$. 

The spectral curve is generically the normalization of the hyperelliptic curve $C$
defined by the equation $y^2=\frac{a(z)}{b(z)}$. Replacing $y$ by $y/b(z)$, 
we obtain the equation
$$
y^2=a(z)b(z).
$$
Thus if $\deg a=n_a$ and $\deg b=n_b$ then the genus of $C$ 
is $\frac{1}{2}(n_a+n_b)-1$. So in our case the genus of $C$ 
equals $\frac{1}{2}(2N-4)-1=N-3$, the dimension of the Garnier system, 
as expected (note that in genus $0$, the connected components of the moduli spaces of $\PGL_2$ and $\GL_2$-bundles are the same, so the dimension of the moduli space of $\GL_2$-bundles 
equals the genus of the spectral curve). 

For $N=4$, we get the equation 
$$
y^2=(z-t_1)(z-t_2)(z-t_3)(z-t_4), 
$$
which gives an elliptic curve. If we make a M\"obius transformation 
sending $(t_1,t_2,t_3,t_4)$ to $(0,1,\infty,t)$ then the equation 
of the spectral curve takes the form 
$$
y^2=z(z-1)(z-t).
$$

\begin{remark} We see that in the case of four parabolic points on $\mathbb P^1$, 
the phase space of the Hitchin system is $\cM^\circ\coloneqq T^*(\mathbb A^1\setminus  \lbrace 0,1,t\rbrace)=(\mathbb A^1\setminus  \lbrace 0,1,t\rbrace)\times \mathbb A^1$, 
and the Hamiltonian is 
$$
H=p^2x(x-1)(x-t).
$$
Thus the generic level curves of $H$ are elliptic curves $p^2x(x-1)(x-t)=C$, 
which are missing the four branch points (i.e., points of order $2$). As indicated 
in Subsection \ref{stabu}, this means that we should expect a partial symplectic compactification 
$\cM$ of $\cM^\circ$ in which these points are present, and that the Hitchin map 
$p\colon \cM^\circ\to \mathcal B=\mathbb A^1$ extends to $p\colon \cM\to \mathcal B=\mathbb A^1$, which is now proper, with generic fibers being complete elliptic curves. 

The variety $\cM$ can be constructed as follows (see \cite{Hl}, Subsection 4.2). 
Let $E$ be the elliptic curve which is a double cover of $\mathbb P^1$ branching at points $0,1,t, \infty$. Let $\cM_*\coloneqq (E\times \mathbb A^1)/(\pm 1)$. This is an orbifold with four $A_1$ singularities
corresponding to points of order $2$ on $E$. Let $\cM$ be the blow-up of $\cM_*$ at these singularities. Then $\cM$ contains four exceptional divisors $D_0,D_1,D_t,D_{\infty}$ 
isomorphic to $\mathbb P^1$, and $\cM\setminus \sqcup_{i=1}^4 D_i=\cM^\circ$.   
Also the map $p$ clearly extends to a proper map $p\colon\cM\to \mathbb A^1$ given by the second projection to $\mathbb A^1/(\pm 1)\cong \mathbb A^1$. The space $\cM$ is the simplest example of the {\it Hitchin moduli space}, which is a natural home of the classical Hitchin system. It is obtained from the original phase space $\cM^\circ$ by partial (smooth) compactification -- attaching four disjoint copies of $\mathbb P^1$.  
\end{remark}

\end{document}